\documentclass [a4paper,10pt]{article}

\usepackage[normalem]{ulem}
\usepackage{enumerate}
\def\be{\begin{equation}}
\def\ee{\end{equation}}
\def\bes{\begin{equation*}}
\def\ees{\end{equation*}}
\def\bea{\begin{equation}\begin{aligned}}
\def\eea{\end{aligned}\end{equation}}
\def\beas{\begin{equation*}\begin{aligned}}
\def\eeas{\end{aligned}\end{equation*}}

\usepackage{arcs}
\usepackage{makeidx}
\usepackage{longtable}

\usepackage{caption}
\usepackage{multirow}
\usepackage{epsfig}
\usepackage{multicol}
\usepackage{wrapfig}
\usepackage{pgf,tikz}
\usetikzlibrary{arrows}
\usepackage[T1]{fontenc}
\usepackage[utf8]{inputenc}
\usepackage[english]{babel}
\usepackage{textcomp}
\usepackage{dsfont}
\usepackage{latexsym}
\usepackage{amssymb}
\usepackage{amsthm}
\usepackage{amsmath}
\DeclareMathAlphabet{\mathpzc}{OT1}{pzc}{m}{en}
\usepackage{yfonts}
\usepackage{xfrac}
\usepackage{fge}
\usepackage{makeidx}

\usepackage{fancyhdr}
\usepackage{newlfont}
\usepackage{graphicx}

\usepackage{mathtools}
\usepackage{comment}
\usepackage{indentfirst}
\usepackage{braket}
\usepackage{scalerel}

\usepackage{etoolbox}
\apptocmd{\lim}{\limits}{}{}
\apptocmd{\sup}{\limits}{}{}
\apptocmd{\inf}{\limits}{}{}
\apptocmd{\liminf}{\limits}{}{}
\apptocmd{\limsup}{\limits}{}{}
\pretocmd{\langle}{\left}{}{}
\pretocmd{\rangle}{\right}{}{}

\DeclareMathOperator{\card}{Card}

\DeclareMathOperator{\pr}{pr}

\DeclareMathOperator{\supp}{Supp}

\DeclareMathOperator{\ad}{ad}

\DeclarePairedDelimiter{\abs}{\lvert}{\rvert}
\DeclarePairedDelimiter{\norm}{\lVert}{\rVert}

\let\originalleft\left
\let\originalright\right
\renewcommand{\left}{\mathopen{}\mathclose\bgroup\originalleft}
\renewcommand{\right}{\aftergroup\egroup\originalright}

\newcommand{\Supp}[1]{\supp\left( #1\right) }

\newcommand{\loc}{\mathrm{loc}}
\newcommand{\dom}{\mathrm{dom}}
\newcommand{\vect}[1]{\mathbf{{#1}}}
\newcommand{\dd}{\mathrm{d}}

\newcommand{\N}{\mathds{N}}
\newcommand{\Z}{\mathds{Z}}

\newcommand{\C}{\mathds{C}}
\newcommand{\Hd}{\mathds{H}}
\newcommand{\R}{\mathds{R}}

\newcommand{\gf}{\mathfrak{g}}
\newcommand{\hf}{\mathfrak{h}}
\newcommand{\vf}{\mathfrak{v}}

\newcommand{\Ff}{\mathfrak{F}}

\newcommand{\Cc}{\mathcal{C}}
\newcommand{\Dc}{\mathcal{D}}
\newcommand{\Ec}{\mathcal{E}}
\newcommand{\Fc}{\mathcal{F}}

\newcommand{\Hc}{\mathcal{H}}

\newcommand{\Kc}{\mathcal{K}}
\newcommand{\Lc}{\mathcal{L}}
\newcommand{\Mc}{\mathcal{M}}

\newcommand{\Oc}{\mathcal{O}}
\newcommand{\Pc}{\mathcal{P}}

\newcommand{\Sc}{\mathcal{S}}

\newcommand{\Zc}{\mathcal{Z}}

\newcommand{\meg}{\leqslant}
\newcommand{\Meg}{\geqslant}

\newcommand{\eps}{\varepsilon}
\renewcommand{\phi}{\varphi}

\newcommand{\mi}{\mu}

\newcounter{count}

\DeclareFontFamily{U}{stixbbit}{}
\DeclareFontShape{U}{stixbbit}{m}{it}{<-> stix-mathbbit}{}
\renewcommand{\Re}{\mathrm{Re\,}}
\renewcommand{\Im}{\mathrm{Im\,}}
\newcommand{\ifr}{\mathfrak{i}}

\usepackage[a4paper, top=2.5 cm,
bottom=2.5 cm, %
right=2.5 cm, %
left=2.5 cm]{geometry}

\usepackage{authblk}

\begin{document}
	\title{Functional Calculus on Non-Homogeneous Operators on Nilpotent Groups}
	\author{Mattia Calzi\thanks{Supported by a research grant of the \emph{Scuola Normale Superiore} and member of the \emph{Gruppo Nazionale per l’Analisi Matematica, la Probabilità e le loro Applicazioni} (GNAMPA) of the \emph{Istituto Nazionale di Alta Matematica} (INdAM).}}
	\affil{	Scuola Normale Superiore,\\
		Piazza dei Cavalieri, 7, 56126 Pisa, Italy \\
		E-mail: \texttt{mattia.calzi@sns.it}    }
	\author{Fulvio Ricci}
	\affil{	Scuola Normale Superiore,\\
		Piazza dei Cavalieri, 7, 56126 Pisa, Italy \\
		E-mail: \texttt{fulvio.ricci@sns.it}  
	}
	\date{}
	
	\theoremstyle{definition}
	\newtheorem{deff}{Definition}[section]
	
	\newtheorem{oss}[deff]{Remark}
	
	\newtheorem{ass}[deff]{Assumptions}
	
	\newtheorem{nott}[deff]{Notation}
	
	\newtheorem{ex}[deff]{Example}

	\theoremstyle{plain}
	\newtheorem{teo}[deff]{Theorem}
	
	\newtheorem{lem}[deff]{Lemma}
	
	\newtheorem{prop}[deff]{Proposition}
	
	\newtheorem{cor}[deff]{Corollary}

	\maketitle

\begin{small}
	\section*{Abstract}
	We study the functional calculus associated with a hypoelliptic left-invariant differential operator $\Lc$ on a connected and simply connected nilpotent Lie group $G$ with the aid of the corresponding \emph{Rockland} operator $\Lc_0$ on the `local' contraction $G_0$ of $G$, as well as of the corresponding Rockland operator $\Lc_\infty$ on the `global' contraction $G_\infty$ of $G$.
	We provide asymptotic estimates of the Riesz potentials associated with $\Lc$ at $0$ and at $\infty$, as well as of the kernels associated with functions of $\Lc$ satisfying Mihlin conditions of every order. 
	We also prove some Mihlin--H\"ormander multiplier theorems for $\Lc$ which generalize analogous results to the non-homogeneous case.
	Finally, we extend the asymptotic study of the density of the `Plancherel measure' associated with $\Lc$ from the case of a quasi-homogeneous sub-Laplacian to the case of a quasi-homogeneous sum of even powers.
\end{small}

\section{Introduction}

This paper deals with functional calculus on non-homogeneous left-invariant hypoelliptic self-adjoint differential operators on nilpotent Lie groups.

Functional calculus on self-adjoint Rockland operators (i.e., left-invariant, hypoelliptic and homogeneous) has been widely studied, in particular on sub-Laplacians (cf., for instance,~\cite{FollandStein,MauceriMeda,Christ,Hebisch,MullerStein,MartiniMuller1,MartiniMuller2,MartiniMuller3,MartiniRicciTolomeo}), but also in greater generality (cf., for instance,~\cite{Hulanicki,HebischZienkiewicz,Martini,Martini3,Calzi1,Calzi2}). Also functional calculus on non-homogeneous sub-Laplacians has been considered (cf., for instance,~\cite{Alexopoulos,Sikora,Martini,Martini3}).

The approach introduced in~\cite{NagelRicciStein} indicates that it is possible to transfer information on operators that are functions of a (positive) Rockland operator $\widetilde\Lc$ on a connected and simply connected graded group $G$, or on its convolution kernel, to analogous information relative to the projection of $\widetilde\Lc$ on a general  connected and simply connected, but not necessarily homogeneous,  quotient group. 

Let $G=\widetilde G/I$ be the quotient group, where we assume that $I$ is not dilation invariant to avoid trivialities. The one-parameter family of isomorphic quotient groups $G_s/I_s$, where $I_s$ is  $I$ dilated by $s\in\R_+$, admits two limits $G_0=\widetilde G/I_0$ and $G_\infty=\widetilde G/I_\infty$ (no longer isomorphic to $G$), where $I_0$ and $I_\infty$ are dilation invariant, so that $G_0$ and $G_\infty$ admit induced gradations from $\widetilde G$.

Correspondingly, the operator $\widetilde \Lc$ induces a family $(\Lc_s)_{s\in[0,+\infty]}$, of projected operators on the different quotients. The limit operators $\Lc_0,\Lc_\infty$ are Rockland, while the other $\Lc_s$ lack homogeneity, remaining however hypoelliptic. More precisely, they are {\it weighted subcoercive}, according to the definition introduced in~\cite{ElstRobinson}.\footnote{Functional calculus on weighted subcoercive operators (or systems of operators) has been developed in~\cite{Martini,Martini2,Martini3}. In these works, the homogeneous limit $G_0$ mentioned above is used, at least for comparison with the homogeneous setting by a contraction argument.} 

The starting point in the analysis of~\cite{NagelRicciStein} is a weighted generating family $X_1,\dots, X_n$ of the Lie algebra $\gf$ of $G$. The (Lie algebra of the) group $G$ is then interpreted as the quotient of the free nilpotent Lie algebra $\Fc$ of sufficiently high step with generators $\widetilde X_1,\dots, \widetilde X_n$; the Lie algebra $\Fc$ is then endowed with the (unique) gradation obtained assigning to each $\widetilde X_j$ a degree equal to the weight of $X_j$. Thus, in the above notation, $\Fc$ is the Lie algebra of $\widetilde G$ and the quotient map is uniquely determined by the condition that each $\widetilde X_j$ is mapped onto $X_j$.
A non-commutative homogeneous polynomial $\Pc$ in $n$ indeterminates (endowed with the same weights of $X_1,\dots,X_n$) is then considered under the assumption that the operator $\widetilde \Lc=\Pc(\widetilde X_1,\dots \widetilde X_n)$ is hypoelliptic (hence Rockland). In particular, also the operator $\Lc=\Pc(X_1,\dots, X_n)$ is hypoelliptic; examples of such operators are the sums of even powers of generating vector fields.

It was proved in~\cite{NagelRicciStein} that there is a fundamental solution $K$ of $\Lc$ satisfying the asymptotic relations\footnote{These formulas assume identifications, as manifolds,  of $G$ with $G_0$ and $G_\infty$, respectively. This will be explained in the next section. More precisely, it is proved in~\cite{NagelRicciStein} that $K(x)$ admits two infinite asymptotic expansions at $0$ and $\infty$, with terms which are homogeneous of increasing and decreasing orders, respectively, relative to the dilations of the corresponding limit group.}
\[
K(x)\sim P(x)+K_0(x)\quad \text{ as }x\to0\ ,\qquad K_(x)\sim K_\infty(x)\quad \text{ as }x\to\infty\ ,
\]
where  $K_0$ and $K_\infty$ are  fundamental solutions of $\Lc_0$ on $G_0$ and of $\Lc_\infty$ on $G_\infty$, respectively, while $P$ is a suitable polynomial on $G_0$.

\medskip

The results of the present paper can be divided into four parts. The first part concerns the heat kernels associated to the operator $\Lc$, i.e., the kernels of the operators $e^{-t\Lc}$. In Section~\ref{sec:1} we recall the basic constructions of~\cite{NagelRicciStein} and then we introduce a (somewhat redundant) family of left-invariant vector fields $X_{s,j}$ on each group $G_s$, $s\in [0,\infty]$, which behaves nicely under dilation (which can no longer be defined as automorphisms of the group $G_s$, but rather as isomorphisms between different $G_s$). 
We then introduce two moduli $\abs{\,\cdot\,}_s$ and $\abs{\,\cdot\,}_{s,*}$ on each $G_s$: the former behaves nicely under dilation and equals a homogeneous norm on $G_0$ near the identity $e$ and a homogeneous norm on $G_\infty$ near $\infty$, under suitable identifications; the latter, inspired by~\cite{Martini,Martini3}, is a compromise between the modulus $\abs{\,\cdot\,}_s$ and the Riemannian distance from $e$ associated with the vectors $X_{s,j}$. The importance of $\abs{\,\cdot\,}_{s,*}$ lies in the fact that it grows much faster than $\abs{\,\cdot\,}_s$ at $\infty$, in general, so that it leads to better multiplier theorems. 
In Section~\ref{sec:2} we then make use of the vector fields $X_{s,j}$ and the moduli $\abs{\,\cdot\,}_{s,*}$ to prove uniform `Gaussian' estimates for the kernels $h_{s,t}$ of the $e^{-t \Lc_s}$ (Theorem~\ref{prop:21:2}); we also consider estimates of the derivatives in $s$ of the $ h_{s,t}$, appropriately defined.

In the second part (Section~\ref{sec:3}) we extend the asymptotic estimates in~\cite{NagelRicciStein} to general complex powers of $\Lc$ (Theorem~\ref{teo:21:1}), defined by analytic continuation in the same fashion of the Euclidean case. Even though it would be possible to use the same techniques employed in~\cite{NagelRicciStein}, we shall rely as much as possible on the estimates on $h_{s,t}$ provided in Theorem~\ref{prop:21:2}; in this way, we are able to describe more precisely also the higher order terms of the obtained developments, in some specific situations (Theorem~\ref{teo:3}).

In the third part (Section~\ref{sec:4}) we give asymptotic estimates to kernels of more general multiplier operators (Theorem~\ref{teo:5}) and prove some multiplier theorems of Mihlin--H\"ormander type (Theorems~\ref{teo:1} and~\ref{teo:2}). For what concerns the asymptotic estimates, here we consider more general functions of the operator $\Lc$ -- namely, functions satisfying Mihlin conditions of every order up to the multiplication by a fractional power. Even though these functions include the complex powers of $\Lc$, Theorem~\ref{prop:21:2} is not completely contained in Theorem~\ref{teo:5}, since several terms of the developments obtained in the latter are only defined up to polynomials. 
We then pass to some multiplier theorems, which are generalization of some of the results presented in~\cite{Martini3} to the non-homogeneous case. While Theorem~\ref{teo:1} is stated in full generality and gives non-homogeneous Mihlin--H\"ormander conditions on the multipliers in the fashion of~\cite{Alexopoulos,Sikora}, Theorem~\ref{teo:2} makes use, in a quite more specific situation, of the techniques introduced in~\cite{Hebisch,HebischZienkiewicz} and then systematically developed in~\cite{Martini,Martini3} to lower the regularity threshold up to half the topological dimension of $G$ (instead of half the growth of the volume of $G$ as in Theorem~\ref{teo:1}). Optimality is achieved when $G$ is a product of Métivier and abelian groups, and $\Lc$ is (any) hypoelliptic sub-Laplacian thereon.

The fourth part (Section~\ref{sec:5}) deals with the spectral Plancherel measure $\beta_\Lc$ and its comparison with $\beta_{\Lc_0}$ and~$\beta_{\Lc_\infty}$ (Theorem~\ref{cor:5}), when $\Lc$ is `quasi-homogeneous', following~\cite{Sikora}. Here we both extend the results of~\cite{Sikora} to sums of even powers of generating homogeneous vector fields (instead of quasi-homogeneous sub-Laplacians), and we also observe that the estimates on the density of $\beta_\Lc$ with respect to the Lebesgue measure on $\R_+$ automatically improve to asymptotic expansions at $0$ and at $\infty$.

\section{General Setting}\label{sec:1}

In this section we shall present the general framework in which we shall work in the sequel. It is basically the same as that of~\cite{NagelRicciStein}, except for the fact that we shall not require that the graded group $\widetilde G$ be a free nilpotent Lie group. We shall briefly repeat the basic constructions for the ease of the reader.

\subsection{Contractions}

Let $\widetilde G$ be a graded, connected and simply connected Lie group with Lie algebra $\widetilde \gf$, with gradation $(\widetilde \gf_j)$; let $\pr_j$ be the projection of $\widetilde \gf$ onto $\widetilde \gf_j$ with kernel $\bigoplus_{j'\neq j}\widetilde \gf_{j'}$, and define $n\coloneqq \max \Set{j>0\colon \gf_j\neq 0}$. 

On $\widetilde G$ we introduce the dilations  $x\mapsto r\cdot x$, $r\in\R_+=(0,\infty)$, adapted to the given gradation, i.e., such that $r\cdot x=r^j x$ if $x\in\widetilde\gf_j$. We shall sometimes denote by $\rho_r$ the dilation by $r$.
A linear subspace $\vf$ of $\widetilde\gf$ is graded, i.e., $\vf=\bigoplus_j\vf\cap\widetilde\gf_j$, if and only if it is homogeneous, i.e., invariant under dilations.
We say that a linear map from a graded subspace $\vf$ of $\widetilde \gf$ to $\widetilde\gf$ is 
\begin{itemize}
\item {\it homogeneous} if it preserves the gradation, i.e., it maps $\vf\cap\widetilde \gf_j$ into $\widetilde \gf_j$ for every $j$;
\item {\it strictly sub-homogeneous} if it maps $\vf\cap\widetilde \gf_j$ into $\bigoplus_{j'<j}\widetilde \gf_{j'}$ for every $j$;
\item {\it strictly super-homogeneous} if it maps $\vf\cap\widetilde \gf_j$ into $\bigoplus_{j'>j}\widetilde \gf_{j'}$ for every $j$.
\end{itemize}

Now, let $G$ be the quotient of $\widetilde G$ by a (not necessarily homogeneous) normal subgroup, and denote by $\pi$ the corresponding projection; we shall assume that $G$ is simply connected.
Let $\ifr$ be the kernel of $ \dd \pi$, and observe that $\ker \pi= \exp_{\widetilde G} \ifr$ since $G$ is simply connected. 
Then, define
\[
\ifr_0\coloneqq \bigoplus_{j=1}^n \pr_j\bigg(\ifr \cap \bigg(\bigoplus_{j'\meg j} \widetilde\gf_{j'}\bigg)   \bigg), \qquad \text{and} \qquad \ifr_\infty\coloneqq \bigoplus_{j=1}^n \pr_j\bigg(\ifr \cap \bigg(\bigoplus_{j'\Meg j} \widetilde\gf_{j'}\bigg)   \bigg).
\]
For $s\in (0,\infty)$, we define  $\ifr_s\coloneqq s^{-1}\cdot \ifr$.

The following result is basically a generalization of~\cite[Proposition 2, Lemma, and Corollary of § 2]{NagelRicciStein}. 

\begin{prop}\label{prop-psi}
	The vector spaces $\ifr_0$ and $\ifr_\infty$ are graded ideals of $\widetilde \gf$ {and have the same dimension as $\ifr$}. In addition, there are two linear mappings $\psi_{0,1}\colon\ifr_0\to\widetilde\gf$ and $\psi_{\infty,1}\colon\ifr_\infty\to\widetilde\gf$ such that
	\begin{itemize}
	\item $\psi_{0,1}$ is strictly sub-homogeneous and $I+\psi_{0,1}$ is a bijection of $\ifr_0$ onto $\ifr$; 
	
	\item $\psi_{\infty,1}$ is strictly super-homogeneous and $I+\psi_{\infty,1}$ is a bijection of $\ifr_\infty$ onto $\ifr$;
	
	\item defining, for $s\in \R_+$, $\psi_{0,s}$ and $\psi_{\infty,s}$ as
	$\psi_{0,s}=s^{-1}\cdot \psi_{0,1}(s\,\cdot\,)$ and $\psi_{\infty,s}=s^{-1}\cdot \psi_{\infty,1}(s\,\cdot\,)$, respectively, these maps are strictly sub- (resp.\ super-)homogeneous and 
	\[
	\lim_{s\to0^+}\psi_{0,s}=0, \qquad \lim_{s\to\infty}\psi_{\infty,s}=0\, ;
	\]	
	
	\item if $\hf_0$ and $\hf_\infty$ are graded complements of $\ifr_0$ and $\ifr_\infty$ in $\widetilde \gf$, respectively, then they are also algebraic complements of $\ifr_s$ for every $s\in \R_+$.
	\end{itemize}
\end{prop}

\begin{proof}
	{\bf1.} It is clear that $\ifr_0$ is a graded subspace of $\widetilde \gf$; let $(\ifr_{0,j}=\ifr_0\cap\widetilde\gf_j)$ be its gradation. Take $x\in \widetilde \gf_{j_1}$ for some $j_1$ and $y\in \ifr_{0,j_2}$ for some $j_2$; let us prove that $[x,y]\in \ifr_{0,j_1+j_2}$. 
	Now, there is $y'\in \ifr$ such that $y-y'\in \bigoplus_{j'<j_2}\widetilde \gf_{j'}$, so that $[x,y]\in [x,y']+\left( \bigoplus_{j'<j_2}\widetilde \gf_{j_1+j'}\right) $, whence $[x,y]= \pr_{j_1+j_2}([x,y'])\in\ifr_{0,j_1+j_2}  $. 
	By the arbitrariness of $x$ and $y$, it follows that $\ifr_0$ is a graded ideal. 
	In the same way one proves that $\ifr_\infty$ is a graded ideal.
	
	{\bf2.} Now, let us define  $\psi_{0,1}$. Observe that, by induction, we may define a basis $(e_{k})$ of $\ifr$ and an increasing sequence $(k_j)$ such that  $(e_{k})_{k\meg k_j}$ is a basis of $\ifr \cap \left(\bigoplus_{j'\meg j}\widetilde \gf_{j'}\right)$ for every $j$.
	Let us prove that, for every $j$, $(\pr_j(e_k))_{k_{j-1}<k\meg k_j}$ is a basis of $\ifr_{0,j}$. Clearly, it will suffice to prove linear independence. 
	Now, if $(\lambda_k)_{k_{j-1}<k\meg k_j}$ is a family of real numbers such that $\sum_k \lambda_k \pr_j(e_k)=0$, then $\sum_{k_{j-1}<k\meg k_j} \lambda_k e_k\in \ifr \cap \left(\bigoplus_{j'< j}\widetilde \gf_{j'}\right)$. Hence, there is a family $(\lambda_k)_{k\meg k_{j-1}}$ of real numbers such that
	\[
	\sum_{k_{j-1}<k\meg k_j} \lambda_k e_k=\sum_{k\meg k_{j-1}} \lambda_k e_k,
	\]
	whence $\lambda_k=0$ for every $k=1,\ldots, k_j$.
	Then, we may simply define $\psi_{0,1}$ as the linear map such that 
	\[
	\psi_{0,1}\big(\pr_j( e_k)\big)= e_k-\pr_j(e_k)= \sum_{j'<j} \pr_{j'}(e_k) .
	\]
	for every $j=1,\ldots, n$ and for every $k=k_{j-1}+1,\ldots, k_j$.
	 Then  $\psi_{0,1}$ is strictly sub-homogeneous and 
	\[
	(I+\psi_{0,1})\big(\pr_j( e_k)\big) =e_k\,,
	\]
	showing that $I+\psi_{0,1}$ maps  $\ifr_0$ onto $\ifr$ bijectively. It is also clear that 
	\be\label{psi_s}
	\psi_{0,s}(\pr_j( e_k))=\sum_{j'<j} s^{j-j'} \pr_{j'}(e_k) ,
	\ee
which tends to 0 as $s\to0^+$.

In a similar way one constructs $\psi_{\infty,1}$ and proves the corresponding properties.  In particular, we see that $\ifr,\ifr_0$, and $\ifr_\infty$ have the same dimension.
	
	{\bf3.} Let $\hf_0$ be a graded complement of $\ifr_0$ in $\widetilde \gf$. Since the mapping $s\mapsto \ifr_s$ is continuous on $[0,\infty]$ (with values in the Grassmannian of $(\dim \ifr)$-dimensional subspaces of $\widetilde \gf$), it follows that $\hf_0$ is an algebraic complement of $\ifr_r$ for some $r>0$. Therefore, $\hf_0=(s^{-1}r)\cdot \hf_0$ is an algebraic complement of $\ifr_s=(s^{-1}r)\cdot \ifr_r$ for every $s\in (0,\infty)$. 
	
	The assertions concerning $\hf_\infty$ are proved in a similar way.
\end{proof}

Observe that,  by \eqref{psi_s} and its analogue for $\psi_{\infty,s}$, the linear mappings $\psi_{0,s}$ and $\psi_{\infty,\sfrac{1}{s}}$ depend polynomially on $s$.

For $s\in [0,\infty]$, consider the quotient Lie algebras $\gf_s=\widetilde\gf/\ifr_s$. Dilation of $\widetilde\gf$ by $r>0$ induces an isomorphism between $\gf_s$ and $\gf_{r^{-1}s}$; in particular, $\gf_s$ is isomorphic to $\gf_1$ for every $s\in (0,\infty)$, while  $\gf_0$ and $\gf_\infty$ need not be isomorphic with any  other $\gf_s$.  We call $\gf_0$ and $\gf_\infty$ the local and the global contractions of $\gf_1$, respectively.

We fix once and for all two \emph{graded} algebraic complements $\hf_0$ and $\hf_\infty$ of $\ifr_0$ and $\ifr_\infty$, respectively. By Proposition \ref{prop-psi}, both $\hf_0$ and $\hf_\infty$ are complementary to $\ifr_s$ for all $s\in\R_+$.

\begin{deff}\label{def:2}
For $s\in[0,\infty)$, let $P_{0,s}$ be the projection of $\widetilde \gf$ onto $\hf_0$ with kernel $\ifr_s$ and, for $s\in(0,\infty]$, let $P_{\infty,s}$ be the projection of $\widetilde \gf$ onto $\hf_\infty$ with kernel $\ifr_s$. 
\end{deff}

\begin{lem}\label{lem:21:2}
	 $P_{0,0}$ is homogeneous and, for every $s\in \R_+$,  $P_{0,s}-P_{0,0}$ is strictly sub-homogeneous; in addition, $P_{0,r s}= r^{-1}\cdot P_{0,s}(r\,\cdot\,)$ for every  $r,s\in\R_+$.
	
	Analogously, $P_{\infty,\infty}$ is homogeneous and, for every $s\in \R_+$,   $P_{\infty,s}-P_{\infty,\infty}$ is strictly super-homogeneous;  in addition, $P_{\infty,r s}= r^{-1}\cdot P_{\infty,s}(r\,\cdot\,)$ for every $r,s\in\R_+$.
\end{lem}

\begin{proof}
	The homogeneity of $P_{0,0}$ and $P_{\infty,\infty}$ is obvious, as well as the scaling properties of the projections. Thus, we may reduce ourselves to proving sub- (resp.\ super-)homogeneity of $P_{0,1}$ and $P_{\infty,1}$. 
	Then, take $x\in \widetilde \gf_k$ for some $k$, and let us prove that $\pr_h(P_{0,1}(x)-P_{0,0}(x))=0$ for $h\Meg k$. 
	Indeed, assume that $\pr_h(P_{0,1}(x))\neq0$ for some $h\Meg k$, and let $h'$ be the maximum of such $h$. 
	Then, $\pr_{h}(x-P_{0,1}(x))=\pr_h(P_{0,1}(x))= 0$ for every $h>h'$, so that $\pr_{h'}(x-P_{0,1}(x))\in \ifr_0$ since $x-P_{0,1}(x)\in \ifr$ by the definition of $P_{0,1}$. 
	Since $\pr_{h'}(P_{0,1}(x))\in \hf_0$ and $\hf_0\cap \ifr_0=\Set{0}$, we then deduce that $h'=k$ and that $x-\pr_k(P_{0,1}(x))\in \ifr_0$, so that $\pr_k(P_{0,1}(x))=P_{0,0}(x)$. 
	
	One proves analogously that $P_{\infty,1}-P_{\infty,\infty}$ is strictly super-homogeneous.
\end{proof}

Each map $P_{0,s}$ (resp.\ $P_{\infty,s}$) induces a Lie algebra structure on $\hf_0$ (resp.\ $\hf_\infty$); we denote by $[\,\cdot\, ,\,\cdot\, ]_{0,s}$ (resp.\ $[\,\cdot\, ,\,\cdot\,]_{\infty,s}$) the corresponding Lie bracket. In other words,
\be\label{[]s}
[x,y]_{0,s}=P_{0,s}[x,y]\ ,\quad \forall\,x,y\in\hf_0\ ,\qquad [x,y]_{\infty,s}=P_{\infty,s}[x,y]\ ,\quad \forall\,x,y\in\hf_\infty\ .
\ee

Notice that, for $r,s\in\R_+$, 
\be\label{scaling[]}
r\cdot [x,y]_{0,s}=[r\cdot x,r\cdot y]_{0,r^{-1}s}\ ,\quad x,y\in\hf_0\ ,\qquad r\cdot [x,y]_{\infty,s}=[r\cdot x,r\cdot y]_{\infty,r^{-1}s}\ ,\quad x,y\in\hf_\infty\ .
\ee

We use the Baker--Campbell--Hausdorff products
induced by the Lie brackets in \eqref{[]s} to realize either $\hf_0$, if $s\in[0,\infty)$, or $\hf_\infty$, if $s\in(0,\infty]$, as the underlying manifold\footnote{In principle, we shall privilege the realization of $G_s$ on $\hf_0$ for $s$ close to 0 and that on  $\hf_\infty$ for $s$ close to $\infty$. We prefer anyhow to keep the double realization for every $s\in\R_+$ in order to avoid apparent discontinuities in $s$ at some finite point on one hand, and a-priori quantifications of ``closeness'' to 0 or $\infty$ on the other.}  of the group $G_s\coloneqq \widetilde G/\exp_{\widetilde G}\ifr_s$. We call $G_0$ and $G_\infty$ the local and the global contractions of $G_1$, respectively.

Notice that
\bea\label{lim[]}
[x,y]_{0,s}&=[x,y]_{0,0}+O(s)\ ,\quad s\to0\ ,\ \text{uniformly for  $x,y$ in a compact subset of $\hf_0$}\ ,\\
[x,y]_{\infty,s}&=[x,y]_{\infty,\infty}+O(1/s)\ ,\quad s\to\infty\ ,\ \text{uniformly for  $x,y$ in a compact subset of $\hf_\infty$}\ ,
\eea
and the analogous formulae for products in $G_s$.

We denote by $\pi_s$, $s\in[0,\infty]$, the canonical projection of $\widetilde G$ onto $G_s$ and by $\dd\pi_s\colon\widetilde\gf\to \gf_s$ its differential. By an abuse of language, we shall keep the same notation whenever  $G_s$, or $\gf_s$, is identified with either $\hf_0$ or $\hf_\infty$.

Since the ideals $\ifr_0$ and $\ifr_\infty$ are graded, the corresponding quotients $\gf_0$ and $\gf_\infty$ inherit a gradation and the corresponding dilations. These dilations coincide with the restriction to $\hf_0$ and $\hf_\infty$, respectively, of the dilations of~$\widetilde\gf$. 

For $s,r\in\R_+$, it follows from \eqref{scaling[]} that dilation by $r$ on either $\hf_0$ or $\hf_\infty$ induces an isomorphism of $G_s$ onto $G_{r^{-1}s}$ for every $s$. Notice that the so-induced mappings $G_s\to G_{r^{-1}s}$ do not depend on the chosen identifications (cf.~\emph{(ii)} and \emph{(iv)} of Proposition~\ref{prop:21:9} below).

We denote by $\widetilde Q$, $Q_0$, and $Q_\infty$ the homogeneous dimensions of $\widetilde G$, $G_0$, and $G_\infty$, respectively.

The following result generalizes~\cite[Proposition 3 and p.~264]{NagelRicciStein}.

\begin{prop}\label{prop:21:9}
	For every $s\in \R_+$, define $\lambda_s\coloneqq(P_{\infty,s})_{|_{\hf_0}}  \colon \hf_0\to \hf_\infty$; let $N$ be a homogeneous norm on $\widetilde \gf$. Then,
	\begin{enumerate}[\rm(i)]
	\item $\lambda_s$ is the unique linear mapping such that $x-\lambda_s(x)\in \ifr_s$ for every $x\in \hf_0$; in addition, $\lambda_s$ is invertible and its inverse $\lambda_s^{-1}=(P_{0,s})_{|_{\hf_\infty}} $ is the unique linear mapping such that $x-\lambda_s^{-1}(x)\in \ifr_s$ for every $x\in \hf_\infty$;
	
	\item$\lambda_s$ intertwines the two identifications of $\gf_s$ with $\hf_0$, resp.\ $\hf_\infty$, i.e., $\lambda_s[x,y]_{0,s}=[\lambda_sx,\lambda_sy]_{\infty,s}$ for all $x,y\in\hf_0$; 
	
		\item $\lambda_s-I$ is strictly super-homogeneous and $\lambda_s^{-1}-I$ is strictly sub-homogeneous;
		
		\item $\lambda_{r s}=r^{-1} \cdot \lambda_s(r\,\cdot\,)$ for every $r>0$.
		
		\item $Q_\infty\Meg Q_0$;
		
		\item $N(\lambda_s(x))=O(N(x))$ for $x\to \infty$ in $\hf_0$  and $N(\lambda_s^{-1}(x))=O(N(x))$ for $x\to 0$ in $\hf_\infty$.
	\end{enumerate}
\end{prop}

\begin{proof}
	By the definition of $P_{0,s},P_{\infty,s}$, the two cosets $x+\ifr_s$ and $\lambda_s(x)+\ifr_s$ coincide for $x\in\hf_0$. This gives (i) and (ii); (iii) and (iv) follow directly from Lemma~\ref{lem:21:2}.

	To prove (v), we argue as in the proof of~\cite[Proposition 3]{NagelRicciStein}.
	Define $\hf_{0,j}\coloneqq\widetilde\gf_j\cap\hf_0$ and $\hf_{\infty,j}\coloneqq\widetilde\gf_j\cap\hf_\infty$, and set $s=1$.
	Since $\lambda_1$ is super-homogeneous we see that, for every $k=1,\ldots, n$,
	\[
	\bigoplus_{j<k}\lambda_1(\hf_{0,j} )+\bigoplus_{j\Meg k} \hf_{\infty,j}=\hf_\infty,
	\]
	so that
	\[
	\sum_{j<k} \dim(\hf_{0,j})+\sum_{j\Meg k} \dim(\hf_{\infty,j})\Meg \dim(\widetilde \gf).
	\]
	Summing up all these inequalities, we see that
	\[
	n \dim(\widetilde \gf)-Q_0+Q_\infty\Meg n \dim(\widetilde \gf),
	\]
	whence $Q_\infty\Meg Q_0$.

	For what concerns  (vi),  fix a norm $\norm{\,\cdot\,}$ on $\widetilde \gf$ and observe that there is a constant $C\Meg 1$ such that
	\[
	\frac{1}{C} \max_{k} \norm{\pr_k(x)}^{\sfrac{1}{k}}\meg N(x)\meg C \max_{k} \norm{\pr_k(x)}^{\sfrac{1}{k}}
	\]
	for every $x\in \widetilde \gf$.
	Further, by the quasi-sub-additivity of $N$, there is a constant $C'>0$ such that, for $x\in\hf_0$,
	\[
	N(\lambda_s(x))\meg C' \sum_k N(\lambda_s(\pr_k(x)))\meg C C'\sum_k \max\big\{\norm{\lambda_s(\pr_k(x))}^{\sfrac{1}{k}}, \norm{\lambda_s(\pr_k(x))}^{\sfrac{1}{n}}  \big\}.
	\]
	Therefore, there is a constant $C''>0$ such that
	\[
	N(\lambda_s(x))\meg C'' \sum_k \max\big\{ \norm{\pr_k(x)}^{\sfrac{1}{k}}, \norm{\pr_k(x)}^{\sfrac{1}{n}}\big\}\meg n C'' +C''\sum_k   \norm{\pr_k(x)}^{\sfrac{1}{k}},
	\]
	so that $N(\lambda_s(x))=O(N(x))$ for $x\to \infty$ in $\hf_0$. The  second part is proved similarly. 
\end{proof}

\subsection{Invariant Vector Fields}

We now pass to the approximation of differential operators, following~\cite[§ 4]{NagelRicciStein}.

\begin{deff}\label{def:3}
	Let $V$ be a homogeneous vector space, with dilations $\rho_r$, $r\in\R_+$.
	If $T$ is a distribution on $V$, we define $\rho_r^* T$ and $T\circ \rho_r$ by
	\[
	\langle \rho_r^*T, \phi\rangle\coloneqq \langle T, \phi\circ \rho_r^{-1}\rangle \qquad \text{and} \qquad \langle T\circ \rho_r,\phi\rangle \coloneqq\langle T,  r^{-Q}\phi\circ \rho_r^{-1}\rangle
	\]
	for every $\phi \in C^\infty_c(V)$, where $Q$ is the homogeneous dimension of $V$. We also define $(\rho_r)_* T\coloneqq (\rho_r^{-1})^*T$.

	We say that a function (or a distribution) $f$ is {\it homogeneous of degree $d\in\C$} if $f\circ\rho_r=r^d f$ for all $r\in\R_+$. 
	We say that $f$ is \emph{log-homogeneous of degree $d\in \N$} if there are a homogeneous polynomial $P$ of degree $d$ on $V$ and a homogeneous norm $N$ such that $f-P \log N$ is homogeneous of degree $d$.\footnote{If $N'$ is another homogeneous norm, then $f-P\log N'= f-P \log N +P\log (N/N')$ is still homogeneous of degree $d$.} 
	
	We say that a continuous linear operator $X\colon C^\infty(V)\to C^\infty(V)$ (for example a (linear) differential operator) is {\it homogeneous of order $d$} if $X(f\circ\rho_r)=r^d(X f)\circ\rho_r$ for all $r>0$.
\end{deff}

Notice that, if $X$ is a left-invariant differential operator under a homogeneous Lie group structure on $V$, then $X f=f*(X\delta_0)$ and $X$ is homogeneous of order $d$ if and only if the distribution $X\delta_0$ is homogeneous of degree~$-Q-d$.

In addition, if $f$ is a function of class $C^\infty$ on $V$ and $M_f$ is the operator of multiplication by $f$, then $f$ is homogeneous of degree $d$ if and only if $M_f$ is homogeneous of order $-d$.

As a consequence, if $X$ is a homogeneous differential operator of order $d$ and $f$ is a homogeneous function of degree $d'$ and of class $C^\infty$, then $f X$ is a homogeneous differential operator of order $d-d'$.

Finally, observe that, if an element $X$ of the enveloping algebra of $\widetilde G$ is homogeneous of degree $d$, then the corresponding left- (or right-)invariant differential operator is homogeneous of order $d$, and conversely.
Similar statements hold for $\widetilde G$, $G_0$, and $G_\infty$.

\medskip

Now, observe that $\hf_0$ and $\hf_\infty$ are graded subspaces of $\widetilde \gf$, so that also $\hf_0\cap \hf_\infty$ is a graded subspace of $\widetilde \gf$. Hence, we may complete a homogeneous basis of $\hf_0\cap \hf_\infty$ to homogeneous bases of $\hf_0$ and $\hf_\infty$, and then complete the union of the two (which is a homogeneous basis of $\hf_0+\hf_\infty$) to a homogeneous basis of $\widetilde \gf$. Consequently, we may state the following definition.

\begin{deff}
	We denote by $(\widetilde X_j)_{j\in J}$ a \emph{homogeneous} basis of  $\widetilde \gf$ such that there are two subsets $J_0$ and $J_\infty$ of $J$ such that $(\widetilde X_j)_{j\in J_0}$ is a basis of $\hf_0$, while $(\widetilde X_j)_{j\in J_\infty}$ is a basis of $\hf_\infty$. We denote by $\dd_j$ the degree of $\widetilde X_j$ (as an element of the graded Lie algebra $\widetilde \gf$, so that $\widetilde X_j$ is homogeneous of order $\dd_j$ as a differential operator). 
	Fixing coordinates on $\widetilde \gf$ associated with the basis $(\widetilde X_j)_{j\in J}$, we denote by $(\partial_j)_{j\in J}$ the corresponding partial derivatives.

	Define $X_{s,j}\coloneqq \dd\pi_s(\widetilde X_j)$ for every $j\in J$ and for every $s\in [0,\infty]$, so that 
	\[
	(r\,\cdot\,)_* X_{s,j}= r^{\dd_j} X_{r^{-1}s,j}
	\]
	for every $s\in [0,\infty]$, for every $r>0$, and for every $j\in J$.
	
	Finally,   fix  a total ordering on $J$ and define, for every $\gamma\in \N^{J}$, 
	\[
	\widetilde{\vect{ X}}^\gamma=\prod_{j\in J} \widetilde X_j^{\gamma_j}, 
	\]
	so that $\widetilde{\vect{ X}}^\gamma$ is homogeneous of order $\dd_\gamma\coloneqq \sum_{j\in J} \gamma_j\dd_j$. 
	Define $\partial^\gamma$ and  $\vect{X}_s^\gamma$, for every $s\in [0,\infty]$, in a similar way. To simplify the notation, we shall identify $\N^{J_0}$ and $\N^{J_\infty}$ with subsets of $\N^J$; when $\gamma\in \N^{J_0}$ (resp.\ $\gamma\in \N^{J_\infty}$), we shall also write $\partial_0^\gamma$ (resp.\ $\partial_\infty^\gamma$) instead of $\partial^\gamma$.
\end{deff}

The following result is a simple generalization of~\cite[Propositions 4 and 5]{NagelRicciStein}. Observe that, even though in~\cite[Propositions 4 and 5]{NagelRicciStein} the polynomials  $p_{0,\gamma,\gamma'}$ and $p_{\infty,\gamma,\gamma'}$ were constructed comparing the products on $G_0$, $G_\infty$, and $G_1$, if one tries to define the matrix $(\delta_{\gamma,\gamma'}+p'_{0,\gamma,\gamma'})$ as the inverse of $(\delta_{\gamma,\gamma'}+p_{0,\gamma,\gamma'})$, then one would only prove that the $p'_{0,\gamma,\gamma'}$ are (everywhere defined) rational functions. 
Consequently, we shall present a different proof.

\begin{prop}\label{lem:21:1}
	For every $\gamma\in \N^J$ there are two \emph{unique} finite families $(p_{0,\gamma,{\gamma'}})_{\gamma'\in \N^{J_0}}$  and  $(p'_{0,\gamma,{\gamma'}})_{\gamma'\in \N^{J_0}}$ of polynomials on $\hf_0$ such that, identifying $G_s$ and $G_0$ with $\hf_0$ for every $s\in [0,\infty)$,
	\[
	 \vect{X}_{s}^\gamma=\vect{X}_{0}^\gamma+ \sum_{\gamma'} s^{\dd_{\gamma}-\dd_{\gamma'}}p_{0,\gamma,{\gamma'}}(s\,\cdot\,) \vect{X}_{0}^{\gamma'}\ , \qquad \vect{X}^\gamma_{0}=\vect{X}_{s}^\gamma+ \sum_{{\gamma'}} s^{\dd_{\gamma}-\dd_{\gamma'}} p'_{0,\gamma,{\gamma'}}(s\,\cdot\,)\vect{X}_{s}^{\gamma'}.
	\]
	In addition, $p_{0,\gamma,{\gamma'}}$ and $p'_{0,\gamma,{\gamma'}}$ are sums of homogeneous  polynomials of degrees strictly greater than $\dd_{\gamma'}-\dd_{\gamma}$.

	Analogously, there are two \emph{unique} finite families  $(p_{\infty,\gamma,{\gamma'}})_{\gamma'\in \N^{J_\infty},\dd_{\gamma'}>\dd_{\gamma}}$ and $(p'_{\infty,\gamma,{\gamma'}})_{\gamma'\in \N^{J_\infty},\dd_{\gamma'}>\dd_{\gamma}}$ of polynomials on $\hf_\infty$ such that, identifying $G_s$ and $G_\infty$ with $\hf_\infty$ for every $s\in  (0,\infty]$,
	\[
	 \vect{X}_{s}^\gamma=\vect{X}_{\infty}^\gamma+ \sum_{\gamma'} s^{\dd_{\gamma}-\dd_{\gamma'}}p_{\infty,\gamma,\gamma'}(s\,\cdot\,) \vect{X}_{\infty}^{\gamma'}\ , \qquad   \vect{X}^\gamma_\infty=\vect{X}_{s}^\gamma+ \sum_{\gamma'} s^{\dd_{\gamma}-\dd_{\gamma'}} p'_{\infty,\gamma,{\gamma'}}(s\,\cdot\,)\vect{X}_{s}^{\gamma'}.
	\]
	In addition, $p_{\infty,\gamma,{\gamma'}}$ and $p'_{\infty,\gamma,{\gamma'}}$  are sums of homogeneous polynomials of degrees strictly smaller than $\dd_{\gamma'}-\dd_{\gamma}$.\footnote{By the general theory, it is also clear that $\sum_j  \gamma'_j\meg \sum_j \gamma_j$ if $p_{0,\gamma,{\gamma'}}\neq 0$, $p'_{0,\gamma,{\gamma'}}\neq0$, $p_{\infty,\gamma,{\gamma'}}\neq 0$, or $p'_{\infty,\gamma,{\gamma'}}\neq0$.} 

	In particular,
	\[
	\lim_{s\to0}\vect{X}_{s}^\gamma=\vect{X}_{0}^\gamma\ ,\qquad \lim_{s\to\infty}\vect{X}_{s}^\gamma=\vect{X}_{\infty}^\gamma\ .
	\]
\end{prop} 

Notice that, when $\gamma\in \N^{J_0}$, it may happen that $p_{0,\gamma,\gamma}\neq 0$ and $p'_{0,\gamma,\gamma}\neq 0$.
Nonetheless, it is always true that both $p_{0,\gamma,\gamma}$ and $p_{0,\gamma,\gamma}'$ vanish at $0$.

For example, consider the case in which $\widetilde G$ is the free $2$-step nilpotent Lie group on three generators $\widetilde X_1$, $\widetilde X_2$ and $\widetilde X_3$ (and the standard dilations), and define $\ifr$ as the vector space generated by $[\widetilde X_1, \widetilde X_2]-\widetilde X_1-\widetilde X_3$, $[\widetilde X_1, \widetilde X_3]$, and $[\widetilde X_2, \widetilde X_3]-\widetilde X_1-\widetilde X_3$. 
Then, $\ifr_0=[\widetilde \gf,\widetilde \gf]$ and $G_1$ is isomorphic to the three-dimensional Heisenberg group, while $G_0$ is isomorphic to $\R^3$. 
Fix coordinates $(x_1,x_2,x_3)$ on $G_1$ corresponding to the basis $(X_{1,1},X_{1,2},X_{1,3})$, so that $X_{0,j}=\partial_{x_j}$ under the identification of $G_0$ and $G_1$ with $\hf_0$.
Then, simple computations show that
\begin{align*}
X_{1,1}&= X_{0,1}-\frac{x_2}{2}X_{0,1}-\frac{x_2}{2}X_{0,3}\\
X_{1,2}&=X_{0,2} + \frac{x_1-x_3}{2}X_{0,1}+\frac{x_1-x_3}{2}X_{0,3}\\
X_{1,3}&= X_{0,3}+\frac{x_2}{2}X_{0,1}+\frac{x_2}{2}X_{0,3},
\end{align*}
while
\begin{align*}
	X_{0,1}&= X_{1,1}+\frac{x_2}{2}X_{1,1}+\frac{x_2}{2}X_{1,3}\\
	X_{0,2}&=X_{1,2} -\frac{x_1-x_3}{2}X_{1,1}-\frac{x_1-x_3}{2}X_{1,3}\\
	X_{0,3}&= X_{1,3}-\frac{x_2}{2}X_{1,1}-\frac{x_2}{2}X_{1,3},
\end{align*}
whence our assertion.

\begin{proof}
	Observe first that~\cite[Theorem 1.1.2]{Varadarajan}   shows that there are two (unique) finite families $(p_{0,\gamma,\gamma'})$ and $(p'_{0,\gamma,\gamma'})$ of $C^\infty$ functions on $\hf_0$ such that
	\[
	 \vect{X}_{1}^\gamma-\vect{X}_{0}^\gamma= \sum_{\gamma'} p_{0,\gamma,{\gamma'}} \vect{X}_{0}^{\gamma'}=- \sum_{{\gamma'}}  p'_{0,\gamma,{\gamma'}}\vect{X}_{1}^{\gamma'}.
	\] 
	Applying the dilation by $s^{-1}$, we then get
	\[
	 \vect{X}_{s}^\gamma-\vect{X}_{0}^\gamma= \sum_{\gamma'} s^{\dd_{\gamma}-\dd_{\gamma'}}p_{0,\gamma,{\gamma'}}(s\,\cdot\,) \vect{X}_{0}^{\gamma'}=- \sum_{{\gamma'}} s^{\dd_{\gamma}-\dd_{\gamma'}} p'_{0,\gamma,{\gamma'}}(s\,\cdot\,)\vect{X}_{s}^{\gamma'}
	\]
	for every $s\in (0,\infty)$. 
	Now, let us prove that the $p_{0,\gamma,\gamma'}$ and the $p_{0,\gamma,\gamma'}'$ are polynomials.

	Observe that $((\vect{X}_s^{\gamma'})_0)_{\gamma'\in \N^{J_0},\sum_j \gamma'_j\meg k}$ is a basis of the space of distributions on $\hf_0$ supported at $0$ and of order at most $k$, for every $k\in \N$ and for every $s\in [0,\infty)$. 
	Therefore, there are two families $(S_{0,\gamma'})$ and $(S'_{0,\gamma'})$ of polynomials on $\hf_0$ such that $(\vect{X}_0^{\gamma''} S_{0,\gamma'})(0)=\delta_{\gamma',\gamma''}$ and $(\vect{X}_1^{\gamma''} S'_{0,\gamma'})(0)=\delta_{\gamma',\gamma''}$ for every $\gamma''\in \N^{J_0}$. 
	Then, define $S_{x,\gamma'}(y)\coloneqq S_{0,\gamma'}(x^{-1}\cdot_{G_0} y)$ and $S'_{x,\gamma'}(y)\coloneqq S'_{0,\gamma'}(x^{-1}\cdot_{G_1} y)$  for every $x,y\in \hf_0$, so that 
	\[
	(\vect{X}^{\gamma''}_0 S_{x,\gamma'})(x)=(\vect{X}^{\gamma''}_1 S'_{x,\gamma'})(x)=\delta_{\gamma',\gamma''}
	\]
	for every $\gamma',\gamma''\in \N^{J_0}$. 
	Therefore,
	\begin{align*}
	p_{0,\gamma,\gamma'}(x)&=(\vect{X}_1^\gamma-\vect{X}^\gamma_0)(S_{x,\gamma'})(x)=(\vect{X}_1^\gamma S_{x,\gamma'})(x)-\delta_{\gamma,\gamma'} \\
	p'_{0,\gamma,\gamma'}(x)&=(\vect{X}_0^\gamma-\vect{X}^\gamma_1)(S'_{x,\gamma'})(x)=(\vect{X}_0^\gamma S'_{x,\gamma'})(x)-\delta_{\gamma,\gamma'}
	\end{align*}
	for every $\gamma'\in \N^{J_0}$ and for every $x\in \hf_0$.
	Now, it is clear that he mappings $\hf_\infty \ni(x,y)\mapsto S_{x,\gamma'}(x\cdot_{G_1} y)= S_{0,\gamma}(x^{-1}\cdot_{G_0}(x\cdot_{G_1}y))$ and $\hf_\infty \ni(x,y)\mapsto S'_{x,\gamma'}(x\cdot_{G_0} y)= S'_{0,\gamma}(x^{-1}\cdot_{G_1}(x\cdot_{G_0}y))$ are polynomials; therefore, it is easily seen that $p_{0,\gamma,\gamma'}$ and $p'_{0,\gamma,\gamma'}$ are polynomials.

	Finally, let us prove that $p_{0,\gamma,\gamma'}$	and $p'_{0,\gamma,\gamma'}$ are sums of homogeneous polynomials of degrees strictly greater than $\dd_{\gamma'}-\dd_{\gamma}$. Indeed, observe that the continuity of $P_{0,s}$ and $[\,\cdot\,,\,\cdot\,]_{0,s}$ in $s$ at $0$ shows that
	\[
	0=\lim_{s\to 0^+} (\vect{X}_0^\gamma-\vect{X}_s^\gamma)=\lim_{s\to 0^+} \sum_{\gamma'} s^{\dd_{\gamma}-\dd_{\gamma'}} p_{0,\gamma,{\gamma'}}(s\,\cdot\,)\vect{X}_{0}^{\gamma'}=\lim_{s\to 0^+} \sum_{\gamma'} s^{\dd_{\gamma}-\dd_{\gamma'}} p'_{0,\gamma,{\gamma'}}(s\,\cdot\,)\vect{X}_{s}^{\gamma'}.
	\]
	Since the $\vect{X}_s^{\gamma'}$ are pointwise linearly independent for every $s$ and converge to the $\vect{X}_0^{\gamma'}$, we must have $p_{0,\gamma,\gamma'}(s\cdot x),p'_{0,\gamma,\gamma'}(s\cdot x)=o(s^{\dd_{\gamma'}-\dd_{\gamma}})$ for $s\to 0^+$, for every $x\in \hf_0$. The assertion follows in this case. 
	
	The properties of the families $(p_{\infty,\gamma,\gamma'})$ and $(p'_{\infty,\gamma,\gamma'})$ are proved even more easily.	
\end{proof}

\subsection{Moduli}

Here we construct some control moduli on $\widetilde G$ and the $G_s$, following~\cite[§ 2.3]{Martini2}.
Cf.~also~\cite{ElstRobinson,NagelSteinWainger} for more details on `weighted' control distances.

\begin{deff}
	For every $x\in \widetilde G$, we define $\abs{x}$ (resp.\ $\abs{x}_*$) as the greatest lower bound of the set of $\eps>0$ such that there are an absolutely continuous curve $\gamma\colon [0,1]\to \widetilde G$ and some measurable  functions $a_j\colon [0,1]\to \R$ such that $\norm{a_j}_\infty\meg \eps^{\dd_j}$ (resp.\ $\norm{a_j}_\infty\meg \min(\eps, \eps^{\dd_j})$) for every $j\in J$, such that $\gamma(0)=e$ and $\gamma(1)=x$, and such that
	\[
	\gamma'(t)= \sum_{j\in J} a_j (t) (\widetilde X_{j})_{\gamma(t)}
	\]
	for almost every $t\in [0,1]$. 
	We define $B(r)$ (resp.\ $B_*(r)$) as the set of $x\in \widetilde G$ such that $\abs{x}<r$ (resp.\ $\abs{x}_*<r$), for every $r>0$.
\end{deff}

\begin{prop}\label{prop:8}
	The following hold:
	\begin{itemize}
		\item $\abs{\,\cdot\,}$ and $\abs{\,\cdot\,}_*$ are finite, symmetric, proper, and vanish only at $e$;
		
		\item  $\abs{z_1 z_2}\meg \abs{z_1}+\abs{z_2}$ and $\abs{z_1 z_2}_{*}\meg \abs{z_1}_{*}+\abs{z_2}_{*}$ for every  $z_1,z_2\in \widetilde G$;
		
		\item $\abs{z}= \abs{z}_{*}$ for every  $z\in \widetilde G$ such that $\abs{z}_{*}\meg 1$ (or, equivalently, $\abs{z}\meg 1$); in addition, $\abs{z}\meg \abs{z}_{*}\meg \abs{z}^n$ for every $z\in \widetilde G$ such that $\abs{z}_{*}\Meg 1$;
		
		\item $B_*(1)^h\subseteq B_*(r)\subseteq B_*(1)^{h+1}$ for every $r\in [h,h+1]$ and for every $h\in \N$;
		
		\item $\abs{\,\cdot\,}$ is a homogeneous norm.
	\end{itemize}
\end{prop}

The proof is simple and is omitted.

\medskip

In order to provide some more insight into the moduli $\abs{\,\cdot\,}$ and $\abs{\,\cdot\,}_*$, let us introduce some more notation.
First, we define  $\abs{x}'_{R}$ as the greatest lower bound of the set of $\eps>0$ such that there are an absolutely continuous curve $\gamma\colon [0,1]\to \widetilde G$ and some measurable  functions $a_j\colon [0,1]\to \R$ such that $\norm{a_j}_\infty\meg \eps$ for every $j\in J$, such that $\gamma(0)=e$ and $\gamma(1)=x$, and such that
\[
\gamma'(t)= \sum_{j\in J} a_j (t) (\widetilde X_{j})_{\gamma(t)}
\]
for almost every $t\in [0,1]$.  Then, it is not hard to see that the following hold:
\begin{itemize}
	\item $\abs{x}_*=\abs{x}'_{R}$ for every $x\in \widetilde G$ such that $\abs{x}_*\Meg 1$ or, equivalently, $\abs{x}'_{R}\Meg 1$;
	
	\item $\abs{x}^n_*\meg \abs{x}'_{R}\meg \abs{x}_*$ for every $x\in \widetilde G$ such that $\abs{x}_*\meg 1$ or, equivalently, $\abs{x}'_{R}\meg 1$;
	
	\item $\abs{x}_*=\max(\abs{x},\abs{x}'_{R})$ for every $x\in \widetilde G$.
\end{itemize}

In addition, if we denote by $d_R$ the (left-invariant) Riemannian distance associated with the (left-invariant) Riemannian metric for which $(\widetilde X_j)_{j\in J}$ is an orthonormal basis, then $\abs{x}'_R\meg d_R(0,x)\meg\dim \widetilde G\,\abs{x}'_R$ for every $x\in \widetilde G$. Consequently, $\abs{\,\cdot\,}_*$ is a reasonable compromise between a homogeneous norm (locally) and a Riemannian distance (globally).

\begin{deff}
	For every $s\in [0,\infty]$ and for every $x\in G_s$, we define
	\[
	\abs{x}_s\coloneqq \inf_{\pi_s(z)=x} \abs{z}\qquad \text{and}\qquad \abs{x}_{s,*}\coloneqq \inf_{\pi_s(z)=x} \abs{z}_*.
	\]
	We define $B_s(r)$ (resp.\ $B_{s,*}(r)$) as the set of $x\in \widetilde G_s$ such that $\abs{x}_s<r$ (resp.\ $\abs{x}_{s,*}<r$), for every $r>0$.
\end{deff}

One may prove that the moduli $\abs{\,\cdot\,}_s$ and $ \abs{\,\cdot\,}_{s,*}$ can be defined in the same fashion of the moduli $\abs{\,\cdot\,}$ and $\abs{\,\cdot\,}_*$. 
We leave the details to the reader.

\begin{prop}\label{prop:7}
	The following hold:
	\begin{enumerate}
		\item $\abs{\,\cdot\,}_s$ and $\abs{\,\cdot\,}_{s,*}$ are  symmetric, subadditive, proper, and vanish only at $e$;

		\item $\abs{x}_s= \abs{x}_{s,*}$ for every $s\in [0,\infty]$ and for every $x\in G_s$ such that $\abs{x}_{s,*}\meg 1$; in addition, $\abs{x}_s\meg \abs{x}_{s,*}\meg \abs{x}_s^n$ for every $x\in G_s$ such that $\abs{x}_{s,*}\Meg 1$;
		
		\item $B_{s,*}(1)^h\subseteq B_{s,*}(r)\subseteq B_{s,*}(1)^{h+1}$ for every $s\in [0,\infty]$, for every $r\in [h,h+1]$, and for every $h\in \N$;
		
		\item $\abs{r\cdot x}_s= r\abs{x}_{r s}$ for every $s\in [0,\infty]$, for every $r>0$, and for every $x\in G_{r s}$;
		
		\item the mappings $[0,\infty]\times \widetilde G\ni(s,z)\mapsto \abs{\pi_s(z)}_s$ and $[0,\infty]\times \widetilde G\ni(s,z)\mapsto \abs{\pi_s(z)}_{s,*}$ are continuous; 
		
		\item  	there is a constant $C>0$ such that
		\[
		\frac{1}{C}\min(\abs{P_{0,s}(z)},\abs{P_{\infty,s}(z)})\meg \abs{\pi_s(z)}_s\meg C \min(\abs{P_{0,s}(z)},\abs{P_{\infty,s}(z)})
		\]
		for every $s\in (0,\infty)$ and for every $z\in \widetilde \gf$.
	\end{enumerate}
\end{prop}

\begin{proof}
	{\bf1--4.}  These assertions follow from the corresponding ones of Proposition~\ref{prop:8}.

	{\bf5.} Fix $z\in \widetilde \gf$ and observe that, since $\abs{\,\cdot\,}$ is proper, for every $s\in [0,\infty]$ there is $y_s\in \ifr_s$ such that $\abs{z+y_s}=\abs{\pi_s(z)}_s$. 
	In particular, $\abs{z+y_s}\meg \abs{z}$, so that the set $\Set{y_s\colon s\in [0,\infty]}$ is relatively compact in~$\widetilde \gf$. 
	Then, fix $s'\in [0,\infty]$ and observe that there is a sequence $(s_k)$ of elements of $[0,\infty]$ converging to $s'$ such that $\lim_{k\to \infty} \abs{\pi_{s_k}(z)}_{s_k}=\liminf_{s\to s'} \abs{\pi_s(z)}_s$. 
	Notice that we may assume that $(y_{s_k})$ converges to some $y'$ in~$\widetilde \gf$, so that $y'\in \ifr_{s'}$. Therefore,
	\[
	\abs{\pi_{s'}(z)}_{s'}\meg \abs{z+y'}= \lim_{k\to \infty} \abs{z+y_{s_k}}=\lim_{k\to \infty} \abs{\pi_{s_k}(z)}_{s_k}=\liminf_{s\to s'} \abs{\pi_{s}(z)}_s.
	\]
	Conversely, take a sequence $(s'_k)$  of elements of $[0,\infty]$ converging to $s'$ such that $\lim_{k\to \infty} \abs{\pi_{s'_k}(z)}_{s'_k}=\limsup_{s\to s'} \abs{\pi_s(z)}_s$, and observe that we may take $y'_{s'_k}\in \ifr_{s'_k}$, for every $k\in\N$, in such a way that the sequence $(y'_{s'_k})$ converges to $y_{s'}$. Therefore,
	\[
	\abs{\pi_{s'}(z)}_{s'}=\abs{z+y_{s'}}=\lim_{k\to \infty} \abs{z+y'_{s'_k}}\Meg \lim_{k\to \infty} \abs{\pi_{s'_k}(z)}_{s'_k}=\limsup_{s\to s'} \abs{\pi_s(z)}_z,
	\]
	whence the first assertion. The second assertion is proved similarly.
	
	{\bf 6.} The assertion follows from Proposition~\ref{prop:21:9} and from~{\bf4} and~{\bf5} above.
\end{proof}

\begin{deff}
	For every $s\in [0,\infty]$, we define $\nu_{G_s}$ as the unique Haar measure on $G_s$ such that $\nu_{G_s}\left(B_s(1) \right)=1$. 
	We define $D_{s}$, the volume growth of $G_s$,  in such a way that $\nu_{G_s}(U^k)\asymp k^{D_{s}}$ for $k\to \infty$ an for every compact neighbourhood $U$ of $e$ (cf., for instance,~\cite[Theorem II.1]{Guivarch}). 
\end{deff}

Notice that $\nu_{G_s}(B_{s,*}(r))\asymp r^{D_{s}} $ as $r\to +\infty$, for every $s\in [0,\infty]$, thanks to~{\bf3} of Proposition~\ref{prop:7}. 

\begin{cor}\label{cor:3}
	The following hold:
	\begin{enumerate}
		\item $D_{s}=D_{1}\Meg \max(D_{0}, D_{\infty})$ for every $s\in (0,\infty)$;
		
		\item $D_{1}\meg Q_\infty$;
		
		\item $D_{0}\meg Q_0$ (resp.\ $D_{\infty}\meg Q_\infty$), with equality if and only if $G_0$ (resp.\ $G_\infty$) is stratified. 
	\end{enumerate}
\end{cor}

Notice that it may happen that  either $D_{0}>D_{\infty}$,  or $D_{0}<D_{\infty}$,  or  $D_{1}> \max(D_{0},D_{\infty})$.
Indeed, consider the case $\widetilde G =\Hd^1\times \R$, $G=\Hd^1$, where $\Hd^1$ is the three-dimensional Heisenberg group; denote by $X,Y,T,U$ a basis of $\widetilde \gf$ such that $[X,Y]=T$ while the other commutators vanish, and endow $\widetilde G$ with coordinates such that $((z,t),u)$ corresponds to $\exp(\Re z X+ \Im z Y+ t T+ u U)$; endow $G$ with similar coordinates and define $\pi((z,t),u)\coloneqq (z,t+u)$. Define dilations on $\widetilde G$ so that $X,Y,T,U$ have degrees $1,1,2,3$, respectively.
Then, $\ifr=(T-U)\R$, $\ifr_0=U\R$, and $\ifr_\infty= T\R$, so that $G_0\cong \Hd^1$, and $G_\infty\cong \R^3$. Hence, in this case, $D_{0}=4>3=D_{\infty}$.

If, in the same example considered above, we choose dilations on $\widetilde G$ in such a way that $X,Y,T,U$ have degrees $1,1,2,1$, respectively, then $\ifr_0=T\R$ and $\ifr_\infty= U\R$. Consequently, $D_{0}=3<4=D_{\infty}$.

Finally, if we consider $\widetilde G$, $\pi$, and $G$ as the products of the ones in the preceding examples, then clearly $D_{1}=8>7=D_{0}=D_{\infty}$.

\begin{proof}
	{\bf1.} Since $G_s$ is isomorphic to $G_1$, for $s\in (0,\infty)$, it is clear that $D_{s}=D_{1}$. In addition, denote by $\gf_s$ the Lie algebra of $G_s$, and define inductively $\gf_{s,[1]}\coloneqq \gf_s$ and $\gf_{s,[j+1]}\coloneqq [\gf_s, \gf_{s,[j]}]$ for every $j\Meg 1$. Then, $D_{s}=\sum_{j\Meg 1} \dim \gf_{s,[j]}$ (cf., for example,~\cite[Theorem II.1]{Guivarch}). Now, since $\lim_{s\to 0^+}[x,y]_s= [x,y]_0$ for every $x,y\in \hf_0$, it is easily seen that $\dim \gf_{0,[j]}\meg \dim \gf_{1,[j]}$ for every $j\in \N$, whence $D_{0}\meg D_{1}$. In the same way one proves that $D_{\infty}\meg D_{1}$.
	
	{\bf2.} Indeed, Proposition~\ref{prop:7} and the above remarks imply that
	\[
	 r^{D_{1}} \asymp\nu_{G_1}(B_{1,*}(r))\meg \nu_{G_1}(B_1(r))\asymp	r^{Q_{\infty}}
	\]
	as $r\to+\infty$. The assertion follows.
	
	{\bf3.} This follows easily from the formula for $D_{s}$ used in~{\bf1}.
\end{proof}

Here is a simple result which will be useful later on. The proof, which is a simple modification of that of~\cite[VIII.1.1]{VaropoulosCoulhonSaloffCoste}, is omitted.

\begin{lem}\label{lem:7}
	For every $s\in [0,\infty]$, for every $p\in [1,\infty]$, for every $f\in C^1(G_s)$, and for every $x\in G_s$,
	\[
	\norm{f(\,\cdot\, x)-f}_p \meg \sum_{j\in J}\abs{x}_s^{\dd_j} \norm{X_{s,j} f}_p.
	\]
\end{lem}

We conclude this subsection with some uniform estimates on the growth of the volume of the balls associated with the $\abs{\,\cdot\,}_{s,*}$. Indeed, observe that the preceding facts prove that for every $s\in [0,\infty]$ there is a constant $C_s>0$ such that $\nu_{G_s}\left( B_{s,*}(r)\right)\meg C_s r^{D_{s}}$ for every $r\Meg 1$; however, we shall need to know that one may take the $C_s$ to be \emph{independent} of $s$. Actually, we shall prove a finer result, showing how the growth of the volume of balls decreases as $s$ approaches $0$ or $\infty$.

\begin{prop}\label{prop:6}
	There are constant $C>0$ and two integers $N_0,N_\infty\Meg 1$ such that
	\[
	\nu_{G_s}\left(B_{s,*}(r)\right)\meg C \begin{cases} 
	\max(r^{D_{0}}, s^{N_0}r^{D_{1}}) & \text{if $s\in [0,1]$}\\
	\max(r^{D_{\infty}}, s^{-N_\infty}r^{D_{1}}) & \text{if $s\in [1,\infty]$}
	\end{cases}
	\]
	for every $s\in [0,\infty]$ and for every $r\Meg 1$. 
	In addition, when $\widetilde G$ is stratified, so that $Q_0=D_{0}$ and $Q_\infty=D_{\infty}=D_{1}$ by Corollary~\ref{cor:3}, one may take $N_0=Q_\infty-Q_0$.
	
	In particular, for every $\eps>0$ there is a constant $C_\eps>0$, independent of $s$, such that
	\[
	\norm{\min((1+\abs{\,\cdot\,}_{s,*})^{-D_{0}-\eps}, s^{-N_0}(1+\abs{\,\cdot\,}_{s,*})^{-D_{1}-\eps}  )}_1\meg C_\eps
	\]
	for every $s\in [0,1]$, while
	\[
	\norm{\min((1+\abs{\,\cdot\,}_{s,*})^{-D_{\infty}-\eps}, s^{N_\infty}(1+\abs{\,\cdot\,}_{s,*})^{-D_{1}-\eps}  )}_1\meg C_\eps
	\]
	for every $s\in [1,\infty]$.
\end{prop}
Notice that, when $\widetilde G$ is not stratified, then (the optimal) $N_0$ and $N_\infty$ may be smaller or larger than $D_{1}-D_{0}$ and $D_{1}-D_{\infty}$, respectively.

Let $F_k$ be the Lie group whose Lie algebra has a basis $X, Y_1,\dots, Y_k$ such that $Y_{j+1}=[X,Y_j]$ for every $j=1,\dots,k-1 $, while the other commutators vanish. Consider $\widetilde G\coloneqq \R\times F_k$, with basis of the corresponding Lie algebra $U,X,Y_1,\dots, Y_k$. Fix $d,d'\in \N^*$ such that $d< k+d'-1$. Give degree $1$ to $X$, degree $j+d'-1$ to $Y_j$ ($j=1,\dots, k$), and degree $d$ to $U$. 
Define $\ifr_1\coloneqq \langle Y_k -U\rangle$, so that $\ifr_s= \langle Y_k- s^{k+d'-1-d}U\rangle$ for every $s\in [0,\infty)$.
Then, we may choose $\hf_0=\langle X,Y_1,\dots, Y_{k-1},U\rangle$ and a neighbourhood of the identity $Q\coloneqq [-1,1]^{k+1}$ (in the coordinates associated with the basis $X,Y_1,\dots, Y_{k-1},U$). Then, the Baker--Campbell--Hausdorff formula shows that, for every $s\in [0,\infty]$,
\[
(x_1,y_1,u_1)\cdot_{G_s}\cdots \cdot_{G_s}(x_h,y_h,u_h)=(P_h(x_1,y_1;\dots; x_h,y_h), u_1+\dots+u_h+ s^{k+d'-1-d}R_h(x_1,y_1;\dots; x_h,y_h) )
\]
for every $(x_j,y_j,u_j)\in Q$ (with  $y_j\in [-1,1]^{k-1}$), $j=1,\dots,h$, where $P_h$ and $R_h$ are suitable polynomial mappings (independent of $s$). Integrating in $(x,y)$ first and then in $u$, we see that
\[
\nu_{\hf_0}(Q^{\cdot_{G_s} h})= (1-s^{k+d'-1-d})\nu_{\hf_0}(Q^{\cdot_{G_0} h})+s^{k+d'-1-d}\nu_{\hf_0}(Q^{\cdot_{G_1} h})\asymp h^{D_{0}}+ s^{k+d'-1-d} h^{D_{1}}
\]
for $h\to \infty$, uniformly for $s\in [0,1]$, where $\nu_{\hf_0}$ denotes Lebesgue measure on $\hf_0$. 
Now, it is not hard to see that this quantity is comparable with $\nu_{G_s}\left( B_{s,*}(h)\right)$ (uniformly for $s\in [0,1]$ and $h\Meg 1$), so that $N_0=k+d'-1-d$, which may be either smaller or larger than $k-1= D_{1}-D_{0}$.

Choosing $d>k+d'-1$, one may then obtain examples with $N_\infty$ either smaller or larger than $D_{1}-D_{\infty}$.
Taking products, examples with both $N_0 -(D_{1}-D_{0})\neq0$ and $N_\infty-(D_{1}-D_{\infty})\neq 0$ (with all combinations of signs) may be produced.

\begin{proof}
	{\bf1.} We consider only the case $s\in [0,1]$, since the case $s\in [1,\infty]$ is completely analogous (or almost trivial when $\widetilde G$ is stratified, see~{\bf4} below).  Define $\widetilde \gf_{[1]}\coloneqq \widetilde \gf$ and, by induction, $\widetilde \gf_{[k+1]}\coloneqq [\widetilde \gf, \widetilde \gf_{[k]}]$, so that $(\widetilde \gf_{[k]})$ is a decreasing sequence of graded ideals of $\widetilde \gf$ (the lower central series).
	Notice that, arguing as in the proof of Proposition~\ref{prop-psi}, one may prove that $\ifr_s\cap \widetilde \gf_{[k]}$ converges to some limit $\ifr_{0,[k]}\subseteq \ifr_0\cap \widetilde \gf_{[k]}$ as $s\to 0^+$, for every $k\in \N^*$.
	Then, for every $k\in\N^*$ choose a graded complement $V_k$ of $(\ifr_0\cap \widetilde \gf_{[k]})+\widetilde \gf_{[k+1]} $ in $\widetilde \gf_{[k]}$ and a graded complement $W_k$ of $\ifr_{0,k}$ in $\ifr_0\cap \widetilde \gf_{[k]}$.
	Observe that $\bigoplus_{k'\Meg k} V_{k'}$ is a graded complement of $\ifr_0\cap \widetilde \gf_{[k]}$ in $\widetilde \gf_{[k]}$ for every $k\in \N^*$, so that we may assume that $\hf_0=\bigoplus_k V_k$. 
	Analogously, observe that $W_k\oplus(\bigoplus_{k'\Meg k} V_{k'})$ is a graded complement of $\ifr_{0,k}$ in $\widetilde \gf_{[k]}$; arguing as in the proof of Proposition~\ref{prop-psi}, we then see that $W_k\oplus(\bigoplus_{k'\Meg k} V_{k'})$ is a graded complement of $\ifr_s\cap \widetilde \gf_{[k]}$  in $\widetilde \gf_{[k]}$ for every $s\in (0,\infty)$.

	Then, we may find a family $(k_j)_{j\in J_0}$ of positive integers and a homogeneous basis $(\widetilde Y_j)_{j\in J_0}$ of $\hf_0$ such that $(\widetilde Y_j)_{k_j=k}$ is a  basis of $V_k$, for every $k\in \N^*$, and such that $\widetilde Y_j$ has degree $\dd_j$ for every $j\in J_0$.
	Choose, in addition, a homogeneous basis $(\widetilde Y_{j})_{j\in J_k}$ of $W_k$ for every $k\in \N^*$ (to make the notation consistent, we assume that the $J_k$, for $k\in \N$, are mutually disjoint); we define $k_j\coloneqq k$  and we denote by $\dd_j$ the degree of $\widetilde Y_j$ for every $j\in J_k$. 
	Define $\widetilde Y^{(s)}_{j}\coloneqq P_{0,s}(\widetilde Y_{j})$ for every $j\in \widetilde J\coloneqq \bigcup_{k\in \N} J_k$; observe that $\widetilde Y^{(s)}_j=\widetilde Y_j$ for some (hence every) $s\in [0,1]$ if and only if $j\in J_0$, and that $\widetilde Y^{(0)}_j=0$ if and only if $j\in \widetilde J \setminus J_0$.

	{\bf2.} Observe that~{\bf5} of Proposition~\ref{prop:7} shows that there is a constant $C_1>0$ such that, under the identification of $G_s$ with $\hf_0$,
	\[
	B_{s,*}(1)\subseteq \sum_{j\in J_0} [-C_1,C_1] \widetilde Y_{j}\eqqcolon Q_{s}
	\]
	for every $s\in [0,1]$. 
	In addition, denoting by $\nu_{\hf_0}$ the (fixed) Lebesgue measure on $\hf_0$, again by~{\bf5} of Proposition~\ref{prop:7} we see that there is a constant $C_2>0$ such that
	\[
	\nu_{G_s} \meg C_2\nu_{\hf_0}  
	\]
	under the identification of $G_s$ with $\hf_0$. Thanks to~{\bf3} of Proposition~\ref{prop:7},  it will then suffice to estimate $\nu_{\hf_0}(Q_{s}^{\cdot_{G_s}h})$ for every $h\in \N^*$ and for every $s\in [0,1]$.
	
	{\bf3.} Now, observe that, arguing as in the proof of~\cite[Theorem 2 of Chapter II, § 6, No.\ 4]{BourbakiLie2}, we see that, for every $j_1,\dots,j_h\in J_0$,
	\[
	\widetilde Y_{j_1}\cdots \widetilde Y_{j_h}=\sum_{m=1}^\infty \frac{(-1)^{m-1}}{m} \sum_{\abs{\ell_1},\dots,\abs{\ell_m}\Meg 1} \frac{1}{\ell_1!\cdots \ell_m!} [\widetilde Y_{j_1}^{\ell_{1,1}}\cdots \widetilde Y_{j_h}^{\ell_{m,h}}   ],
	\]
	where 
	\[
	[\widetilde Y_{j_1}^{\ell'_{1,1}}\cdots \widetilde Y_{j_h}^{\ell'_{m,h}}   ]= (\ad (\widetilde Y_{j_1})^{\ell'_{1,1}}\cdots \ad (\widetilde Y_{j_h})^{\ell'_{1,h}})\cdots (\ad (\widetilde Y_{j_1}^{\ell'_{m,1}})\cdots \ad (\widetilde Y_{j_{h}})^{\ell'_{m,h}}) \widetilde Y_{j_{\bar h}}  ,
	\]
	where $\bar h\coloneqq \max \Set{h'\colon \ell_{m,h'}\neq 0}$, $\ell'_{m'}\coloneqq\ell_{m'}$ for every $m'=1,\dots,m-1$, and $\ell'_m=\ell_m-(\delta_{h',\bar h})_{h'}$.
	Then,  taking~{\bf1} into account, we see that 
	\[
	[Y_{j_1}^{\ell_{1,1}}\cdots Y_{j_h}^{\ell_{m,h}}   ]\in \langle (\widetilde Y_j)_{k_j \Meg \abs{\ell_1+\dots+\ell_m},\, \dd_j\meg \ell_{1,1}\dd_{j_1}+\cdots+ \ell_{m,h} \dd_{j_h}  } \rangle+\ifr_{1}
	\]
	for every $\ell_1,\dots ,\ell_m$ with $\abs{\ell_1},\dots, \abs{\ell_m}\Meg 1$ and for every $m\in \N^*$. Therefore, there is a constant $C_3>0$ such that
	\[
	Q_{s}^{\cdot_{G_s}h}\subseteq \sum_{j\in \widetilde J} [-C_3 h^{k_j} ,C_3 h^{k_j}] \widetilde Y_{j}^{(s)}
	\]
	for every $s\in [0,1]$ and for every $h\in \N^*$. 
	Now, arguing by induction on $\card(\widetilde J)\Meg \card(J_0)$, we see that
	\[
	 \sum_{j\in \widetilde J} [-C_3 h^{k_j} ,C_3 h^{k_j}] \widetilde Y_{j}^{(s)}= \bigcup_{\substack{J',J''\subseteq \widetilde J\\ J'\cap J''=\emptyset\\ \card(J'')=\card(J_0)}} \bigcup_{\substack{\eps'\in \Set{-1,1}^{J'}\\ \eps''\in \Set{-1,1}^{J''}}} \left( \sum_{j'\in J'} C_3 h^{k_{j'}}\eps'_{j'} \widetilde Y^{(s)}_{j'}+\sum_{j''\in J''}[0, 1] C_3 \eps''_{j''} h^{k_{j''}} \widetilde Y^{(s)}_{j''}   \right)
	\]
	Therefore,
	\[
	\begin{split}
	\nu_{\hf_0}(Q_{s}^{_{G_s}h})&\meg 2^{2\card(\widetilde J)-\card(J_0)}\sum_{J'\subseteq \widetilde J, \card(J')=\card(J_0)}  \nu_{\hf_0}\left(\sum_{j\in J'}[0, 1] C_3 h^{k_{j}} \widetilde Y^{(s)}_{j}\right)\\
		&=2^{2\card(\widetilde J)}  \left(\frac{C_3}{2}\right)^{\card(J_0)} \sum_{J'\subseteq \widetilde J, \card(J')=\card(J_0)} h^{k_{J'}} \det\left(\widetilde Y^{(s)}_{J'}\right),
	\end{split}
	\]
	where $k_{J'}\coloneqq \sum_{j\in J'} k_j$ and $\det\left(\widetilde Y^{(s)}_{J'}\right)$ is the determinant of the basis $\left(\widetilde Y^{(s)}_j\right)_{j\in J'}$ of $\hf_0$ with respect to the measure $\nu_{\hf_0}$ (that is, with respect to any basis whose fundamental parallelotope has measure $1$).
	
	Now, take $J'\subseteq \widetilde J$ such that $\card(J')=\card(J_0)$ and $\det\left(\widetilde Y^{(s)}_{J'}\right)\neq 0$ for some (hence every) $s\in (0,1]$. Observe that, since $\det\left(\widetilde Y^{(s)}_{J'}\right)=\det\left(\widetilde Y^{(0)}_{J'}\right)+O(s)$, the first assertion will be established if we prove that $k_{J'}\meg D_{1}$ for every such $J'$, and that $k_{J'}= D_{0}$ if $\det\left(\widetilde Y^{(0)}_{J'}\right)\neq 0$.
	Then, for every $k\in \N^*$ define $J'_k$ as the set of $j\in J'$ such that $k_j=k$, and observe that $(\widetilde Y_j)_{j\in \bigcup_{k'\Meg k}J'_{k'}}$ is the basis of a graded subspace of $\widetilde \gf_{[k]}$ whose intersection with $\ifr_s\cap \widetilde \gf_{[k]}$ is $0$ for every $s\in (0,1]$, since $(\widetilde Y_j^{(s)})_{j\in \bigcup_{k'\Meg k}J'_{k'}}$ is the basis of a subspace of $\widetilde \gf_{[k]}+\ifr_s$ whose intersection with $\ifr_s$ is $0$ and $\widetilde Y_j-\widetilde Y_j^{(s)}\in \ifr_s$ for every $s\in (0,1]$ and for every $j\in J'$. 
	Therefore, $\sum_{k'\Meg k} \card(J'_{k'})\meg \dim [(\gf_{[k]}+\ifr_s)/\ifr_s]$ so that, summing over $k\in \N^*$,
	\[
	k_{J'}=\sum_{k\in \N^*} k \card(J'_k)\meg \sum_{k\in \N^*}\dim [(\gf_{[k]}+\ifr_s)/\ifr_s]= D_{s}
	\]
	for every $s\in (0,1]$,
	where the last equality follows from~\cite[Theorem II.1]{Guivarch}.
	
	Finally, assume that $\det\left(\widetilde Y^{(0)}_{J'}\right)\neq 0$. Then, by~{\bf1} we see that $J'=J_0$, so that the assertion follows arguing as before.

	{\bf4.} Now, assume that $\widetilde G$ is stratified. Then, it is clear that $\widetilde \gf_{[k]}=\bigoplus_{q\Meg k} \widetilde \gf_q$, so that the assertion for $s\in [1,\infty]$ is trivial. Then, consider the preceding construction for $s\in [0,1]$ and observe that $k_j\meg\dd_j$  for every $j\in \widetilde J$, with equality when $j\in J_0$. 
	Take $J'$ as in~{\bf3}.
	
	Observe that we may construct, by induction on $k=1,\dots,n$, mutually disjoint subsets $J'_k$ of $J'$ such that $J'\cap J_0\subseteq J'_k$  and such that $(\widetilde Y_j^{(s)})_{j\in \bigcup_{k'\Meg k} J'_{k'}}$ is the basis of a graded  complement of $\left( \bigoplus_{q<k} \widetilde \gf_q\right)\cap \hf_0$ in $\hf_0$, for every $k=1,\dots,n$. 
	Define $k'_j\coloneqq k$ for every $j\in J'_k$ and for every $k=1,\dots,n$, and observe that  $\dd_j\Meg k'_j$ for every $j\in J'$ thanks to Proposition~\ref{lem:21:1}. 
	Furthermore,  define $\widetilde Y^{J'}_j$ as the homogeneous component of degree $k'_j$ of $\widetilde Y_j^{(1)}$ for every $j\in J'$, and observe that $(\widetilde Y_j^{J'})_{j\in J'}$ is a basis of $\hf_0$. 
	In addition, arguing as in~{\bf3} above we see that $ \sum_{k=1}^n k \card(J'_k)=\sum_{k=1}^n k \dim(\hf_0\cap \widetilde \gf_k) =Q_0=D_{0}$ since $G_0$ is a stratified group, so that
	\[
	h^{k_{J'}} \det\left( \widetilde Y^{(s)}_{J'} \right)= h^{D_{0}} (h s)^{k_{J'}-D_{0}} s^{\sum_{j\in J'} \dd_j-k_{J'}} \det\left((s^{k'_j-\dd_j}  \widetilde Y^{(s)}_j)_{j\in J'}\right).
	\]
	Now, observe that $\sum_{j\in J'} \dd_j-k_{J'}\Meg 0 $, so that our assertion will be established if we prove that $\det\left((s^{k'_j-\delta_j}  \widetilde Y^{(s)}_j)_{j\in J'}\right) $ is independent of $s\in [0,1]$, and hence equal to $\det\left((\widetilde Y^{J'}_j)_{j\in J'}\right)$. 
	To prove this fact one may use Gauss elimination to the family $(s^{k'_j-\delta_j}  \widetilde Y^{(s)}_j)$ (more precisely, to the matrix of the coordinates of the vectors $\widetilde Y^{(s)}_j$ with respect to the basis $(\widetilde Y_j)_{j\in J_0}$)  and observe that, by homogeneity arguments, the resulting family is (linearly independent and) independent of $s$. 
\end{proof}

\section{Estimates of the Heat Kernel}\label{sec:2}

We now introduce the operators in which we shall be mainly interested.
{
Fix a homogeneous left-invariant differential operator $\widetilde \Lc$ on $\widetilde G$ such that $\widetilde \Lc+\widetilde \Lc^*$ is a positive Rockland operator of degree $\delta$; then, we define $\Lc_s\coloneqq \dd\pi_s(\widetilde \Lc)$ for every $s\in [0,\infty]$. We shall sometimes write $\Lc$ instead of $\Lc_1$ to simplify the notation.\footnote{Notice that in~\cite{NagelRicciStein} the operator $\widetilde \Lc$ is only required to be Rockland; nonetheless, since we are interested in the corresponding heat kernels,  additional restrictions have to be imposed.}}
Then, the operators $\widetilde \Lc$, $\Lc$, and $\Lc_{s}$ are weighted subcoercive, hence hypoelliptic (cf.~\cite[Theorem 2.3]{Martini2}).

Denote by $(\widetilde h_t)_{t>0}$ the heat kernel of $\widetilde \Lc$, which we shall consider as a semigroup of \emph{measures} on $\widetilde G$. 
In addition, for every $s\in [0,\infty]$ and for every $t>0$, we shall define $h_{s,t}\coloneqq (\pi_s)_*(\widetilde h_t)$, so that $(h_{s,t})_{t>0}$ is the heat kernel of $\Lc_s$. Observe that
\[
h_{r s,t}=(r^{-1}\,\cdot\,)_* h_{s, r^\delta t}
\]
for every $r>0$, for every $s\in [0,\infty]$, and for every $t>0$.

We  fix a Lebesgue measure on $\widetilde \gf$ and identify  $\widetilde h_t$ with its density. With the $h_{s,t}$ we shall be more careful, though.
Indeed, for $s\in (0,\infty)$ the group $G_s$ can be identified with both $\hf_0$ and $\hf_\infty$, and it is not possible to find Lebesgue measures on $\hf_0$ and $\hf_\infty$ which induce the same measure on $G_s$ for all $s\in (0,\infty)$. 
Therefore, we shall fix two Lebesgue measures on $\hf_0$ and $\hf_\infty$ and define two densities $h_{0,s,t}$ and $h_{\infty,s,t}$ of $h_{s,t}$ accordingly.

Precisely, for $s\in [0,\infty)$, we  define $h_{0,s,t}$ as the density of $(P_{0,s})_*(\widetilde h_{t})$  with respect to the fixed Lebesgue measure on $\hf_0$; in this way, $h_{0,s,t}$ becomes the (density of) $h_{s,t}$, under the identification of   $\gf_s$ (hence of $G_s$) with $\hf_0$ given in Definition~\ref{def:2}.
Observe that, with these choices (and with a suitable Lebesgue measure on $\ifr_0$, independent of $s$),
\[
h_{0,s,t}(x)=\int_{\ifr_0} \widetilde h_t(x+y+\psi_{0,s}(y))\,\dd y
\]
for every $s\in [0,\infty)$, for every $t>0$, and for every $x\in \hf_0$.

Analogously, for $s\in (0,\infty]$ we shall define $h_{\infty,s,t}$ as the density of $(P_{\infty,s})_*(\widetilde h_{t})$ with respect to the fixed Lebesgue measure on $\hf_\infty$. Similar remarks apply.

We now prove some uniform estimates on $h_{0,s,t}$ and $h_{\infty,s,t}$ and their derivatives which cannot be derived from the general estimates for weighted subcoercive operators.

\begin{teo}\label{prop:21:2}
	Fix $c>0$ and $d\in\R$, and let $X_0$ and $X_\infty$ be two homogeneous differential operators with continuous coefficients on $\hf_0$ and $\hf_\infty$, respectively, of order $d$. 
	Then, for every $k\in \N$ there are two constants $C,b>0$ (independent of $s$) such that
	\[
	\abs{X_\infty \partial_s^k h_{\infty,\sfrac{1}{s},t}(x)}\meg  \frac{C}{ t^{\frac{Q_\infty+ d+k}{\delta}} } e^{-b \abs{ \pi_{\sfrac{1}{s}}(t^{-\sfrac{1}{\delta}}\cdot x)}_{\sfrac{1}{s},*} ^{\frac{\delta}{\delta-1}}}\meg \frac{C}{ t^{\frac{Q_\infty+ d+k}{\delta}} } e^{-b\left( \frac{\abs{\pi_{\sfrac{1}{s}}(x)}_{\sfrac{1}{s}}}{t^{\sfrac{1}{\delta}}} \right)^{\frac{\delta}{\delta-1}}}, 
	\]
	for every $s\in [0,\infty)$, for every $x\in \hf_\infty$, and for every $t> c s^\delta$, while
	\[
	\abs{X_0 \partial_s^k h_{0,s,t}(x)}\meg \frac{C}{ t^{\frac{Q_0+ d-k}{\delta}} } e^{-b \abs*{\pi_s(t^{-\sfrac{1}{\delta}}\cdot x)}_{s,*}^{\frac{\delta}{\delta-1}}}\meg \frac{C}{ t^{\frac{Q_0+ d-k}{\delta}} } e^{-b \left( \frac{\abs*{ \pi_s(x)}_s}{t^{\sfrac{1}{\delta}}}\right) ^{\frac{\delta}{\delta-1}}}
	\]
	for every $s\in [0,\infty)$, for every $x\in \hf_0$, and for every  $t\in (0,c s^{-\delta}]$.
\end{teo}

\begin{proof}
	{\bf1.} Consider the first assertion; notice that we may reduce to the case in which $X_\infty=f \partial_\infty^\alpha$, where $f$ is a continuous homogeneous function on $\hf_\infty$ of degree $\dd_\alpha-d$; notice that $\dd_\alpha-d>0$ since $f$ is continuous.
	Now, observe that, with a change of variables,
	\[
	h_{\infty,\sfrac{1}{s},t}(x)=  \int_{\ifr_\infty}\widetilde h_t(x+y+\psi_{\infty,\sfrac{1}{s}}(y)) \,\dd y=t^{-\frac{ Q_\infty}{\delta}} \int_{\ifr_\infty} \widetilde h_1(t^{-\sfrac{1}{\delta}}\cdot x+y+\psi_{\infty,(t^{-\sfrac{1}{\delta}}s)^{-1}}(y))\,\dd y
	\]
	for every $x\in \hf_\infty$, for every $s\in [0,\infty)$, and for every $t>0$. Therefore, Faà di Bruno's formula shows that
	\[
	\begin{split}
	X_\infty\partial_s^{k}h_{\infty,\sfrac{1}{s},t}(x)&=\frac{f(x)}{t^{\frac{Q_\infty+\dd_\alpha+k}{\delta}}}\sum_{\sum_{\ell=1}^k \ell \abs{\gamma_\ell}=k}\frac{k!}{\gamma!} \int_{\ifr_\infty} \partial_\infty^\alpha\partial^{\gamma_1+\cdots+ \gamma_k}\widetilde h_1(t^{-\sfrac{1}{\delta}}\cdot x+y+\psi_{\infty,(t^{-\sfrac{1}{\delta}}s)^{-1}}(y)) \times\\
		&\qquad \times \prod_{\ell=1}^{k} \left( \frac{1}{\ell!} \partial_{s'}^\ell\big\vert_{s'=t^{-\sfrac{1}{\delta}}s} \psi_{\infty,\sfrac{1}{s'}}(y)  \right)^{\gamma_\ell}\,\dd y
	\end{split}
	\]
	for every $s\in [0,\infty)$, for every $t>0$, for every $k\in \N$, and for every $x\in \hf_\infty$.
	In addition, observe that $\partial_{s'}^\ell \left(\pr_j\circ\psi_{\infty,\sfrac{1}{s'}}\right)$ is a (linear) polynomial of degree at most $j-\ell$ for every $j=2,\ldots,n$ and for every $\ell=1,\ldots, j-1$, and is $0$ otherwise.
	Therefore, there are $C_1,\widetilde b>0$ such that
	\begin{multline*}
		\abs*{f(x)\sum_{\sum_{\ell=1}^k \ell \abs{\gamma_\ell}=k}\frac{k!}{\gamma!}  \partial_\infty^\alpha\partial^{\gamma_1+\cdots+ \gamma_k}\widetilde h_1(t^{-\sfrac{1}{\delta}}\cdot x+y+\psi_{\infty,(t^{-\sfrac{1}{\delta}}s)^{-1}}(y)) \cdot \prod_{\ell=1}^{k} \left( \frac{1}{\ell!} \partial_{s'}^\ell\big\vert_{s'=t^{-\sfrac{1}{\delta}}s} \psi_{\infty,\sfrac{1}{s'}}(y)  \right)^{\gamma_\ell}}\\
		\meg C_1\abs{ x}^{\dd_\alpha-d}\sum_{\sum_{\ell=1}^k \ell \abs{\gamma_\ell}=k}  e^{-\widetilde b\abs{t^{-\sfrac{1}{\delta}}\cdot x+y+\psi_{\infty,(t^{-\sfrac{1}{\delta}}s)^{-1} }(y))}_*^{\frac{\delta}{\delta-1}}} (1+\abs{y})^{ \dd_{\gamma}-k}
	\end{multline*}
	for every $(x,y)\in\hf_\infty\oplus \ifr_\infty$, for every $s\in [0,\infty)$ and for every $t> c s^\delta$  (cf.~\cite[Theorem 2.3 (e)]{Martini2}).
	Now,
	\[
	\abs{t^{-\sfrac{1}{\delta}}\cdot (x+y+\psi_{\infty,\sfrac{1}{s}}(y))}_*\Meg \frac{1}{2}\abs{\pi_{\sfrac{1}{s}}(t^{-\sfrac{1}{\delta}}\cdot x)}_{\sfrac{1}{s},*}+\frac{1}{2}\abs{t^{-\sfrac{1}{\delta}}\cdot (x+y+\psi_{\infty,\sfrac{1}{s}}(y))}
	\]
	for every $(x,y)\in \hf_\infty\oplus \ifr_s$, for every $t>0$, and for every $s\in (0,\infty]$, with some abuses of notation.	
	Therefore,
	\[
	\begin{split}
	\abs{X_\infty\partial_s^k h_{\infty,\sfrac{1}{s},t}(x)}&	\meg   \frac{C_1}{ t^{\frac{Q_\infty+d+k}{\delta}}}e^{-2^{\frac{\delta}{1-\delta}}\widetilde b \abs{\pi_{\sfrac{1}{s}}(t^{-\sfrac{1}{\delta}}\cdot x)}_{\sfrac{1}{s},*}^{\frac{\delta}{\delta-1}} } \abs{t^{-\sfrac{1}{\delta}}\cdot x}^{\dd_\alpha-d} \times\\	
	&\qquad \times\sum_{\sum_{\ell=1}^k \ell \abs{\gamma_\ell}=k}  \int_{\ifr_\infty}e^{-2^{\frac{\delta}{1-\delta}}\widetilde b\abs*{t^{-\sfrac{1}{\delta}}\cdot x+y+\psi_{\infty,(t^{-\sfrac{1}{\delta}}s)^{-1}}( y)}^{\frac{\delta}{\delta-1}}} (1+\abs{y})^{\dd_{\gamma}-k}\,\dd y
	\end{split}
	\]
	
	Now, fix a norm $\norm{\,\cdot\,}$ on $\widetilde \gf$, and recall that $\psi_{\infty,\sfrac{1}{s'}}$ is strictly super-homogeneous, so that there are two constants $C_2,C_2'>0$ such that
	\[
	\begin{split}
	\abs*{\psi_{\infty,\sfrac{1}{s'}}(y)}&\meg C_2 \sum_{j=2}^n \norm*{\left( \pr_j\circ\psi_{\infty,\sfrac{1}{s'}}\right) (y_1,\ldots, y_{j-1})}^{\frac{1}{j}}\\
	&\meg C_2 \sum_{j=2}^n\norm*{\pr_j\circ\psi_{\infty,\sfrac{1}{s'}}} \left(\norm*{y_1}+\ldots+\norm*{y_{j-1}} \right)^{\frac{1}{j}}\\
	&\meg C_2' \max\left(\abs*{y}^{\frac{n-1}{n}},\abs*{y}^{\frac{1}{n}}\right),    
	\end{split}
	\]
	for every $y\in \ifr_\infty$ with homogeneous components $y_1,\dots,y_n$, and for every $s'\in [0, c^{\sfrac{1}{\delta}} ]$. In addition, observe that all homogeneous norms on $\widetilde G$ are equivalent and that both $\hf_\infty$ and $\ifr_\infty$ are homogeneous subspaces of $\widetilde \gf$, so that there is a constant $C_3>0$ such that $\abs{z_1+ z_2}\meg C_3(\abs{z_1}+\abs{z_2})$ for every $z_1,z_2\in \widetilde \gf$, and such that $\abs{x+y}\Meg \frac{1}{C_3}(\abs{x}+\abs{y})$ for every $(x,y)\in \hf_\infty \oplus \ifr_\infty$.
	In addition,  since $\frac{\delta}{\delta-1}\Meg1$, there is a constant $C_4\Meg 1$ such that
	\[
	a_1^{\frac{\delta}{\delta-1}}+a_2^{\frac{\delta}{\delta-1}} \meg (a_1+a_2)^{\frac{\delta}{\delta-1}}\meg C_4\left( a_1^{\frac{\delta}{\delta-1}}+a_2^{\frac{\delta}{\delta-1}}\right) 
	\]  
	Then, for every $x\in \hf_\infty$, for every $y\in \ifr_\infty$ and for every $t>c s^\delta$,
	\[
	\begin{split}
	\abs*{t^{-\sfrac{1}{\delta}}\cdot x+y+\psi_{\infty,(t^{-\sfrac{1}{\delta}}s)^{-1}}(y)}^{\frac{\delta}{\delta-1}}&\Meg\frac{1}{C_3^{\frac{\delta}{\delta-1}} C_4}\abs*{t^{-\sfrac{1}{\delta}}\cdot x+y}^{\frac{\delta}{\delta-1}}- \abs*{\psi_{\infty,(t^{-\sfrac{1}{\delta}}s)^{-1}}(y)}^{\frac{\delta}{\delta-1}}\\
	&\Meg \frac{1}{C_3^{\frac{2\delta}{\delta-1}}C_4}\left(  \abs{t^{-\sfrac{1}{\delta}}\cdot x}^{\frac{\delta}{\delta-1}}+ \abs{y} ^{\frac{\delta}{\delta-1}} \right) - C_2' \max\left(\abs*{y}^{\frac{n-1}{n}},\abs*{y}^{\frac{1}{n}}\right) .
	\end{split}
	\]
	Therefore, there is a  constants $C_5>0$ such that 
	\[
	2^{\frac{\delta}{1-\delta}}\widetilde b \abs*{t^{-\sfrac{1}{\delta}}\cdot x+y+\psi_{\infty,(t^{-\sfrac{1}{\delta}}s)^{-1}}(y)}^{\frac{\delta}{\delta-1}}\Meg \frac{1}{C_5}\left(\abs{t^{-\sfrac{1}{\delta}}\cdot x}^{\frac{\delta}{\delta-1}}+\abs{y} ^{\frac{\delta}{\delta-1}}\right) -C_5
	\]
	for every $(x,y)\in \hf_\infty\oplus \ifr_\infty$, for every $s\in [0,\infty)$, and for every $t>c s^\delta$.
	
	Hence, there is a constant $C_6>0$ such that
	\[
	\abs{t^{-\sfrac{1}{\delta}}\cdot x}^{\dd_\alpha-d} \sum_{\sum_{\ell=1}^k \ell \abs{\gamma_\ell}=k}  \int_{\ifr_\infty}e^{-2^{\frac{\delta}{1-\delta}}\widetilde b\abs*{t^{-\sfrac{1}{\delta}}\cdot x+y+\psi_{\infty,(t^{-\sfrac{1}{\delta}}s)^{-1}}(y)}^{\frac{\delta}{\delta-1}}} (1+\abs{y})^{\dd_{\gamma}-k}\,\dd y\meg C_6
	\]
	for every $x\in \hf_\infty$, for every $s\in [0,\infty)$, and for every $t>c s^\delta$, so that
	\[
	\begin{split}
	\abs{X_\infty\partial_s^k h_{\infty,\sfrac{1}{s},t}(x)}&\meg \frac{C_1 C_6}{t^{\frac{Q_\infty+d+k}{\delta}}}e^{-2^{\frac{\delta}{1-\delta}}\widetilde b \abs{\pi_{\sfrac{1}{s}}(t^{-\sfrac{1}{\delta}}\cdot x)}_{\sfrac{1}{s},*}^{\frac{\delta}{\delta-1}}  } .
	\end{split}
	\]
	
	{\bf2.} Consider, now, the second assertion. Observe that we may assume that $X_0=f \partial^\alpha_{\hf_0}$ for some $\alpha$ and some continuous homogeneous function $f$ on $\hf_0$ with degree $\dd_\alpha-d$. Notice that $\dd_\alpha-d>0$ since $f$ is continuous.
	Then, Faà di Bruno's formula shows that
	\[
	X_0\partial_s^{k}h_{0,s,t}(x)=\frac{f(x)}{t^{\frac{Q_0+\dd_\alpha-k}{\delta}}}\sum_{\sum_{\ell=1}^k \ell \abs{\gamma_\ell}=k}\frac{k!}{\gamma!} \int_{\ifr_0} \partial_0^\alpha\partial^{\gamma}\widetilde h_1(t^{-\sfrac{1}{\delta}}\cdot x+y+\psi_{0,t^{\sfrac{1}{\delta}}s}(y)) \cdot \prod_{\ell=1}^{k} \left( \frac{1}{\ell!} \partial_{s'}^\ell\big\vert_{s'=t^{\sfrac{1}{\delta}}s} \psi_{0,s'}(y)  \right)^{\gamma_\ell}\,\dd y
	\]
	for every $s\in [0,1]$, for every $t>0$, for every $k\in \N$, and for every $x\in \hf_0$.
	In addition, observe that $\partial_{s'}^\ell \left(\pr_j\circ\psi_{0,s'}\right)$ is a (linear) polynomial of degree at most $n$ and of homogeneous order at least $j+\ell$ for every $j=1,\ldots,n-1$ and for every $\ell=1,\ldots, n- j$, and is $0$ otherwise.
	Therefore, there are $C_1,\widetilde b>0$ such that 
	\begin{multline*}
		\abs*{f(x)\sum_{\sum_{\ell=1}^k \ell \abs{\gamma_\ell}=k}\frac{k!}{\gamma!}  \partial_0^\alpha\partial^{\gamma}\widetilde h_1(t^{-\sfrac{1}{\delta}}\cdot x+y+\psi_{0,t^{\sfrac{1}{\delta}}s}(y)) \cdot \prod_{\ell=1}^{k} \left( \frac{1}{\ell!} \partial_{s'}^\ell\big\vert_{s'=t^{\sfrac{1}{\delta}}s} \psi_{0,s'}(y)  \right)^{\gamma_\ell}}\\
		\meg \abs{ x}^{\dd_\alpha-d} \sum_{\sum_{\ell=1}^k \ell \abs{\gamma_\ell}=k} C_1 e^{-\widetilde b\abs{t^{-\sfrac{1}{\delta}}\cdot x+y+\psi_{0, t^{\sfrac{1}{\delta}}s}(y)}_{*}^{\frac{\delta}{\delta-1}}} \abs{y}^{\dd_{\gamma}+k}(1+\abs{y})^{n\abs{\gamma}-\dd_\gamma-k}
	\end{multline*}
	for every $(x,y)\in\hf_0\oplus \ifr_0$, for every $s\in [0,\infty)$ and for every $t >0$ (cf.~\cite[Theorem 2.3]{Martini2}). 
	Therefore, arguing as in~{\bf1} we see that
	\[
	\begin{split}
	\abs{X_0\partial_s^k h_{0,s,t}(x)}&	\meg \frac{e^{ -2^{\frac{\delta}{1-\delta}}\widetilde b \abs{\pi_s(t^{-\sfrac{1}{\delta}} \cdot x)}_{s,*}^{\frac{\delta}{\delta-1}}}   }{t^{\frac{Q_0+d-k}{\delta}}} C_1 \abs{t^{-\sfrac{1}{\delta}}\cdot x  }^{\dd_\alpha-d} \sum_{\sum_{\ell=1}^k \ell\abs{\gamma_\ell}=k}  \int_{\ifr_0} e^{-2^{\frac{\delta}{1-\delta}}\widetilde b \abs{t^{-\sfrac{1}{\delta}}\cdot x+y+\psi_{0,t^{\sfrac{1}{\delta}}s}( y)}^{ \frac{\delta}{\delta-1}}} \times\\
	&\qquad \times \abs{y}^{\dd_{\gamma}+k}(1+\abs{y})^{n\abs{\gamma}-\dd_\gamma-k}\,\dd y
	\end{split}
	\]
	for every $x\in \hf_0$, for every $s\in [0,\infty)$, and for every $t>0$. 
	
	Now, observe that there is a constant $C_2\Meg 1$ such that
	\[
	\frac{1}{C_2}\min(\norm{z},\norm{z}^{\sfrac{1}{n}})\meg \abs{z}\meg C_2\max(\norm{z},\norm{z}^{\sfrac{1}{n}})
	\]
	for every $z\in \widetilde \gf$. 
	In addition, observe that the linear mapping $L_{s'}\colon x+y\mapsto x+y+\psi_{0,s'}(y)$ is an automorphism of $\widetilde \gf$ for every $s'\in [0,\infty)$, and that the mapping $[0,\infty)\ni s' \mapsto L_{s'}\in \Lc(\widetilde \gf)$ is continuous. 
	Therefore, there is a constant $C_3>0$ such that $\norm{L_{s'}^{-1}}\meg C_3$ for every $s'\in [0,c^{\sfrac{1}{\delta}}]$. In particular,  assuming that  $\norm{x+y}=\norm{x}+\norm{y}$ for every $(x,y)\in \hf_0\oplus \ifr_0$ for simplicity,
	\[
	\begin{split}
	\frac{2}{C_2 C_3^{\sfrac{1}{n}}}+\abs{t^{-\sfrac{1}{\delta}}\cdot x+y+\psi_{0,t^{\sfrac{1}{\delta}}s}(y)}&\Meg \frac{1}{C_2} \left(C_3^{-\sfrac{1}{n}}  +\norm{t^{-\sfrac{1}{\delta}}\cdot x+y+\psi_{0,t^{\sfrac{1}{\delta}}s}( y)}^{\sfrac{1}{n}}\right)\\
	&\Meg  \frac{1}{2 C_2 C_3^{\sfrac{1}{n}}}\left( 2+ \norm{t^{-\sfrac{1}{\delta}}\cdot x}^{\sfrac{1}{n}}+\norm{y}^{\sfrac{1}{n}}\right) \\
	&\Meg \frac{1}{2 C_2^{\frac{n+1}{n}  } C_3^{\sfrac{1}{n}}} \left( \abs{t^{-\sfrac{1}{\delta}}\cdot x}^{\sfrac{1}{n}}+\abs{y}^{\sfrac{1}{n}} \right)
	\end{split}
	\]	
	for every $(x,y)\in \hf_0\oplus \ifr_0$, for every $s\in [0,\infty)$, and for every $t\in (0,c s^{-\delta}]$.
	Hence, there is a constant $C_4>0$ such that 
	\[
	\abs{t^{-\sfrac{1}{\delta}}\cdot x  }^{\dd_\alpha-d} \sum_{\sum_{\ell=1}^k \ell\abs{\gamma_\ell}=k}  \int_{\ifr_0} e^{-2^{\frac{\delta}{1-\delta}}\widetilde b \abs{t^{-\sfrac{1}{\delta}}\cdot x+y+t\psi_{0,t^{\sfrac{1}{\delta}}s}( y)}^{ \frac{\delta}{\delta-1}}}  \abs{y}^{\dd_{\gamma}+k}(1+\abs{y})^{n\abs{\gamma}-\dd_\gamma-k}\,\dd y\meg C_4
	\]
	for every  $(x,y)\in \hf_0\oplus \ifr_0$, for every $s\in [0,1]$, and for every $t\in (0,c s^{-\delta}]$, so that
	\[
	\abs{X_0\partial_s^k h_{0,s,t}(x)}
	\meg C_1 C_4\frac{e^{ -2^{\frac{\delta}{1-\delta}}\widetilde b \abs{\pi_s(t^{-\sfrac{1}{\delta}} \cdot x)}_{s,*}^{\frac{\delta}{\delta-1}}}   }{t^{\frac{Q_0+d-k}{\delta}}} 
	\]
	$(x,y)\in \hf_0\oplus \ifr_0$, for every $s\in [0,\infty)$, and for every $t\in (0,c s^{-\delta}]$.
	The proof is complete.
\end{proof}

\begin{prop}\label{prop:5}
	For every $c>0$ and for every $\gamma$ there are $C>0$ and $\omega>0$ such that for every $s\in [0,\infty]$ and for every $t>0$,
	\[
	\norm*{ \vect{X}^\gamma_s h_{s,t} e^{c\abs{\,\cdot\,}_{s,*}} }_1\meg \frac{C}{t^{\frac{\dd_\gamma}{\delta}}} e^{\omega t}.
	\]
\end{prop}

\begin{proof}
	Observe that~\cite[Theorem 2.3 (f)]{Martini2} implies that there are $C$ and $\omega$ such that
	\[
	\norm*{ \widetilde{\vect{X}}^\gamma \widetilde h_{t} e^{c\abs{\,\cdot\,}_{*}} }_1\meg \frac{C}{t^{\frac{\dd_\gamma}{\delta}}} e^{\omega t}
	\]
	for every $t>0$. Therefore,
	\[
	\begin{split}
	\norm*{ \vect{X}^\gamma_s h_{s,t} \,e^{c\abs{\,\cdot\,}_{s,*}} }_1&=\norm*{(\pi_s)_*\left( \widetilde{ \vect{X}}^\gamma \widetilde h_{t}\, e^{c\left( \abs{\,\cdot\,}_{s,*} \circ \pi_s\right)  } \right)}_1\meg \norm*{\widetilde{ \vect{X}}^\gamma \widetilde h_{t}\, e^{c\abs{\,\cdot\,}_{*}}  }_1\meg \frac{C}{t^{\frac{\dd_\gamma}{\delta}}} e^{\omega t}
	\end{split}
	\]
	for every $t>0$ and for every $s\in [0,\infty]$.
\end{proof}

\section{Riesz Potentials}\label{sec:3}

We keep the notation of the preceding section.
Here we generalize the asymptotic study of the fundamental solutions made in~\cite{NagelRicciStein} to the complex powers of $\Lc_s$. 
Notice first that, while the convolution kernels of $\Lc_s^{-\frac{\alpha}{\delta}}$ (the Riesz potentials) are easily defined when $\Re \alpha<Q_\infty$, in order to define them also for $\Re \alpha \Meg Q_\infty$ we shall need to argue by analytic continuation.

\subsection{Definition and (Log-)Homogeneity of Riesz Potentials}

In the following statement, functions on $\hf_0$ (resp.\ $\hf_\infty$) are identified with distributions by means of the \emph{fixed} Lebesgue measure on $\hf_0$ (resp.\ $\hf_\infty$).

\begin{prop}\label{prop:21:8}
	For every $s\in (0,\infty]$ there is a unique meromorphic $\Sc'(G_s)$-valued mapping $\alpha \mapsto I_{s,\alpha}$ on $\C$, with poles of order at most $1$ at the elements of $Q_\infty+\N$, such that the following hold:
	\begin{enumerate}
		\item if $\alpha\in \C$, $-\delta k_1<\Re \alpha< Q_\infty+k_2$ for some $k_1,k_2\in \N$, and $\alpha \not \in  Q_\infty+\N$,  then
		\[
		\begin{split}
		I_{\infty,s,\alpha}&= \frac{1}{\Gamma(\frac{\alpha}{\delta})} \int_0^1 t^{\frac{\alpha}{\delta}} \left( h_{\infty,s,t} - \sum_{j<k_1} (-\Lc_s)^j \delta_0 \frac{t^j}{j!} \right)\,\frac{\dd t}{t}+ \sum_{j<k_1} \frac{1}{j!(\frac{\alpha}{\delta}+j)\Gamma(\frac{\alpha}{\delta})}(-\Lc_s)^j\delta_0\\
		&\qquad+ \frac{1}{\Gamma(\frac{\alpha}{\delta})} \int_1^{+\infty} t^{\frac{\alpha}{\delta}}\left( h_{\infty,s,t}- \sum_{\dd_\gamma<k_2} \partial_\infty^\gamma h_{\infty,s,t}(0) \frac{(\,\cdot\,)^\gamma}{\gamma!}  \right)\,\frac{\dd t}{t}+P_{s,\alpha,k_2},
		\end{split}
		\]
		where $	P_{s,\alpha,k_2}$ is a sum of homogeneous polynomials on $\hf_\infty$ of degree at most $k_2-1$;
		
		\item $I_{\infty,s,\alpha}\in L^1_\loc(\hf_\infty)$ when $\Re \alpha >0$;
		
		\item the restriction of $I_{\infty,s,\alpha} $ to $\hf_\infty \setminus \Set{0}$ has a density of class $C^\infty$;
		
		\item $I_{\infty,s,- \delta k}= \Lc_s^{k} \delta_0$ for every $k\in \N$.
	\end{enumerate}

	Similar assertions hold for $s=0$, replacing $\hf_\infty$ with $\hf_0$, $\partial_\infty$ with $\partial_0$, and $Q_\infty$ with $Q_0$.  
\end{prop}

\begin{proof}
	Fix $s\in (0,\infty]$. In addition, fix $k_1,k_2\in \N$ and observe that, if $0<\Re \alpha<\frac{Q_\infty}{\delta}$, then
	\[
	\begin{split}
	I_{\infty,s,\alpha}&= \frac{1}{\Gamma\left(\frac{\alpha}{\delta}\right)}\int_0^\infty t^{\frac{\alpha}{\delta}} h_{\infty,s,t}\,\frac{\dd t}{t}\\
		&=  \frac{1}{\Gamma(\frac{\alpha}{\delta})} \int_0^1 t^{\frac{\alpha}{\delta}} \left( h_{\infty,s,t} - \sum_{j<k_1} (-\Lc_s)^j\delta_0 \frac{t^j}{j!} \right)\,\frac{\dd t}{t}+ \sum_{j<k_1} \frac{1}{j!(\frac{\alpha}{\delta}+j)\Gamma(\frac{\alpha}{\delta})}(-\Lc_s)^j\delta_0\\
	&\qquad+ \frac{1}{\Gamma(\frac{\alpha}{\delta})} \int_1^{+\infty} t^{\frac{\alpha}{\delta}}\left( h_{\infty,s,t}- \sum_{\dd_\gamma<k_2} \partial_\infty^\gamma h_{\infty,s,t}(0) \frac{(\,\cdot\,)^\gamma}{\gamma!}  \right)\,\frac{\dd t}{t}+P_{s,\alpha,k_2},
	\end{split}
	\]
	where
	\[
	P_{s,\alpha,k_2}(x)\coloneqq \frac{1}{\Gamma(\frac{\alpha}{\delta})} \sum_{\dd_\gamma<k_2}\frac{x^\gamma}{\gamma!} \int_1^{+\infty} t^{\frac{\alpha}{\delta}} \partial_\infty^\gamma h_{\infty,s,t}(0) \,\frac{\dd t}{t}
	\]
	for every $x\in \hf_\infty$. 
	Taking into account Lemmas~\ref{lem:21:5},~\ref{lem:21:4}, and~\ref{lem:21:3}, it will suffice to prove that the mapping $\alpha \mapsto P_{s,\alpha,k_2}$ extends to a meromorphic mapping on $\C$ with poles of order at most $1$ at the elements of $Q_\infty+\N$. 
	Indeed, 
	\[
	\begin{split}
	\partial^\gamma_{\infty} h_{\infty,s,t^{-\delta}}(0)&= \int_{\ifr_\infty} \partial_\infty^\gamma\widetilde h_{t^{-\delta}}\left(y+\psi_{\infty,s}(y)\right)\,\dd y=t^{Q_\infty+\dd_\gamma}\int_{\ifr_\infty}  \partial_\infty^\gamma\widetilde h_1\left(y+t\cdot \psi_{\infty,s}(t^{-1}\cdot y)\right)\,\dd y
	\end{split}
	\]
	for every $x\in G$ and for every $t>0$. In addition, since $\psi_{\infty,s}$ is linear and strictly super-homogeneous,
	\[
	t\cdot \psi_{\infty,s}(t^{-1}\cdot y)= \sum_{\ell<j} t^{j-\ell} \pr_j(\psi_{\infty,s}(y_\ell))
	\]
	for every $t>0$ and $y\in \ifr_\infty$. 
	As a consequence, the mapping $t\mapsto\partial^\gamma_{\infty} h_{\infty,s,t^{-\delta}}(0)$ extends to a mapping of class $C^\infty$ on $\R$. 
	Let $\sum_{j\Meg Q_\infty+\dd_\gamma} b_{s,\gamma,j} t^j$ be its Taylor development at the origin. 
	
	Now, fix $N\in \N$ and observe that, for $\Re \alpha<N+1$,
	\[
	\begin{split}
	\int_1^{+\infty}t^{\frac{\alpha}{\delta}}\partial^\gamma_{\infty} h_{\infty,s,t}(0)\,\frac{\dd t}{t}&=\delta\int_{0}^1 t^{- \alpha} \partial^\gamma_{\infty} h_{\infty, t^{-\delta}}(0)\,\frac{\dd t}{t} \\
	&=\delta\sum_{Q_\infty+\dd_\gamma\meg j\meg N} b_{s,\gamma,j}\int_0^1 t^{-\alpha+j}\,\frac{\dd t}{t}+\int_0^1 t^{-  \alpha}O\left(t^{N+1} \right)\,\frac{\dd t}{t}\\
	&=\delta\sum_{Q_\infty+\dd_\gamma\meg j\meg N} \frac{b_{s,\gamma,j}}{j- \alpha}+\int_0^1 O\left( t^{-\Re \alpha+N} \right)\,\dd t.
	\end{split}
	\]
	By the arbitrariness of $N$, it follows that the mapping $\alpha \mapsto \int_1^{+\infty}t^{\frac{\alpha}{\delta}}\partial^\gamma_{\hf_\infty} h_{\infty,s,t}(0)\,\frac{\dd t}{t}$ extends to a meromorphic mapping on $\C$ with poles of order at most $1$ at every element of $Q_\infty+\dd_\gamma+\N$.
	Summing up all these facts, it follows that the mapping $\alpha\mapsto I_{\infty,s,\alpha}$ extends to a meromorphic mapping on $\C$, with poles of order at most $1$ at every element of $Q_\infty+\N$. Finally, it is clear that $I_{\infty,s,-k }=\Lc_s^{k}\delta_0$.
	
	The case $s=0$ is treated similarly.
\end{proof}

\begin{deff}
	Fix $s\in (0,\infty]$. For every $\alpha\in \C$ such that $\alpha\not \in Q_\infty+\N$, we define $I_{s,\alpha}$ as the distribution on $G_s$ induced by the distribution $I_{\infty,s,\alpha}$ of Proposition~\ref{prop:21:8} under the identification of $G_s$ with~$\hf_\infty$ (in other words, $I_{s,\alpha}=(\pi_s)_*(I_{\infty,s,\alpha})$ by an abuse of notation).
	We define $I_{s,\alpha}$, for $\alpha\in  Q_\infty+\N$, as the $0$-th order term of the Laurent expansion of the mapping $\alpha'\mapsto I_{s,\alpha'}$ at $\alpha$.\footnote{Notice that the mapping $\alpha'\mapsto I_{s,\alpha'}$ \emph{may} be regular at $\alpha$.}
	
	We denote by $I_{0,s,\alpha}$ the distribution on $\hf_0$ induced by $I_{s,\alpha}$ under the identification of $G_s$ with $\hf_0$, for every $s\in (0,\infty)$; $I_{0,0,\alpha}$ is defined as in Proposition~\ref{prop:21:8}. We define $I_{0,\alpha}\coloneqq (\pi_0)_*(I_{0,0,\alpha})$, with the same abuse of notation used above.
\end{deff}

\begin{prop}\label{prop:9}
	For every $s\in (0,\infty]$, for every $r>0$, and for every $\alpha\in \C$, the following hold:
	\begin{enumerate}
		\item $(r\,\cdot\,)_* I_{s,\alpha}=r^{-\alpha} I_{r^{-1}s,\alpha}$ if $I_{s,\,\cdot\,}$ is regular at $\alpha$ (in which case also $I_{r^{-1}s,\,\cdot\,}$ is regular at $\alpha$);
		
		\item $(r\,\cdot\,)_* I_{s,\alpha}=r^{-\alpha} I_{r^{-1}s,\alpha}+r^{-\alpha} \log r P_{r^{-1}s, \alpha}$, where $P_{r^{-1}s,\alpha}$ is a polynomial such that $P_{r^{-1}s,\alpha}(x)=O(\abs{x}_{r^{-1}s}^{-Q_\infty+\alpha})$ for $x\to \infty$, if $I_{s,\,\cdot\,}$ has a pole at $\alpha$ (in which case also $I_{r^{-1}s,\,\cdot\,}$ has a pole at $\alpha$). In addition, $P_{\infty,\alpha}$ is homogeneous of degree $-Q_\infty+\alpha$.
	\end{enumerate}
	Analogous assertions hold for $s=0$, replacing $\hf_\infty$ with $\hf_0$ and $Q_\infty$ with $Q_0$.
\end{prop}

In particular, $I_{\infty,\alpha}$ is homogeneous of degree $-Q_\infty+\alpha$ when $I_{\infty,\,\cdot\,}$ regular at $\alpha$, while $I_{\infty,\alpha}$ is log-homogeneous of degree $-Q_\infty+\alpha$ otherwise (cf.~Definition~\ref{def:2}).

Analogous statements hold for $I_{0,\alpha}$, with the obvious modifications.

\begin{proof}
	The first assertion for $0<\Re \alpha<Q_\infty$ follows easily from the equality $(r\,\cdot\,)_* h_{s,t}=h_{r^{-1}s, r^\delta t}$, which holds for every $r>0$, for every $s\in [0,\infty]$, and for every $t>0$. The general statement then holds by holomorphy.

	For what concerns the second assertion, take $s\in (0,\infty]$ and a pole $\alpha$ of $I_{s,\,\cdot\,}$, so that, in particular,  $\alpha\in Q_\infty+\N$. Then, for every $\alpha'\neq \alpha$ in a neighbourhood of $\alpha$, and for every $r>0$
	\[
	(r\,\cdot\,)_* I_{s,\alpha'}= r^{-\alpha'}I_{r^{-1} s, \alpha'},
	\]
	so that, taking the $0$-th order term of the Laurent expansions of both sides of the equality at $\alpha$,
	\[
	(r\,\cdot\,)_*I_{s,\alpha}= r^{-\alpha}I_{r^{-1}s, \alpha}- r^{-\alpha} \log r \lim_{\alpha'\to \alpha} (\alpha'-\alpha) I_{r^{-1}s,\alpha'}.
	\]
	Now, with the notation of Proposition~\ref{prop:21:8}, it is easily seen that, chosen $k_1=0$ and $ k_2=-Q_\infty+\alpha+1$, 
	\[
	\lim_{\alpha'\to \alpha} (\alpha'-\alpha) I_{\infty,r^{-1}s,\alpha}(x)=\lim_{\alpha'\to \alpha} (\alpha'-\alpha) P_{r^{-1}s,\alpha,k_2}(x).
	\]
	By inspection of the proof of  Proposition~\ref{prop:21:8}, it is easily seen that $\lim_{\alpha'\to \alpha} (\alpha'-\alpha) P_{\infty,\alpha',k_2}$ is a homogeneous polynomial of degree $-Q_\infty+\alpha$, whence the result.
\end{proof}

\subsection{Asymptotic Developments}

We keep the notation of the preceding sections.
We prove some asymptotic developments of the $I_{\alpha}$ generalizing those proved in~\cite{NagelRicciStein} for the fundamental solutions. Even though the procedure of~\cite{NagelRicciStein} may be generalized to the present setting, we prefer to   give a different proof, which is shorter and gives a little more insight into the meaning of the further terms of the development.
We then present, under rather restrictive assumptions, another proof which describes quite explicitly the terms of the development.

\begin{teo}\label{teo:21:1}
	Take  and $\alpha\in \C$ and $s\in \R_+$. Then, the following hold:
	\begin{enumerate}		
		\item there is a sequence of log-homogeneous functions $(I_{\infty,\alpha}^{(k)})$ of class $C^\infty$ on $\hf_\infty\setminus \Set{0}$ such that  $I_{\infty,\alpha}^{(0)}=I_{\infty,\infty,\alpha}$, such that $ I_{\infty,\alpha}^{(k)}$ has degree $-Q_\infty+\alpha-k$, and such that for every $N\in \N$ and for every $\gamma$ there is a constant $C_{N,\gamma}>0$ such that, for every $s\in [1,\infty)$,
		\[
		\abs*{\partial_{\infty}^\gamma \left( I_{\infty,s,\alpha}-  \sum_{k<N} s^{-k} I_{\infty,\alpha}^{(k)}\right) (x)}\meg \frac{C_{N,\gamma} s^{-N}}{\abs{x}^{Q_\infty-\Re\alpha+\dd_\gamma+N}}(1+\abs{\log\abs{s\cdot x}})
		\]
		for every $x\in \hf_\infty$ such that $\abs{x}\Meg s^{-1}$; the factor $1+\abs{\log\abs{s\cdot x}}$ may be omitted if $\alpha\not \in Q_\infty+\dd_\gamma+N+\N$;
		
		\item there are a sequence $(P_{\alpha,k})$  of homogeneous polynomials on $\hf_0$ and a sequence $(I_{0,\alpha}^{(k)})$ of log-homogeneous functions of class $C^\infty$ on $\hf_0\setminus \Set{0}$ such that  $I_{0,\alpha}^{(0)}=I_{0,0,\alpha}$, such that $ I_{0,\alpha}^{(k)}$ has degree $-Q_0+\alpha+k$, such that $P_{0,\alpha,k}$ has degree $k$, and such that for every $N\in \N$ and for every $\gamma$ there is a constant $C'_{N,\gamma}>0$ such that, for every $s\in (0,1]$,
		\[
		\abs*{\partial_{0}^\gamma \left( I_{0,s,\alpha}-  \sum_{k<N} s^{k} I_{0,\alpha}^{(k)}- \sum_{k<-Q_0+\Re\alpha+N} s^{k-Q_0+\alpha} P_{\alpha, k}\right) (x)}\meg \frac{C'_{N,\gamma} s^{N}}{\abs{x}^{Q_0-\Re\alpha+\dd_\gamma-N}}(1+\abs{\log\abs{s\cdot x}})
		\]
		for every $x\in \hf_0$ such that $0\neq\abs{x}\meg s^{-1}$; the factor $1+\abs{\log\abs{s\cdot x}}$ may be omitted if $\alpha\not \in Q_0+\dd_\gamma-N+\N$.
	\end{enumerate}
\end{teo}

\begin{proof}
	{\bf1.} Define $H_\infty(s',t,x)\coloneqq h_{\infty,\sfrac{1}{s'},t}(x)$ for every $s'\in [0,\infty)$, for every $t>0$, and for every $x\in \hf_\infty$, to simplify the notation.
	Take $\alpha\in \C$ such that $0<\Re \alpha<Q_\infty$, and observe that a Taylor expansion of $H_\infty$ in the first variable gives
	\[
	 I_{\infty,s,\alpha}= \frac{1}{\Gamma(\frac{\alpha}{\delta})} \int_0^1 t^{\frac{\alpha}{\delta}} h_{\infty,s,t}\,\frac{\dd t}{t}+ \sum_{k<N} \frac{s^{-k}}{k!\Gamma(\frac{\alpha}{\delta}) }\int_1^{+\infty} \partial_1^k H_\infty(0,t,\,\cdot\,) \,\frac{\dd t}{t}+s^{-N} R_{s,\alpha,N},  
	\]
	where
	\[
	R_{s,\alpha,N}=\frac{1}{(N-1)! \Gamma(\frac{\alpha}{\delta})} \int_1^{+\infty} t^{\frac{\alpha}{\delta}}  \int_0^1 \partial_{1}^N H_\infty\left(\frac{\theta}{s},t,\,\cdot\, \right)(1-\theta)^{N-1}\,\dd \theta\, \frac{\dd t}{t}  .
	\]
	Now, Lemma~\ref{lem:21:4} implies that the mapping $\alpha \mapsto \frac{1}{\Gamma(\frac{\alpha}{\delta})} \int_0^1 t^{\frac{\alpha}{\delta}}  h_{\infty,s,t}\,\frac{\dd t}{t}$ extends to an entire function with values in $\Ec'(\hf_\infty)+\Sc(\hf_\infty)$. Next, observe that 
	\[
	H_\infty(s',t,x)=t^{-\frac{Q_\infty}{\delta}} H_\infty(t^{-\sfrac{1}{\delta}}s',1,t^{-\sfrac{1}{\delta}}\cdot x  )
	\]
	for every $s'\in [0,\infty)$, for every $t>0$, and for every $x\in \hf_\infty$, so that
	\[
	\partial_1^k H_\infty(0,t,x)= t^{-\frac{Q_\infty+k}{\delta}} \partial_1^k H_\infty(0,1,t^{-\sfrac{1}{\delta}}\cdot x)
	\]
	for every $x\in \hf_\infty$ and for every $t>0$.

	Therefore, Lemma~\ref{lem:21:4} and  the estimates of $\partial_1^k H_\infty(0,t,\,\cdot\,)$ provided in Theorem~\ref{prop:21:2} show that, for $0< \Re \alpha< Q_\infty+k$,
	\[
	\frac{1}{k!\Gamma(\frac{\alpha}{\delta}) }\int_1^{+\infty} t^{\frac{\alpha}{\delta}} \partial_1^k H_\infty(0,t,\,\cdot\,) \,\frac{\dd t}{t}=\frac{1}{k!\Gamma(\frac{\alpha}{\delta}) }\int_0^{+\infty}t^{\frac{\alpha}{\delta}} \partial_1^k H_\infty(0,t,\,\cdot\,)\,\frac{\dd t}{t}+R'_{\alpha,k},
	\]
	where $R'_{\alpha,k}$ is an entire function of $\alpha$ with values in $\Ec'(\hf_\infty)+\Sc(\hf_\infty)$. 
	In addition, we also see that the mapping, initially defined for $0< \Re \alpha< Q_\infty+k$,
	\[
	\alpha \mapsto I_{\infty,\alpha}^{(k)}\coloneqq \frac{1}{k !\Gamma(\frac{\alpha}{\delta}) }\int_0^{+\infty} t^{\frac{\alpha}{\delta}} \partial_1^k H_\infty(0,t,\,\cdot\,) \,\frac{\dd t}{t}
	\]
	extends to a meromorphic function on $\C$ such that $I_{\infty,\alpha}^{(k)}$ is homogeneous of degree $-Q_\infty+\alpha-k$ for every $\alpha$ in the domain of holomorphy of $I_{\infty,\,\cdot\,}^{(k)}$ (argue as in the proof of Propositions~\ref{prop:21:8} and~\ref{prop:9}). 
	Log-homogeneity holds at poles, where $I_{\infty,\alpha}^{(k)}$ denotes the $0$-th order term of the Laurent expansion of $I_{\infty,\,\cdot\,}^{(k)}$ (argue as in the proof of Proposition~\ref{prop:9}).	
	
	Finally, assume that $\Re \alpha< Q_\infty+ N$. Observe that there is a constant $C>0$ such that
	\[
	\abs{\pi_{s'}(x)}_{s'}\Meg C (\abs{x}-1)
	\]
	for every $s'\in [1,\infty]$ and for every $x\in \hf_\infty$, thanks to Proposition~\ref{prop:7}.
	Therefore, Theorem~\ref{prop:21:2} and the preceding computations imply that $R_{s,\alpha,N}$ is well-defined  for $x\neq 0$  and that there are there are two constants $C'>0$ and  $b>0$ such that, for every $\gamma$,
	\[
	\abs{\partial_\infty^\gamma R_{s,\alpha,N}(x)}\meg C' \int_1^{+\infty} t^{\frac{\Re \alpha-Q_\infty-\dd_\gamma-N}{\delta} } e^{-b\left( \frac{\abs{x}}{t^{\sfrac{1}{\delta}}}  \right)^{\frac{\delta}{\delta-1}}}\,\frac{\dd t}{t}.
	\]
	Hence,
	\[
	\abs{\partial_\infty^\gamma R_{s,\alpha,N}(x)}\meg C' \abs{x}^{ \Re \alpha-Q_\infty-\dd_\gamma-N  }\int_0^{+\infty} t^{\frac{-\Re \alpha+Q_\infty+\dd_\gamma+N}{\delta} }  e^{ -b t^{\frac{1}{\delta-1}} }\,\frac{\dd t}{t}  .
	\]
	The assertion follows for $s$ fixed. In order to get uniform estimates for $s\in [1,\infty)$, reduce to the case $s=1$ by means of Proposition~\ref{prop:9}. 
	Let us give some more details in the case in which $I_{s,\,\cdot\,}$ has a pole at $\alpha$. Indeed, for every $s\in [1,\infty]$, $I_{\infty,s,\alpha}-\sum_{k<N} s^{-k} I_{\infty,\alpha}^{(k)}$ equals
	\[
	s^{-\alpha} (s^{-1}\,\cdot\,)_*\left(I_{\infty,1,\alpha}-\sum_{k<N}  I_{\infty,\alpha}^{(k)}\right)+s^{-\alpha} \log s \,(s^{-1}\,\cdot\,)_*\left( P_{1,\alpha}\circ \pi_1-\sum_{k<N}  P'_{\alpha,k}  \right),
	\]
	where $P_{1,\alpha}$ is defined in Proposition~\ref{prop:9}, while $P'_{\alpha,k}$ is a suitable homogeneous polynomial on $\hf_\infty$ of degree $-Q_\infty+\alpha-k$. 
	Since the term $s^{-\alpha} (s^{-1}\,\cdot\,)_*\left( I_{\infty,1,\alpha}-\sum_{k<N}  I_{\infty,\alpha}^{(k)}\right)$ satisfies the estimates of the statement, all we need to prove is that $ P_{1,\alpha}\circ \pi_1-\sum_{k<N}  P'_{\alpha,k}$ has degree at most $-Q_\infty+\alpha-N$. 
	One may prove this by expressing $P_{1,\alpha}$ and the $P'_{\alpha,k}$ in terms of $h_{s,t}$ and its derivatives in $s^{-1}$. 
	Nonetheless, since the above proof shows that $ I_{\infty,s,\alpha}-\sum_{k<N} s^{-k} I_{\infty,\alpha}^{(k)}$ satisfies the estimates of the statement (with constants depending on $s$), the same necessarily applies to $ s^{-\alpha} \log s (s^{-1})_* \left( P_{1,\alpha}\circ \pi_1-\sum_{k<N}  P'_{\alpha,k}  \right)$, whence our claim.
	
	{\bf2.} Define $H_0(s,t,x)\coloneqq h_{0,s,t}(x)$ for every $s\in [0,\infty)$, for every $t>0$, and for every $x\in \hf_0$, to simplify the notation.
	Take $\alpha\in \C$ such that $0<\Re \alpha<Q_0$, and observe that a Taylor expansion of $H_0$ in the first variable gives
	\[
	I_{0,s,\alpha}=  \sum_{k<N} s^k \frac{1}{k!\Gamma(\frac{\alpha}{\delta}) }\int_0^1  t^{\frac{\alpha}{\delta}} \partial_1^k H_0(0,t,\,\cdot\,)\,\frac{\dd t}{t}+ s^N R_{s,\alpha,N}+\frac{1}{\Gamma(\frac{\alpha}{\delta})} \int_1^{+\infty} t^{\frac{\alpha}{\delta}} h_{0,s,t}\,\frac{\dd t}{t},  
	\]
	where
	\[
	R_{s,\alpha,N}=\frac{1}{\Gamma(\frac{\alpha}{\delta})(N-1)! } \int_0^1 t^{\frac{\alpha}{\delta}}  \int_0^1 \partial_1^N H_0(\theta s,t,\,\cdot\,) (1-\theta)^{N-1}\,\dd \theta\, \frac{\dd t}{t}  .
	\]
	Now, by means of Lemma~\ref{lem:21:3} we see that the mapping $\alpha \mapsto \frac{1}{\Gamma(\frac{\alpha}{\delta})} \int_1^{+\infty} t^{\frac{\alpha}{\delta}} h_{0,s,t}\,\frac{\dd t}{t} $ extends to a meromorphic function on $\C$ with values in $\Ec(G)$.
	In addition, as in~{\bf1} one may prove that
	\[
	\partial_1^k H_0(0,t,x)=t^{-\frac{Q_0-k}{\delta}}\partial_1^k H_0(0,1,t^{-\sfrac{1}{\delta}}\cdot x)
	\]
	for every $x\in \hf_0$ and for every $t>0$. Therefore, making use of the estimates of $\partial_1^k H_0(0,t,\,\cdot\,)$ provided in Theorem~\ref{prop:21:2}, we see that
	\[
	R'_{s,\alpha,k}\coloneqq \frac{1}{k!\Gamma(\frac{\alpha}{\delta})}\int_1^{+\infty} t^{\frac{\alpha}{\delta}}\partial_1^k H_0(0,t,\,\cdot\,)\frac{\dd t}{t}\in C^\infty(G_s),
	\]
	initially defined for $\Re \alpha<Q_0-k$, extends to a meromorphic function on $\C$. In addition, we also see that the mapping, initially defined for $-k<\Re \alpha<Q_0-k$,
	\[
	\alpha \mapsto I_{0,\alpha}^{(k)}\coloneqq\frac{1}{k!\Gamma(\frac{\alpha}{\delta}) }\int_0^1  t^{\frac{\alpha}{\delta}} \partial_1^k H_0(0,t,\,\cdot\,)\,\frac{\dd t}{t}+R'_{s,\alpha,k}
	\]
	extends to a meromorphic mapping on $\C$.\footnote{Using the estimates of $H_0$ provided in Theorem~\ref{prop:21:2}, it is not hard to see that $\int_0^1  t^{\frac{\alpha}{\delta}} \partial_1^k H_0(0,t,\,\cdot\,)\,\frac{\dd t}{t}\in L^1(G_s) $ for $\Re \alpha>-k$.} Let us prove that $I_{0,\alpha}^{(k)}$ is homogeneous of degree $-Q_0+ \alpha+k$ for $\alpha$ in the domain of holomorphy of $I_{0,\alpha}^{(k)}$. By analyticity, we may reduce to prove this fact for  $-k<\Re \alpha<Q_0-k$, in which case
	\[
	\begin{split}
	I_{0,\alpha}^{(k)}&=\frac{1}{k!\Gamma(\frac{\alpha}{\delta}) }\int_0^\infty  t^{\frac{\alpha}{\delta}}  \partial_1^k H_0(0,t,\,\cdot\,) \,\frac{\dd t}{t},
	\end{split}
	\]
	so that the assertion is easily established. Log-homogeneity holds at the poles of $I_{0,\,\cdot\,}^{(k)}$, where $I_{0,\alpha}^{(k)}$ denote the $0$-th order term of the Laurent expansion of $I_{0,\,\cdot\,}^{(k)}$ at $\alpha$.
	
	Finally, take $\gamma$ and assume that $\Re \alpha>Q_0+\dd_\gamma-N$; in addition, define $\abs{x}'\coloneqq \inf_{s\in [0,1]} \abs{\pi_s(x)}_s$ for every $x\in \hf_0$, and observe that $\abs{x}'>0$ for every non-zero $x\in \hf_0$, and that there is a constant $C>0$ such that
	\[
	\frac{1}{C}\abs{x}\meg \abs{x}'\meg C \abs{x}
	\]
	for every $x\in \hf_0$ such that $\abs{x}\meg 1$ (cf.~Proposition~\ref{prop:7}).
	In addition, Theorem~\ref{prop:21:2} and the preceding computations imply that there are there are two constants $C'>0$ and  $b>0$ such that
	\[
	\abs{\partial_0^\gamma R_{s,\alpha,N}(x)}\meg C'\int_0^{1} t^{\frac{\Re \alpha-Q_0-\dd_\gamma+ N}{\delta} } e^{-b\left(\frac{\abs{x}'}{t^{\sfrac{1}{\delta}}}\right) ^{\frac{\delta}{\delta-1}}}\,\frac{\dd t}{t},  
	\]
	so that $R_{s,\alpha,N}(x)$ is well defined  for $x\neq 0$. In addition,
	\[
	\abs{\partial_0^\gamma  R_{\alpha,N}(x)}\meg C' \abs{x}'^{  \Re \alpha-Q_0-\dd_\gamma+N  }\int_0^{+\infty} t^{\frac{Q_0+\dd_\gamma-N-\Re \alpha}{\delta}} e^{b t^{\frac{1}{\delta-1}} }\,\frac{\dd t}{t}  .
	\]
	The assertion follows for $s$ fixed. In order to get uniform estimates for $s\in (0,1]$, reduce to the case $s=1$ by means of Proposition~\ref{prop:9} (argue as in~{\bf1}).
\end{proof}

Observe that, with the same techniques used to prove~\cite[Theorem 2]{NagelRicciStein}, one may prove the following result. 

\begin{cor}\label{cor:4}
	Take $s\in (0,\infty)$, $\gamma\in \N^J$, and $\alpha\in \C$ such that $\Re \alpha\Meg \dd_\gamma$. 
	Then, for every $p,q\in (1,\infty)$ such that $\frac{\Re \alpha-\dd_\gamma}{Q_\infty}\meg \frac{1}{p}-\frac{1}{q}\meg \frac{\Re \alpha-\dd_\gamma}{Q_0}$, convolution on the right with $\vect{X}^\gamma_s I_{s,\alpha}$ induces a bounded operator from $L^p(G_s)$ into $L^q(G_s)$.
\end{cor}

Notice that, if convolution on the right with $\vect{X}^\gamma_s I_{s,\alpha}$  induces a bounded operator $T_s$ from $L^p(G_s)$ into $L^q(G_s)$ for some $p,q\in (1,\infty)$ and for some $s\in (0,\infty)$, then $\frac{\Re \alpha-\dd_\gamma}{Q_\infty}\meg \frac{1}{p}-\frac{1}{q}\meg \frac{\Re \alpha-\dd_\gamma}{Q_0}$. Indeed, take $r>0$ and $f\in L^p(G_{r^{-1}s})$, and define $\rho_r(x)\coloneqq r\cdot x$ for every $x\in G_{r^{-1}s}$. Then,
\[
T_s(f\circ \rho_r)=(f\circ \rho_r)* (\vect{X}^\gamma_s I_{s,\alpha})= [f*(\rho_r)_*(\vect{X}^\gamma_s I_{s,\alpha})  ]\circ \rho_r= r^{-\alpha+\dd_\gamma} (T_{r^{-1}s} f)\circ \rho_r,
\]
where $T_{r^{-1}s}$ is given by convolution on the right with $\vect{X}^\gamma_{r^{-1}s} I_{r^{-1}s,\alpha}$. 
Now, identify $G_{s'}$ with $\hf_0$ for every $ s'\in [0,s]$. Observe that, denoting by $\nu_{\hf_0}$ the fixed Lebesgue measure on $\hf_0$, we have $\nu_{G_{s'}}=a_{0,s'} \nu_{\hf_0}$ for some $a_{0,s'}>0$; in addition, the mapping $ s'\mapsto a_{0,s'}$ is continuous on $[0,s]$ thanks to~{\bf5} of Proposition~\ref{prop:7}. Therefore, there is a constant $C>0$ such that
\[
C r^{-\sfrac{Q_0}{p}}  \norm{f}_{L^p(\hf_0)}=C \norm{f\circ \rho_r}_{L^p(\hf_0)}\Meg r^{-\Re \alpha+\dd_\gamma}\norm{(T_{r^{-1}s} f)\circ \rho_r}_{L^q(\hf_0)}=r^{-\Re \alpha+\dd_\gamma-\sfrac{Q_0}{q}} \norm{T_{r^{-1}s} f}_{L^q(\hf_0)} .
\]
for every $f\in L^p(\hf_0)$ and for every $r\Meg 1$. Now, $T_{r^{-1}s}f$ converges pointwise to $ T_0 f$ as $r\to +\infty$ for every $f\in C^\infty_c(\hf_0)$ with vanishing moments of all orders,\footnote{Notice that the set of such $f$ is dense in $L^p(\hf_0)$, since $p\in (1,\infty)$.} so that
\[
 \norm{T_{0} f}_{L^q(\hf_0)}\meg C\norm{f}_{L^p(\hf_0)} \liminf_{r\to \infty} r^{\Re \alpha-\dd_\gamma+ Q_0\left( \frac{1}{q}-\frac{1}{p}  \right)}
\]
for every such $f$. Since $T_0\neq 0$, it follows that $\Re \alpha-\dd_\gamma\Meg Q_0\left( \frac{1}{p}-\frac{1}{q}  \right)$. The other inequality is proved similarly.

\begin{proof}
	We shall briefly indicate the procedure employed in~\cite{NagelRicciStein}, for the sake of completeness. 
	When $p\neq q$, observe that $\vect{X}^\gamma_s I_{s,\alpha}$ belongs to weak $L^r$ for every $r\in (1,\infty)$ such that $\frac{\Re \alpha-\dd_\gamma}{Q_\infty}\meg \frac{1}{r'}\meg \frac{\Re \alpha-\dd_\gamma}{Q_0}$ thanks to Theorem~\ref{teo:21:1}. Then, arguing  as in the proof of~\cite[Proposition 1.19]{FollandStein}, we see that weak $L^r$ convolves $L^p(G_s)$ into $L^q(G_s)$ for $\frac{1}{p}-\frac{1}{q}=\frac{1}{r'}$ and $p,q,r\in(1,\infty)$.
	
	When $p=q$, take $\tau\in C^\infty_c(G_s)$ so that $\tau$ equals $1$ in a neighbourhood of $e$. We shall prove that $\tau \vect{X}^\gamma_s I_{s,\alpha}$ convolves $L^p(G_s)$ into itself; one may prove analogously that also $(1-\tau) \vect{X}^\gamma_s I_{s,\alpha}$ convolves $L^p(G_s)$ into itself and conclude the proof.
	Now, Proposition~\ref{lem:21:1} and Theorem~\ref{teo:21:1} show that $\tau \vect{X}^\gamma_s I_{s,\alpha}$ equals $\tau \vect{X}^\gamma_0 I_{0,\alpha}$ up to an integrable function, under the identification of $G_s$ with $G_0$ through $\hf_0$; consequently, it will suffice to show that $\tau \vect{X}^\gamma_0 I_{0,\alpha}$ convolves $L^p(G_s)$ into itself (with respect to the convolution of $G_s$).
	Now, it is clear that there is a constant $C>0$ such that
	\[
	\abs{\vect{X}_0^{\gamma'}(\tau \vect{X}^\gamma_0 I_{0,\alpha})(x)}\meg C\abs{x}_s^{-Q_0-\dd_{\gamma'}}
	\]
	for every $x\in G_s$ and for every $\gamma'$ with length at most $1$, thanks to~{\bf6} of Proposition~\ref{prop:7}. By Lemma~\ref{lem:7} we then see that we may take $C$ in such a way that
	\[
	\abs{ (\tau \vect{X}^\gamma_0 I_{0,\alpha})(x y^{-1})-(\tau \vect{X}^\gamma_0 I_{0,\alpha})(x) }\meg \frac{\abs{y}_s}{\abs{x}_s^{Q_0+1}}
	\]
	for every $x,y\in G_s$ such that $\abs{x}_s>2 \abs{y}_s>0$.  In addition, since $\vect{X}^\gamma_0 I_{0,\alpha}$ is a homogeneous distribution of degree $-Q_0+\alpha -\dd_\gamma$, it is clear that $\vect{X}^\gamma_0 I_{0,\alpha}$ has zero mean on the unit sphere (relative to $\abs{\,\cdot\,}_0$) \emph{when $\Im \alpha=0$}. Similar remarks apply to $ (\vect{X}^\gamma_0 I_{0,\alpha})^*$.
	Taking into account~\cite[Lemma of Chapter III, § 3.1]{Goodman}, it is not hard to see that we may apply~\cite[Theorem of Chapter III, § 4.3]{Goodman}, so that  $\tau\vect{X}^\gamma_0 I_{0,\alpha}$ convolves $L^p(G_s)$ into itself for every $p\in (1,\infty)$. 
\end{proof}

\begin{oss}
	Observe that, if $G=\R^n$ and $\Lc= \Delta^2-\Delta$, then it is not hard to prove that $I_{1,\alpha}=I^\Delta_{\alpha} * J_{\alpha}$, where $J_\alpha=( (1-\Delta)^{-\frac{\alpha}{2}})\delta_0$ and $I^\Delta_\alpha$ is the kernel of $(-\Delta)^{-\frac{\alpha}{2}}$ defined by analytic continuation, for every $\alpha\in \C$. 
	Then, observe that, for $\alpha\in (0,\infty)\setminus (n+\N )$, $I^\Delta_\alpha$ and $J_\alpha$ keep a constant sign (in particular, they vanish nowhere), so that
	\[
	I_{1,\alpha}(0)=(J_\alpha*I_\alpha^\Delta)(0)= \int_{\R^n} J_\alpha(x) I^{\Delta}_\alpha(-x)\,\dd x\neq0
	\]
	when $\alpha>n$.
	Hence, the polynomials appearing in the local expansion of $I_\alpha$ in Theorem~\ref{teo:21:1} cannot be omitted, in general.
\end{oss}

\begin{teo}\label{teo:3}
	Take $\alpha\in \C$ and $s\in \R_+$. Then, the following hold:
	\begin{enumerate}
		\item assume that $G_1=G_\infty$ (under the identification through $\hf_\infty$) as Lie groups and that $[\Lc_1, \Lc_\infty]=0$. Let $d_\infty$ be the least degree of the non-zero homogeneous components of $\Lc_1-\Lc_\infty$. Then, for every $N\in \N$ and for every $\gamma$ there is a constant $C_{N,\gamma}>0$ such that, for every $s\in [1,\infty)$,
		\[
		\abs*{\vect{X}_\infty^\gamma\left(I_{s,\alpha}-\sum_{k<N}\binom{-\alpha/\delta}{k} (\Lc_s-\Lc_\infty)^k I_{\infty,\alpha+\delta k}  \right)(x)}\meg \frac{s^{-N} C_{N,\gamma}}{\abs{x}^{Q_\infty-\Re\alpha+\dd_\gamma+N(d_\infty-\delta)}}(1+\abs{\log\abs{s\cdot x}})
		\]
		for every $x\in \hf_\infty$ such that $\abs{x}\Meg s^{-1}$; the factor $1+\log\abs{s\cdot x}$ may be omitted if $\alpha\not \in Q_\infty+\dd_\gamma+N(d_\infty-\delta)+\N$;
		
		\item assume that $G_1=G_0$ (under the identification through $\hf_0$)  as Lie groups and that $[\Lc_1, \Lc_0]=0$. Let $d_0$ be the greatest degree of non-zero homogeneous components of $\Lc_1-\Lc_0$. Then, there is a sequence $(P_{\alpha,k})$ of homogeneous polynomials on $G_0$ such that $P_{\alpha,k}$ has degree $k$ for every $k\in\N$, and such that for every $N\in \N$ and for every $\gamma$ there is a constant $C_{N,\gamma}>0$ such that, for every $s\in (0,1]$,
		\begin{multline*}
		\abs*{\vect{X}_0^\gamma\left(I_{s,\alpha}-\sum_{k<N}\binom{-\alpha/\delta}{k} (\Lc_s-\Lc_0)^k I_{0,\alpha+\delta k}- \sum_{k<-Q_0+\Re \alpha+N(\delta-d_0)} s^{k-Q_0+\alpha} P_{\alpha,k}  \right)(x)}\\
		\meg \frac{s^{N} C_{N,\gamma}}{\abs{x}^{Q_0-\Re\alpha+\dd_\gamma-N(\delta -d_0)}}(1+\abs{\log\abs{s\cdot x}})
		\end{multline*}
		for every $x\in \hf_0$ such that $0\neq\abs{x}\meg s^{-1}$; the factor $1+\log\abs{s\cdot x}$ may be omitted if $\alpha\not \in Q_0+\dd_\gamma-N(\delta -d_0)+\N$.
	\end{enumerate}
\end{teo}

Let us make some examples. Take a $2$-step nilpotent Lie group $G$ and a hypoelliptic sub-Laplacian $\Lc$ thereon. Then, we may endow $G$ with the structure of a stratified group in such a way that $\Lc=\Lc_\infty+\Lc'$, where $\Lc_\infty$ and $\Lc'$ are homogeneous sums of squares of degrees $2$ and $4$, respectively. By means of the construction described in the introduction, we may choose a $2$-step stratified group $\widetilde G$ and a sub-Laplacian $\widetilde \Lc$ in such a way that $G_s=G_\infty$ as Lie groups\footnote{This is a general fact when $\widetilde G$ is a $2$-step stratified group, cf.~the proof of Theorem~\ref{teo:2}} and $\Lc_s= \Lc_\infty+s^{-2}\Lc'$ for every $s\in (0,\infty]$. Thus, in this case the first part of Theorem~\ref{teo:3} applies.

If, in the preceding example, we define $\Lc_1=\Lc_\infty^k+\Lc'$ for some $k\Meg 3$, then, applying an analogous construction, we get $G_s=G_0$ as Lie groups and $\Lc_s=\Lc_\infty^k+s^{2(k-1)}\Lc'$ for every $s\in [0,\infty)$, so that the second part of Theorem~\ref{teo:3} applies.

\begin{proof}
	{\bf1.} Assume that $G_1=G_\infty$ as Lie groups and  that $[\Lc_1,\Lc_\infty]=0$. 
	Define, for every $t> 0$, for every $s\in (0,\infty]$ and for every $\theta\in [0,1]$,
	\[
	h_{s,t}^{(\theta)}\coloneqq h_{\infty,(1-\theta) t}*h_{s,\theta t};
	\]
	observe that $(h_{s,t}^{(\theta)})_t$ is a semi-group under convolution and that the mapping $\theta \mapsto h_{s,t}^{(\theta)}\in \Sc'(G)$ is of class $C^\infty$ on $[0,1]$, with  
	\[
	\frac{\dd^k}{\dd \theta^k} h_{s,t}^{(\theta)}=(-t)^k (\Lc_s-\Lc_\infty)^k h_{s,t}^{(\theta)}
	\]
	for every $t>0$, for every $s\in (0,\infty]$, and for every $\theta\in [0,1]$. 
	Now, Proposition~\ref{prop:7} and Theorem~\ref{prop:21:2} imply that for every $\gamma$ and for every $k\in \N$ there are two constants $C,b>0$ such that
	\[
	\abs{\vect{X}^\gamma_\infty (\Lc_s-\Lc_\infty)^k h_{s,t}^{(\theta)}(x) }\meg \frac{C}{t^{\frac{Q_\infty  +\dd_\gamma+k d_\infty}{\delta}}} e^{-b \left(\frac{\abs{x}_s}{t^{\sfrac{1}{\delta}}}  \right)^{\frac{\delta}{\delta-1}} }
	\]
	for every $s\in [1,\infty]$, for every $t\Meg 1$, for every $\theta\in [0,1]$, and for every $x\in G_s$. 
	Now, take $\alpha\in \C$ such that $0<\Re \alpha<Q_\infty$, and observe that a Taylor expansion of $h_{s,t}^{(\theta)}$ in $\theta$ gives
	\[
	I_{s,\alpha}= \frac{1}{\Gamma(\frac{\alpha}{\delta})} \int_0^1 t^{\frac{\alpha}{\delta}} h_{s,t}\,\frac{\dd t}{t}+ \sum_{k<N} \frac{(-1)^k}{k!\Gamma(\frac{\alpha}{\delta}) }\int_1^{+\infty}  t^{\frac{\alpha}{\delta}+k} (\Lc_s-\Lc_\infty)^{k} h_{\infty,t}\,\frac{\dd t}{t}+R_{s,\alpha,N},  
	\]
	where
	\[
	R_{s,\alpha,N}=\frac{(-1)^N}{\Gamma(\frac{\alpha}{\delta})(N-1)! } \int_1^{+\infty} t^{\frac{\alpha}{\delta}+N}  \int_0^1 (\Lc_s-\Lc_\infty)^N h_{s,t}^{(\theta)} (1-\theta)^{N-1}\,\dd \theta\, \frac{\dd t}{t}  .
	\]
	The proof then proceeds as that of Theorem~\ref{teo:21:1}.
	
	The case in which $G_s=G_0$ and $[\Lc_s,\Lc_0]=0$ is treated similarly.
\end{proof}

\section{Spectral Measures and Multipliers}\label{sec:4}

In this section, we assume that $\widetilde \Lc$ is Rockland and formally self-adjoint, but not necessarily positive. Then $\Lc_s^2$ is weighted subcoercive, so that $(\Lc_s)$ is a weighted subcoercive system in the sense of~\cite{Martini,Martini2}.
Then, the operator $\Lc_{s}$, considered as an unbounded operator on $L^2(G_s)$ with initial domain $C^\infty_c(G_s)$, is essentially self-adjoint (cf.~\cite[Proposition 3.2]{Martini2}). We shall then denote by $\sigma(\Lc_s)$ the corresponding  spectrum. 

Now, if $m\colon \sigma(\Lc_s)\to \C$ is bounded and Borel measurable, then there is a unique distribution $\Kc_{\Lc_s}(m)$ on $G_s$ such that
\[
m(\Lc_s) \phi=\phi * \Kc_{\Lc_s}(m) 
\]
for every $\phi \in C^\infty_c(G_s)$ (cf.~\cite[Subsection 3.2]{Martini2}). In addition, there is a unique positive Radon measure $\beta_{\Lc_s}$ on $\sigma(\Lc_{s})$ such that $\Kc_{\Lc_s}$ extends to an isometry of $L^2(\beta_{\Lc_s})$ into $L^2(G_s)$  (cf.~\cite[Theorem 3.10]{Martini2}).

\begin{lem}\label{lem:3}
	There is a constant $C>0$ such that
	\[
	\beta_{\Lc_s}([-r,r])\meg C \frac{\min\left((s^{-1} r^{\sfrac{1}{\delta}})^{Q_0},(s^{-1} r^{\sfrac{1}{\delta}})^{Q_\infty} \right)}{\min(s^{-Q_0},s^{-Q_\infty}  )}
	\]
	for every $s\in [0,\infty]$ and for every $r>0$. In particular, $\beta_{\Lc_s}(\Set{0})=0$ for every $s\in [0,\infty]$.
\end{lem}

For the proof, argue as in~\cite[§ 2]{Sikora}, using the estimates for the heat kernel associated with $\Lc^2_s$ provided in Theorem~\ref{prop:21:2}.

\begin{prop}\label{prop:2}
	The following hold:
	\begin{enumerate}
		\item for every bounded Borel measurable  function $m\colon \R\to\C$, for every $s\in [0,\infty]$, and for every $r>0$, 
		\[
		\Kc_{\Lc_{r s}}(m)=(r^{-1}\,\cdot\,)_*\Kc_{\Lc_s}(m(r^\delta\,\cdot\,));
		\]
		
		\item $\beta_{\Lc_{r s}}= \nu_{G_s}(B_s(r)) (r^\delta\,\cdot\,)_*(\beta_{\Lc_s})$;
			
		\item the mapping $[0,\infty]\ni s' \mapsto \beta_{\Lc_{s'}}$ is vaguely continuous.
	\end{enumerate}
\end{prop}

\begin{proof}
	Fix $s\in [0,\infty]$, $r>0$, and $m\in \Sc(\R)$, so that $\Kc_{\widetilde\Lc}(m)\in \Sc(\widetilde G)$, where $\Kc_{\widetilde \Lc}(m)$ denotes the right convolution kernel of $m(\widetilde \Lc)$ (cf.~\cite[Proposition 4.2.1]{Martini}). Now,~\cite[Proposition 3.7]{Martini2}, applied to the quasi-regular representation of $\widetilde G$ in $L^2(G_s)$, implies that
	\[
	(\pi_s)_*(\Kc_{\widetilde\Lc}(m))=\Kc_{\Lc_s}(m).
	\]
	Now, $\pi_{r s}=(r^{-1}\,\cdot\,)\circ \pi_s\circ (r\,\cdot\,)$; in addition, since $(r\,\cdot\,)_*\widetilde \Lc= r^\delta \widetilde \Lc$ by homogeneity, we have
	\[
	\Kc_{\Lc_{r s}}(m)= (r^{-1}\,\cdot\,)_* (\pi_s)_*(\Kc_{\widetilde \Lc}(m(r^\delta\,\cdot\,)))= (r^{-1}\,\cdot\,)_*\Kc_{\Lc_s}(m(r^\delta\,\cdot\,)).
	\]
	Therefore, by means of the spectral calculus we see that $\Kc_{\Lc_{r s}}(m)=(r^{-1}\,\cdot\,)_*\Kc_{\Lc_s}(m(r^\delta\,\cdot\,))$ for every bounded Borel measurable function $m\colon \R\to \C$.
	In addition, if $m\in \Sc(\R)$, then, identifying $\Kc_{\Lc_{rs}}(m)$ and $\Kc_{\Lc_s}(m)$ with their densities with respect to $\nu_{G_{rs}}$ and $\nu_{G_s}$, respectively,
	\[
	\begin{split}
	\int_{\R} m\,\dd\beta_{\Lc_{r s}}(\lambda)&=\Kc_{\Lc_{r s}}(m)(e)\\
		&= \nu_{G_s}(B_s(r))  \Kc_{\Lc_s}(m(r^\delta\,\cdot\,))(e)\\
		&=\nu_{G_s}(B_s(r)) \int_{\R} m\,\dd (r^\delta\,\cdot\,)_*\beta_{\Lc_s}(\lambda),
	\end{split}
	\]
	so that $\beta_{\Lc_{r s}}=\nu_{G_s}(B_s(r)) (r^\delta\,\cdot\,)_*\beta_{\Lc_s}$ by the arbitrariness of $m$. In addition, it is easily seen that $ \Kc_{\Lc_{s'}}(m)(e)$ converges to $\Kc_{\Lc_s}(m)(e)$ as $s'\to s$ (see also Lemma~\ref{lem:2} below). Thanks to Lemma~\ref{lem:3} and the preceding remarks, this is sufficient to prove that  $\beta_{\Lc_{s'}}$ converges vaguely to $\beta_{\Lc_s}$ as $s'\to s$.
\end{proof}

\subsection{Asymptotic Developments}

\begin{deff}
For every $m\in \Sc(\R)$ and for every $s\in [0,\infty]$, we denote by $\Kc_{0,s}(m)$ (for $s\neq \infty$) and $\Kc_{\infty,s}(m)$ (for $s\neq 0$) the densities of the measures corresponding to $\Kc_{\Lc_s}(m)$ on $\hf_0$ and $\hf_\infty$, respectively, under the usual identifications.
\end{deff}

\begin{lem}\label{lem:2}
	The mappings
	\[
	[0,\infty)\ni s \mapsto\Kc_{0,s}\in \Lc(\Sc(\R); \Sc(\hf_0))\qquad \text{and}\qquad [0,\infty)\ni s \mapsto \Kc_{\infty,\sfrac{1}{s}}\in \Lc(\Sc(\R); \Sc(\hf_\infty))
	\]
	are of class $C^\infty$.
\end{lem}

\begin{proof}
	We prove only the first assertion.
	Observe that $\Kc_{\Lc_s}=(\pi_s)_*\circ\Kc_{\widetilde \Lc}$; since $\Kc_{\widetilde \Lc}\in \Lc(\Sc(\R); \Sc(\widetilde G))$ by~\cite[Proposition 4.2.1]{Martini}, it will suffice to prove that the mapping $[0,\infty)\ni s \mapsto (P_{0,s})_*\in \Lc(\Sc(\widetilde \gf); \Sc(\hf_0))$ is of class $C^\infty$.
	Now, for every $s\in [0,\infty)$, denote by $L_s$ the automorphism $x+y\mapsto x+y+\psi_{0,s}(y)$ of (the vector space) $\widetilde \gf\cong \hf_0\oplus \ifr_0$, and observe that $L_s$ depends polynomially on $s$, so that we may define $L_s$ for every $s\in \R$. 
	With this modification, it is readily verified that $L_s$ is still a measure-preserving automorphism of $\widetilde \gf$ for every $s\in \R$, since $\psi_{0,s}$ is strictly sub-homogeneous. Then, observe that $P_{0,s}= P_{0,0}\circ L_s^{-1}$ for every $s\in [0,\infty)$; since $(P_{0,0})_*\in \Lc(\Sc(\widetilde \gf); \Sc(\hf_0))$, it will suffice to prove that the mapping $\R\ni s\mapsto (L_s^{-1})_*\in\Lc(\Sc(\widetilde \gf))$ is of class $C^\infty$. 
	However, this last assertion is an easy consequence of the fact that the mapping $\R\ni s \mapsto L_s\in \Lc(\widetilde \gf) $ is of class $C^\infty$.	
\end{proof}

\begin{deff}
	For every $k\in \N$, for every $s\in [0,\infty)$, and for every $m\in \Sc(\R)$, define
	\[
	\Kc_{0,s}^{(k)}(m)=\frac{\dd^k}{\dd s'^k}\bigg\vert_{s'=s} \Kc_{0,s'}(m),
	\]
	and
	\[
	\Kc_{\infty,{\sfrac{1}{s}}}^{(k)}(m)=\frac{\dd^k}{\dd s'^k}\bigg\vert_{s'=s}  \Kc_{\infty,{\sfrac{1}{s'}}}(m).
	\]
\end{deff}

\begin{lem}\label{lem:4}
	Take $k\in \N$ and $m\in \Sc(\R)$. Then, for every $s\in [0,\infty)$ and for every $r>0$,
	\begin{align*}
	\Kc_{0,s}^{(k)}(m(r^\delta\,\cdot\,))&=r^k (r\,\cdot\,)_*\Kc_{0, r s}^{(k)}(m)\\
	\Kc_{\infty,\sfrac{1}{s}}^{(k)}(m(r^\delta\,\cdot\,))&=r^{-k} (r\,\cdot\,)_*\Kc_{\infty,{\sfrac{r}{s}}}^{(k)}(m)
	\end{align*}
\end{lem}

\begin{proof}
	The assertion follows from Proposition~\ref{prop:2} by differentiation.
\end{proof}

\begin{deff}
	Take $r\in \R$, and define $\Mc_r(\R^*)$ as the set of $m\in C^\infty(\R^*)$ such that for every $k\in \N$ there is a constant $C_k>0$
	\[
	\abs*{\frac{\dd^k}{\dd \lambda^k} m(\lambda)  }\meg C_k \abs{\lambda}^{r-k}
	\]
	for every $\lambda\in \R^*$. 
	We endow $\Mc_r(\R^*)$ with the corresponding semi-norms.
\end{deff}

\begin{deff}
	Take $r\in \R\cup \Set{\infty}$ and $s\in [0,\infty]$. We define $\Sc'_r(G_s)$ as the dual of the set $\Sc_r(G_s)$ of $\phi \in \Sc(G_s)$ such that $\int_{G_s} \phi(x) P(x)\,\dd x=0$ for every polynomial $P$ such that $P(x)=O(\abs{x}^r_s)$ for $x\to \infty$; we thus identify $\Sc'_r(G_s)$ with the quotient of $\Sc'(G_s)$ by the set of the polynomials as above. 
	Similar definitions replacing $G_s$ with $\hf_0$ or $\hf_\infty$, and $\abs{\,\cdot\,}_s$ with $\abs{\,\cdot\,}$.
\end{deff}

	Observe that, if $s\in (0,\infty]$, then $\Sc_r(G_s)=\Sc_r(\hf_\infty)$ under the identification of $G_s$  with $\hf_\infty$ (cf.~{\bf6} of Proposition~\ref{prop:7}).

\begin{deff}
	Take $r\in \R$, and define $\Cc\Zc_r(\hf_0)$ as the set of $K\in \Sc'_{-Q_0-r}(\hf_0)$, such that the following hold:
	\begin{itemize}
		\item for every $\alpha$ such that $\dd_\alpha>-Q_0-r$, $\partial_0^\alpha K$ has a density of class $C^\infty$ on $\hf_0\setminus \Set{0}$ and there is a constant $C_\alpha>0$ such that
		\[
		\abs*{\partial^\alpha_{0} K}\meg \frac{C_\alpha}{\abs{x}^{Q_0+r+\dd_\alpha}}
		\]
		for every non-zero $x\in \hf_0$;

		\item there is a constant $C>0$ such that
		\[
		\abs*{\langle K, \phi(\theta\, \cdot\,) \rangle}\meg C\theta^{r}
		\]
		for every $\phi \in  \Sc_{-Q_0-r}(\hf_0)$ such that $\supp{\phi}\subseteq B(1)$, and $\max_{\sum_j \alpha_j\meg([r]+1)_+}\norm{\partial_0^\alpha \phi}_\infty\meg 1$.
	\end{itemize} 
	We endow $\Cc\Zc_r(\hf_0)$ with the corresponding semi-norms.
	
	We define $\Cc \Zc_r(\hf_\infty)$ in a similar way. For every $s\in (0,\infty)$ we define $\Cc\Zc_r(G_s)$ as the set of $K\in \Sc'_{-Q_\infty-r}(G_s)$ such that there are $K_0\in \Ec'(G_s)+\Sc(G_s)$ and $K_\infty\in C^\infty(G_s)\cap \Sc'_{-Q_\infty-r}(G_s)$ such that $K=K_0+K_\infty$, and such that the distributions on $\hf_0$ and $\hf_\infty$ corresponding to $K_0$ and $K_\infty$ belong to $ \Cc\Zc_r(\hf_0)$ and $\Cc\Zc_r(\hf_\infty)$, respectively. We endow $\Cc\Zc_r(G_s)$ with the corresponding topology.
	
	Finally, we denote by $\nu_{\R_+}$ the Haar measure on $(\R_+,\,\cdot\,)$ such that $\int_0^\infty f\,\dd \nu_{\R_+}= \int_0^{\infty} f(x)\,\frac{\dd x}{x}$ for every $f\in C_c(\R_+)$.
\end{deff}

\begin{prop}\label{prop:3}
	Take $r\in \R$, a set $B$ and a bounded family $(\phi_{t,b})_{t\in (0,\infty), b\in B}$ of elements of $\Sc_{r}(\hf_0)$. Then, the mapping $t \mapsto t^{-r} (t\,\cdot\,)_*\phi_{t,b}\in \Sc'_{-Q_0-r}(\hf_0)$ is $\nu_{\R_+}$-integrable and the set of  
	\[
	K_b\coloneqq \int_0^{+\infty} t^{-r} (t\,\cdot\,)_*\phi_{t,b}\,\frac{\dd t}{t},
	\]
	as $b$ runs through $B$, is bounded in $\Cc\Zc_r(\hf_0)$. In addition, $K_b$ has a representative $\widetilde K_b$, for every $b\in B$, such that for every $\alpha$ there is a constant $C_\alpha>0$ such that
	\[
	\abs{\partial_0^\alpha\widetilde K_b(x)}\meg \frac{C_\alpha}{\abs{x}^{Q_0+r+\dd_\alpha}}(1+\abs{\log\abs{x}})
	\]
	for every $x\in \hf_0\setminus \Set{0}$, and for every $b\in B$; the factor $1+\abs{\log\abs{x}}$ may be omitted if $-Q_0-r-\dd_\alpha\not \in \N$. 
	In addition, if  $L$ is a bounded subset of $C^{([r]+1)_+}_c(\hf_0)$, then there is a constant $C'>0$ such that
	\[
	\abs*{\langle\widetilde K_b, \psi(\theta\,\cdot\,)\rangle}\meg C'\theta^r (1+\abs{\log \theta})
	\]
	for every $\theta>0$, for every $\psi\in L$, and for every $b\in B$; the factor $1+\abs{\log \theta}$ may be omitted	if $-Q_0-r\not \in \N$. 
\end{prop}

Notice that, arguing in the spirit of~\cite[Theorem 2.2.1]{NagelRicciStein2}, where the case $r=0$ is essentially considered, one may prove that for every bounded family $(K_b)_{b\in B}$ of elements of $\Cc \Zc_r(\hf_0)$ there is a bounded family $(\phi_{t,b})_{t>0,b\in B}$ of elements of $\Sc_\infty(\hf_0)$ such that $K_b=\int_0^{+\infty} t^{-r}(t\,\cdot\,)_* \phi_{t,b}\,\frac{\dd t}{t}$ for every $b\in B$.

Analogous statements hold for $\hf_\infty$.

\begin{proof}
	{\bf1.} Let $L$ be a subset of $\Sc_{-Q_0-r}(\hf_0)$ which is bounded in $C^{([r]+1)_+}_c(\hf_0)$. For every $\psi\in L$, denote by $P_\psi$ the Taylor polynomial of $\psi$ of degree $[r]$ about $0$. Then, there is a constant $C_1>0$ such that
	\[
	\abs{(\psi-P_\psi)(x)}\meg C_1 \abs{x}^{([r]+1)_+}
	\]
	for every $x\in \hf_0$. Since $\phi_{t,b}\in \Sc_r(\hf_0)$ for every $t>0$ and for every $b\in B$, 
	\[
	\begin{split}
	\int_0^{\sfrac{1}{\theta}} t^{-r} \abs{ \langle (t\,\cdot\,)_* \phi_{t,b}, \psi(\theta\,\cdot\,)\rangle  }\,\frac{\dd t}{t}&=\int_0^{\sfrac{1}{\theta}} t^{-r}\abs{ \langle \phi_{t,b}, (\psi-P_\psi)(t\theta\,\cdot\,)\rangle }\,\frac{\dd t}{t}\\
		&\meg \theta^{r} C_1\int_0^1 t^{([r]+1)_+ -r} \norm*{\abs{\,\cdot\,}^{([r]+1)_+}\phi_{t,b}}_1  \,\frac{\dd t}{t}
	\end{split}
	\]
	for every $\theta>0$, for every $\psi\in L$, and for every $b\in B$. On the other hand, denote by $P_{t,b}$ the Taylor polynomial of $\phi_{t,b}$ of degree $-Q_0+[-r]$ about $0$, for every $t>0$ and for every $b\in B$, and observe that there is a constant $C_2>0$ such that
	\[
	\abs{( \phi_{t,b}-P_{t,b})(x) }\meg C_2 \abs{x}^{(-Q_0+[-r]+1)_+},
	\]
	for every $t>0$ and for every $b\in B$. 
	Then, since $\psi\in \Sc_{-Q_0-r}(\hf_0)$,
	\[
	\begin{split}
	\int_{\sfrac{1}{\theta}}^{+\infty}  t^{-r} \abs{ \langle (t\,\cdot\,)_* \phi_{t,b}, \psi(\theta\,\cdot\,)\rangle  }\,\frac{\dd t}{t}&=\int_{\sfrac{1}{\theta}}^{+\infty}  t^{-r}\abs{ \langle (t\theta)_*(\phi_{t,b}-P_{t,b}), \psi\rangle }\,\frac{\dd t}{t}\\
	&\meg \theta^{r} C_2\norm*{ \abs{\,\cdot\,}^{(-Q_0+[-r]+1)_+}\psi}_1\int_1^{+\infty} t^{-Q_0-r-(-Q_0+[-r]+1)_+}\,\frac{\dd t}{t}.
	\end{split}
	\]
	
	Next, take  $\alpha$ such that $\dd_\alpha>-Q_0-r$, and observe that there is a constant $C_{3,\alpha}>0$ such that
	\[
	\abs{\partial^\alpha_{0} \phi_{t,b}(x) }\meg \frac{C_{3,\alpha}}{(1+\abs{x})^{Q_0+r+\dd_\alpha+1}}
	\]
	for every $x\in \hf_0$, for every $t>0$ and for every $b\in B$.  Then, fix a non-zero $x\in \hf_0$, and observe that
	\[
	\begin{split}
	\int_{0}^{+\infty} t^{-r} \abs*{\partial^\alpha_{0}(t\,\cdot\,)_* \phi_{t,b}(x)}\,\frac{\dd t}{t}&=\abs{x}^{-Q_0-r-\dd_\alpha}\int_{0}^{+\infty} t^{Q_0+r+\dd_\alpha} \abs*{ \partial_0^\alpha \phi_{\abs{x}/t,b}((\abs{x}/t)^{-1}\cdot x)  }\,\frac{\dd t}{t}\\
		&\meg C_{3,\alpha}\abs{x}^{-Q_0-r-\dd_\alpha} \int_0^{+\infty} \frac{t^{Q_0+r+\dd_\alpha}}{(1+t)^{Q_0+r+\dd_\alpha+1}}\,\frac{\dd t}{t} 
	\end{split}
	\]
	for every $b\in B$. 
	
	Taking into account all the preceding inequalities, we see that the mapping $t \mapsto t^{-r} (t\,\cdot\,)_*\phi_{t,b}\in \Sc'_{-Q_0-r}(\hf_0)$ is $\nu_{\R_+}$-integrable and that the set of  $K_b$,	as $b$ runs through $B$, is bounded in $\Cc\Zc_r(\hf_0)$. 
	
	{\bf2.}  Keep the notation of~{\bf1}, and denote by $P_{t,b,j}$ the homogeneous component of $P_{t,b}$ of degree $j$, for every $j=0,\dots, -Q_0+[-r]$; define $P'_{t,b}\coloneqq \sum_{j<-Q_0-r} P_{t,b,j}$.
	Then, the arguments of~{\bf1} show that
	\[
	\widetilde K_b \coloneqq \int_0^1 t^{-r} (t\,\cdot\,)_*(\phi_{t,b}-P_{t,b}')\,\frac{\dd t}{t}+ \int_1^{+\infty} t^{-r} (t\,\cdot\,)_*(\phi_{t,b}-P_{t,b})\,\frac{\dd t}{t}
	\]
	defines a representative of $K_b$ in $\Sc'(\hf_0)$ (treat $\int_0^1 t^{-r} (t\,\cdot\,)_* P'_{t,b}\,\frac{\dd t}{t}$ separately). In addition, arguing as in~{\bf1} we see that, for every $\alpha$,
	\[
	\abs{x}^{Q_0+r+\dd_\alpha}\left[ \int_0^{\abs{x}} t^{-r} \partial_0^\alpha (t\,\cdot\,)_*\phi_{t,b}(x)\,\frac{\dd t}{t}+ \int_{\abs{x}}^{+\infty} t^{-r}\partial_0^\alpha (t\,\cdot\,)_*(\phi_{t,b}-P_{t,b})(x)\,\frac{\dd t}{t}\right] 
	\]
	is uniformly bounded as $x$ runs through $\hf_0\setminus \Set{0}$, and $b$ runs through $B$.	
	Now, take $j\in \N$ such that $j\meg -Q_0+[-r]$. If $j<-Q_0-r$, then clearly 
	\[
	\abs{x}^{r+Q_0+\dd_\alpha}\abs*{  \int_{0}^{\abs{x}} t^{-r} \partial_0^\alpha (t\,\cdot\,)_*P_{t,b,j}(x)\,\frac{\dd t}{t}  }
	\]
	is bounded as $x$ runs through $\hf_0\setminus \Set{0}$, and $b$ runs through $B$.	
	Finally, if $j=-Q_0-r$, then clearly
	\[
	\frac{\abs{x}^{r+Q_0+\dd_\alpha}}{1+\abs{\log\abs{x}}}\abs*{  \int_{1}^{\abs{x}} t^{-r} \partial_0^\alpha (t\,\cdot\,)_*P_{t,b,j}(x)\,\frac{\dd t}{t}  }
	\]
	is bounded as $x$ runs through $\hf_0\setminus \Set{0}$, and $b$ runs through $B$. The other estimates are proved in a similar way.	Thus, $\widetilde K_b$ is the required representative of $K_b$.
\end{proof}

\begin{cor}\label{cor:2}
	Take $s\in (0,\infty)$, $r\in \R$, a set $B$, and a family $(\phi_{t,b})_{t>0, b\in B}$ such that $\phi_{t,b}\in \Sc_r(G_{s t})$ for every $t>0$ and for every $b\in B$, and such that for every $k\in \N$ there is a constant $C_k>0$ such that
	\[
	\abs{\vect{X}_{s t}^\gamma \phi_{t,b}(x)}\meg \frac{C_k}{(1+\abs{x}_{s t})^k}
	\]
	for every $\gamma$ such that $\dd_\gamma\meg k$, for every $b\in B$, for every $t>0$, and for every $x\in G_{s t}$.
	Then, the mapping $t \mapsto t^{-r} (t\,\cdot\,)_*\phi_{t,b}\in \Sc'_{-Q_\infty-r}(G_s)$ is $\nu_{\R_+}$-integrable for every $b\in B$, and the set of  
	\[
	K_b\coloneqq \int_0^{+\infty} t^{-r} (t\,\cdot\,)_*(\phi_{t,b} \nu_{G_{st}})\,\frac{\dd t}{t},
	\]
	as $b$ runs through $B$, is bounded in $\Cc\Zc_r(G_s)$. 
\end{cor}

\begin{proof}
	By an abuse of notation, we shall identify $G_{s'}$ with $\hf_0$ if $s'\in (0,1)$ and with $\hf_\infty$ if $s'\in (1,\infty)$. In addition we shall identify the measures $\phi_{t,b}\nu_{G_{st}}$ with its density $\widetilde \phi_{t,b}$ with respect to the fixed Lebesgue measure of $\hf_0$, for $t<s^{-1}$, or to the fixed Lebesgue measure of $\hf_\infty$, for $t>s^{-1}$. 
	Then, $\widetilde \phi_{t,b}$ differs from $\phi_{t,b}\circ \pi_{st}$ by a multiplicative constant which stays bounded as $t$ runs through $\R_+$.

	Observe first that, using Proposition~\ref{prop:21:9} and~{\bf6} of Proposition~\ref{prop:7}, it is not difficult to show that there is a constant $C>0$ such that
	\[
	1+\abs{x}_{s'}\Meg C(1+\abs{x}^{\sfrac{1}{n}} )
	\]
	for every $x\in \hf_0$ and for every $s'\in (0,1)$, while
	\[
	1+\abs{x}_{s'}\Meg C (1+\abs{x})
	\]
	for every $x\in \hf_\infty$ and for every $s'\in (1,\infty)$.

	Hence, the set of $\widetilde \phi_{t,b}$, as $t$ runs through $(0,s^{-1})$ and $b$ runs through $B$, is bounded in $\Sc(\hf_0)$, while the set of $\widetilde \phi_{t,b}$, as $t$ runs through $(s^{-1},\infty)$ and $b$ runs through $B$, is bounded in $\Sc(\hf_\infty)$.
	Consequently, Proposition~\ref{prop:3} and its proof imply that the
	\[
	K_{b,0}\coloneqq \int_0^{s^{-1}} t^{-r} (t\,\cdot\,)_* \phi_{t,b}\,\frac{\dd t}{t}
	\]
	are well-defined elements of $\Sc'(G_s)$ and stay bounded in $\Cc\Zc_r(\hf_0)$; analogously, the
	\[
	K_{b,\infty}\coloneqq \int_{s^{-1}}^{+\infty} t^{-r} (t\,\cdot\,)_* \phi_{t,b}\,\frac{\dd t}{t}
	\]
	are well-defined elements of $\Sc'_{-Q_0-r}(G_s)$, and  stay bounded in $\Cc\Zc_r(\hf_\infty)$.

	It will then suffice to prove that $K_{b,0}$ equals a Schwartz function in a neighbourhood of $\infty$, and that (every representative of) $K_{b,\infty}$ is of class $C^\infty$ on the whole of $G_s$ (with the required boundedness).
	
	On the one hand, take $k\Meg 1$ and $\alpha$, and observe that there is a constant $C_{k,\alpha}>0$ such that
	\[
	\abs{\partial^\alpha_{0}\widetilde \phi_{t,b}(x)}\meg \frac{C_{k,\alpha}}{(1+\abs{x})^{k+Q_0+r+\dd_\alpha}}
	\]
	for every $x\in \hf_0$, for every $t\in (0,s^{-1})$, and for every $b\in B$. Then,
	\[
	\begin{split}
	\abs{\partial_0^\alpha K_{b,0}(x)}&\meg C_{k,\alpha}\int_0^{s^{-1}} \frac{t^{-Q_0-r-\dd_\alpha}}{(1+\abs{t^{-1}\cdot x})^{k+Q_0+r+\dd_\alpha}}\,\frac{\dd t}{t} \meg  \frac{C_{k,\alpha} s^{-k}}{k\abs{x}^{k+Q_0+r+\dd_\alpha}} 
	\end{split}
	\]
	for every non-zero $x\in \hf_0$ and for every $b\in B$. By the arbitrariness of $k$ and $\alpha$, it follows that the $(1-\tau)K_{b,0}$ stay in a bounded subset of $\Sc(\hf_0)$ as $b$ runs through $B$, where $\tau$ is an element of $C^\infty_c(\hf_0)$ which equals $1$ on a neighbourhood of $0$.
	
	On the other hand, denote by $P_{t,b,j}$ the homogeneous component of degree $j$ of the Taylor series of $\widetilde \phi_{t,b}\in \Sc(\hf_\infty)$  about $0$, for every $t>0$, for every $b\in B$, and for every $j\in \N$. If $k>-Q_\infty-r$, then
	\[
	\begin{split}
	\int_{s^{-1}}^{+\infty} t^{-r} (t\,\cdot\,)_* P_{t,b,k}\,\frac{\dd t}{t}
	\end{split}
	\]
	is a well-defined homogeneous polynomial of degree $k$,	while for every $\alpha$ there is a constant $C'_{k,\alpha}>0$ such that
	\[
	\begin{split}
	\abs*{\int_{s^{-1}}^{+\infty} t^{-r} \partial^{\alpha}_{\infty}(t\,\cdot\,)_* \left(\widetilde \phi_{t,b}-\sum_{j< k} P_{t,b,j} \right)(x)\,\frac{\dd t}{t}}\meg C'_{k,\alpha} \abs{x}^{ (k-\dd_\alpha)_+ }
	\end{split}
	\]
	for every $x\in \hf_\infty$, and for every $b\in B$. 
	Define
	\[
	\widetilde K_{\infty,b}\coloneqq  \int_{s^{-1}}^{+\infty} t^{-r} (t\,\cdot\,)_* \left(\widetilde \phi_{t,b}-\sum_{j\meg -Q_\infty-r} P_{t,b,j} \right)\,\frac{\dd t}{t}, 
	\]
	so that $\widetilde K_{b,\infty}$ is a well-defined representative of $ K_{b,\infty}$ (under the identification of $G_s$ with $\hf_\infty$) by the proof of Proposition~\ref{prop:3}.
	Then, for every $k\in \N$, the
	\[
	\widetilde K_{\infty,b}=\int_{s^{-1}}^{+\infty} t^{-r} (t\,\cdot\,)_* \left(\widetilde \phi_{t,b}-\sum_{j\meg k} P_{t,b,j} \right)\,\frac{\dd t}{t} + \sum_{-Q_\infty-r<j\meg k}\int_1^{+\infty} t^{-r} (t\,\cdot\,)_* P_{t,b,j}\,\frac{\dd t}{t}
	\]
	stay bounded in $C^k(\hf_\infty)$. By the arbitrariness of $k$, it follows that the $\widetilde K_{\infty,b}$ stay bounded in $C^\infty(\hf_\infty)$.	
\end{proof}

\begin{teo}\label{teo:5}
	Take $r\in \R$. Then, the following hold:
	\begin{itemize}
		\item for every $k\in \N$, the continuous linear mapping $\Kc_{0,0}^{(k)}\colon C^\infty_c(\R^*)\to \Cc\Zc_{r\delta-k}(\hf_0) $ induces a unique continuous linear mapping $\Kc_{0,0}^{(k)}\colon \Mc_r(\R^*)\to \Cc\Zc_{r\delta-k}(\hf_0) $ such that, if $\Ff$ is a bounded filter on $\Mc_r(\R^*)$ which converges pointwise to some $m$ in $\Mc_r(\R^*)$, then $\Kc_{0,0}^{(k)}(\Ff)$ converges to $\Kc_{0,0}^{(k)}(m)$ in $\Sc'_{-Q_0-r\delta+k}(\hf_0)$;
		
		\item for every $k\in \N$, the continuous linear mapping $\Kc_{\infty,\infty}^{(k)}\colon C^\infty_c(\R^*)\to \Cc\Zc_{r\delta+k}(\hf_\infty) $ induces a unique continuous linear mapping $\Kc_{\infty,\infty}^{(k)}\colon \Mc_r(\R^*)\to \Cc\Zc_{r\delta+k}(\hf_\infty) $ such that, if $\Ff$ is a bounded filter on $\Mc_r(\R^*)$ which converges pointwise to some $m$ in $\Mc_r(\R^*)$, then $\Kc_{\infty,\infty}^{(k)}(\Ff)$ converges to $\Kc_{\infty,\infty}^{(k)}(m)$ in $\Sc'_{-Q_\infty-r\delta-k}(\hf_\infty)$;
		
		\item for every $s\in (0,\infty)$, the continuous linear mapping $\Kc_{\Lc_s}\colon C^\infty_c(\R^*)\to \Cc\Zc_{r\delta}(G_s) $ induces a unique continuous linear mapping $\Kc_{\Lc_s}\colon \Mc_r(\R^*)\to \Cc\Zc_{r\delta}(G_s) $ such that, if $\Ff$ is a bounded filter on $\Mc_r(\R^*)$ which converges pointwise to some $m$ in $\Mc_r(\R^*)$, then $\Kc_{\Lc_s}(\Ff)$ converges to $\Kc_{\Lc_s}(m)$ in $\Sc'_{-Q_\infty-r\delta}(G_s)$.
	\end{itemize}
	In addition, let $M$ be a bounded subset of $\Mc_r(\R^*)$, and take $\tau_0\in C^\infty_c(\hf_0)$ and $\tau_\infty\in C^\infty_c(\hf_\infty)$ such that $\tau_0$ and $\tau_\infty$ equal $1$ in a neighbourhood of $0$. Then, the following hold:\footnote{We denote by $\Kc_{0,s}(m)$ and $\Kc_{\infty,s}(m)$ the distributions on $\hf_0$ and $\hf_\infty$, respectively, corresponding to $\Kc_{\Lc_s}(m)$ under the usual identifications.}
	\begin{itemize}
		\item for every $N\in \N$ there is a bounded family $(K_{0,m,N,s})_{m\in M, s\in (0,1]}$ of elements of $ \Cc\Zc_{r\delta-N}(\hf_0)$ 
		 such that 
		\[
		\tau_0 \left(\Kc_{0,s}(m)-\sum_{k< N } s^k \Kc_{0,0}^{(k)}(m)
		\right) = s^N \tau_0 K_{0,m,N,s}
		\]
		in $\Sc'_{-Q_0-\delta r+N}(\hf_0)$, for every $m\in M$ and for every $s\in (0,1]$;
		
		\item for every $N\in \N$ there is a bounded family $(K_{\infty,m,N,s})_{m\in M, s\in [1,\infty)}$ of elements of $ \Cc\Zc_{r\delta+N}(\hf_\infty)$  such that 
		\[
		(1-\tau_\infty) \left(  \Kc_{\infty,s}(m)-\sum_{k< N } s^{-k} \Kc_{\infty,\infty}^{(k)}(m)\right) = s^{-N} (1-\tau_\infty) K_{\infty,m,N,s}
		\]
		in $\Sc'_{-Q_\infty-\delta r}(\hf_0)$, 	for every $m\in M$ and for every $s\in [1,\infty)$.
	\end{itemize} 
\end{teo}

\begin{proof}
	Let $M$ be a bounded subset of $\Mc_r(\R^*)$ and fix a positive function $\phi \in C^\infty_c(\R^*)$ such that $\int_0^\infty \phi(y^\delta\cdot \lambda)\, \frac{\dd y}{y}=1$ for every $\lambda\in \R^*$. 	
	Let us prove that the family $(y^{\delta r} m(y^{-\delta}\,\cdot\,) \phi  )_{y>0, m\in M} $ is bounded in $\Sc(\R)$.
	Indeed, take $h$ and observe that there is a constant $C_h>0$ such that
	\[
	\abs*{\frac{\dd^p}{\dd \lambda^p} m(\lambda) }\meg C_h \abs{\lambda}^{r-p}
	\]
	for every $\lambda\in \R^*$, for every $m\in M$, and for every $p=0,\dots,h$. Then,
	\[
	y^{\delta r} \abs*{\frac{\dd^h}{\dd \lambda^h} \left[ m(y^{-\delta}\,\cdot\,) \phi  \right] (\lambda)} \meg \sum_{h_1+h_2=h} \frac{h!}{h_1! h_2!} C_{h_1}\abs{\lambda}^{r-h_1} \norm{\phi}_{W^{h_2,\infty}(\R)} \chi_{\Supp{\phi}}(\lambda)
	\]
	for every $\lambda\in \R^*$, for every $y>0$, and for every $m\in M$, whence the assertion.	
	Next, let us prove that $\Kc^{(k)}_{0,s}(m')$ (and analogously $\Kc^{(k)}_{\infty,s}(m')$) has all vanishing moments for $m'\in C^\infty_c(\R^*)$. It will suffice to prove our assertion for $k=0$, hence for $\Kc_{\Lc_s}(m')$. Now, for every $h\in \N$ we have $m'_h\coloneqq (\,\cdot\,)^{-h} m'\in C^\infty_c(\R)$, so that $\Kc_{\Lc_s}(m')=\Lc_s^h \Kc_{\Lc_s}(m'_h)$. Since every polynomial is $\Lc_s^h$-harmonic for sufficiently large $h$ (use Proposition~\ref{lem:21:1} or observe that a similar property applies to $\widetilde \Lc$ by homogeneity arguments), the assertion follows by (sesquilinear) transposition. 
	
	Therefore, for every $k\in \N$, the family $( y^{\delta r} \Kc^{(k)}_{0,s}( m(y^{-\delta }\,\cdot\,) \phi  )  )_{y>0,s\in [0,1],m \in M}$ is bounded in $\Sc_\infty(\hf_0)$, while the family  $(  y^{\delta r} \Kc^{(k)}_{\infty,s}( m(y^{-\delta }\,\cdot\,) \phi  ) )_{y>0,s\in [1,\infty],m \in M}$ is bounded in $\Sc_\infty(\hf_\infty)$.
	Hence, Propositions~\ref{prop:3} and~\ref{cor:2} show that the
	\[
	\Kc_{0,0}^{(k)} (m)\coloneqq\int_0^{+\infty} y^{-r \delta+k} (y\,\cdot\,)_* \Kc_{0,0}^{(k)} ( m(y^{-\delta }\,\cdot\,) \phi   )\,\frac{\dd y}{y}
	\]
	are well defined and stay in a bounded subset of $\Kc_{r\delta-k}(\hf_0)$ for every $k\in \N$, that the
	\[
	\Kc_{\infty,\infty}^{(k)} (m)\coloneqq\int_0^{+\infty} y^{-r \delta-k} (y\,\cdot\,)_* \Kc_{\infty,\infty}^{(k)} ( m(y^{-\delta }\,\cdot\,) \phi   )\,\frac{\dd y}{y}
	\]
	are well defined and stay in a bounded subset of $\Kc_{r\delta+k}(\hf_\infty)$ for every $k\in \N$, and that the
	\[
	\Kc_{\Lc_s} (m)\coloneqq\int_0^{+\infty} y^{-r \delta}(y\,\cdot\,)_* \Kc_{\Lc_{s y}} ( m(y^{-\delta }\,\cdot\,) \phi   )\,\frac{\dd y}{y}
	\]
	are well defined and stay in a bounded subset of $\Kc_{r\delta}(G_s)$. Therefore, the so-defined linear mappings $\Kc_{0,0}^{(k)} $, $\Kc_{\infty,\infty}^{(k)}$, and $\Kc_{\Lc_s}$ are continuous; in addition, by Proposition~\ref{prop:2} and Lemma~\ref{lem:4} and the choice of $\phi$, they agree with their previous definition on $\Sc(\R)$, $\Sc(\R)$, and $\Mc_r(\R^*)\cap \ell^\infty(\R^*)$, respectively.
	
	Now, if $\Ff$ is a filter on $M$ which converges to some $m_0$ pointwise on $\R^*$, then $y^{\delta r} \Ff(y^{-\delta}\,\cdot\,) \phi  $ converges pointwise to $y^{\delta r} m_0(y^{-\delta}\,\cdot\,) \phi $, hence in $\Sc(\R)$. 
	As a consequence, also $\Kc^{(k)}_{0,0}(y^{\delta r} \Ff(y^{-\delta}\,\cdot\,) \phi )$ converges to $\Kc^{(k)}_{0,0}(y^{\delta r} m_0(y^{-\delta}\,\cdot\,) \phi )$ in $\Sc(\hf_0)$ for every $k\in \N$.
	Hence, $\Kc_{0,0}^{(k)}(\Ff)$ converges to $\Kc_{0,0}^{(k)}(m_0)$ in $\Sc'_{-Q_0-r \delta+k}(\hf_0) $ for every $k\in \N$. 
	The analogous assertions concerning $\Kc_{\Lc_s}$, for $s\in (0,\infty)$, and $\Kc_{\infty,\infty}^{(k)}$, for $k\in \N$, are proved similarly.
	The first three assertions of the statement are therefore established.

	Now, observe that, for every $N\in \N$,
	\begin{equation*}\label{eq:1}
	\Kc_{0,{s y}}(m(y^{-\delta}\,\cdot\,  ) \phi )= \sum_{k<N} \Kc_{0,0}^{(k)}(m( y^{-\delta}\,\cdot\,  ) \phi ) \frac{(s y)^k  }{k !}+ (s y)^{N }\int_0^1 \Kc_{0,{ s y \theta}}^{(N)}(m( y^{-\delta}\,\cdot\,  ) \phi) \frac{(1- \theta)^{N-1}}{(N-1)!}\,\dd \theta
	\end{equation*}
	for $y\in (0,s^{-1})$, while
	\begin{equation*}\label{eq:2}
	\Kc_{\infty,{s y}}(m( y^{-\delta}\,\cdot\,  ) \phi )= \sum_{k<N} \Kc_{\infty,\infty}^{(k)}(m( y^{-\delta}\,\cdot\,  ) \phi ) \frac{(s y)^{-k}  }{k !}+ (s y)^{-N}\int_0^1 \Kc_{\infty,{\sfrac{ s y }{ \theta}}}^{(N)}(m( y^{-\delta}\,\cdot\,  ) \phi) \frac{(1- \theta)^{N-1}}{(N-1)!}\,\dd  \theta
	\end{equation*}
	for $y\in (s^{-1},\infty)$. 
	In addition, the $ \int_0^1 \Kc_{0,{y s  \theta}}^{(N)}(m( y^{-\delta}\,\cdot\,  ) \phi) \frac{(1- \theta)^{N-1}}{(N-1)!}\,\dd  \theta$ are bounded in $\Sc_\infty(\hf_0)$ as $y$ run through $(0,s^{-1})$, while the $\int_0^1 \Kc_{\infty,{\sfrac{ s y}{ \theta}}}^{(N)}(m( y^{-\delta}\,\cdot\,  ) \phi) \frac{(1- \theta)^{N-1}}{(N-1)!}\,\dd \theta$ are  bounded in $\Sc_\infty(\hf_\infty)$ as $y$ runs through $(s^{-1},\infty)$. 
	In addition observe that, if $V$ is a finite-dimensional vector space, $F$ is a closed subset of $V$ with non-empty interior, and $\Pc$ is a linearly independent finite set of polynomials on $V$, then the mapping $\Sc(F)\ni \phi\mapsto (\int \phi(x) P(x)\,\dd x)\in \C^\Pc$ is onto, where $\Sc(F)$ is the set of $\phi \in \Sc(V)$ supported in $F$, with the topology induced by $\Sc(V)$.
	Applying~\cite[Proposition 12 of Chapter II, § 4, No.\ 7]{BourbakiTVS}, we see that, if $B_0$ and $B_\infty$ are bounded subsets of $C^\infty(\hf_0)$ and $\Sc(\hf_\infty)$, respectively, then there are bounded subsets $B'_0$ and $B'_\infty$ of $\Sc_{r\delta -N}(\hf_0) $ and $\Sc_{r\delta+N}(\hf_\infty)$ such that $\tau_0(B_0)=\tau_0 B'_0$ and $(1-\tau_\infty)B_\infty=(1-\tau_\infty) B'_\infty$.  
	Hence, Corollary~\ref{cor:2} (and its proof) again implies that the
	\[
	\tau_0\left( \Kc_{0,s}(m)-\sum_{k< N } s^{k}\Kc_{0,0}^{(k)}(m)\right) 
	\]
	stay bounded in $\tau_0\Kc_{r\delta-N}(\hf_0)$ as $m$ runs through $M$ and $s$ is fixed, 
	while the
	\[
	(1-\tau_\infty) \left( \Kc_{\infty,s}(m)-\sum_{k< N } s^{-k}\Kc_{\infty,\infty}^{(k)}(m) \right) 
	\]
	stay bounded in  $(1-\tau_\infty)\Kc_{r \delta+N}(\hf_\infty)$ as $m$ runs through $M$ and $s$ is fixed.
	In order to establish uniform boundedness for general $s$ as in the statement, it suffices to reduce to the case $s=1$, taking into account Proposition~\ref{prop:2} and Lemma~\ref{lem:4}.
	The proof is therefore complete.
\end{proof}

\subsection{Multiplier Theorems}

Here we shall repeat the arguments of~\cite[§ 4.1]{Martini} in order to provide a multiplier theorem for the operators $\Lc_s$, which will imply some sort of continuity for the mapping $s\mapsto \Kc_{\Lc_s}(m)$ for more general $m$.
Even though the following results hold when $\widetilde \Lc$ is  self-adjoint, in order to avoid some technical issues we shall assume that $\widetilde\Lc$ is positive.

\emph{In this section, when $\mi$ is a measure on $G_s$ which is absolutely continuous with respect to the Haar measure, we shall write $\norm{\mi}_{L^p(\nu_{G_s})}$ to denote the $L^p$ norm of its density with respect to $\nu_{G_s}$, $p\in [1,\infty]$.  }

We recall the definition of some Besov spaces on $\R$ (cf.~\cite[Theorem of Section 2.6.1]{Triebel} and~\cite[Section 5]{Amann}).

\begin{deff}
	Take $\alpha>0$. Then, $B^\alpha_{\infty,\infty}(\R)$ is the space of $f\in L^\infty(\R)$ such that 
	\[
	\sup_{x\neq 0} \frac{\norm{\Delta^{([\alpha]+1)}_x f }_\infty}{\abs{x}^\alpha}<\infty,
	\]
	where $ \Delta_x^{([\alpha]+1)} f\coloneqq \sum_{j=0}^{[\alpha]+1} (-1)^{[\alpha]+1-j} \binom{[\alpha]+1}{j}f(\,\cdot\,+ j x)$ for every $x\in \R$, endowed with the corresponding topology. 
	We denote by $b^\alpha_{\infty,\infty}(\R)$ the closure of $B^{\alpha+1}_{\infty,\infty}(\R)$ in $B^\alpha_{\infty,\infty}(\R)$.
	
	We also denote by $H^\alpha(\R)$ the classical Sobolev space of $f\in L^2(\R)$ such that $\Fc^{-1}( (1+\abs{\,\cdot\,}^2)^{\alpha/2} \Fc f)\in L^2$, where $\Fc$ denotes the Fourier transform.
\end{deff}

Recall that $\nu_{\R_+}$ is a Haar measure on the multiplicative group $\R_+$.

\begin{prop}\label{prop:4}
	For every $r>0$, for every $\gamma$, and for every $\alpha_1,\alpha_2\Meg 0$ such that $\alpha_2>\alpha_1$ there is a constant $C>0$ such that
	\[
	\norm{\vect{X}_s^\gamma \Kc_{\Lc_s}(m) (1+\abs{\,\cdot\,}_{s,*})^{\alpha_1} }_{L^2(\nu_{G_s})} \meg C \norm{m}_{B^{\alpha_2}_{\infty,\infty}(\R)}
	\]
	for every $m\in B^{\alpha_2}_{\infty,\infty}(\R)$ with $\supp{m}\subseteq [-r,r]$, and for every $s\in [0,\infty]$.
	
	If, in addition, $\beta_{\Lc_1}$ has a density with respect to $\nu_{\R_+}$ bounded by $\min[ (\,\cdot\,)^{\sfrac{Q_0}{\delta}}, (\,\cdot\,)^{\sfrac{Q_\infty}{\delta}}]$, then we may take $C$ in such a way that
	\[
	\norm{\vect{X}_s^\gamma \Kc_{\Lc_s}(m) (1+\abs{\,\cdot\,}_{s,*})^{\alpha_1} }_{L^2(\nu_{G_s})}  \meg C \norm{m}_{H^{\alpha_2}(\R)}
	\]
	for every $m\in H^{\alpha_2}(\R)$ with $\supp{m}\subseteq [-r,r]$, and for every $s\in [0,\infty]$.
\end{prop}

\begin{proof}
	Proceed as in the proofs of~\cite[Lemma 4.1.1 to Theorem 4.1.6]{Martini}, taking into account the following modifications and remarks: 
	\begin{itemize}
		\item define 
		\[
		E_\ell= e^{i\ell e^{-1-(\,\cdot\,)} }-1=\sum_{k=1}^\infty \frac{(i \ell)^k}{k!} e^{-k-k(\,\cdot\,)},
		\]
		for every $\ell\in \Z$;
		
		\item replace the references to~\cite[2.3 (e) and (f)]{Martini2} with Theorem~\ref{prop:21:2} and Proposition~\ref{prop:5};
		
		\item $\sup_{s\in [0,\infty]}\beta_{\Lc_s}([-r,r])$ is finite for every $r>0$ thanks to Lemma~\ref{lem:3};
		
		\item if $\beta_{\Lc_1}\meg C'\min[ (\,\cdot\,)^{\sfrac{Q_0}{\delta}}, (\,\cdot\,)^{\sfrac{Q_\infty}{\delta}}]\cdot \nu_{\R_+}$ for some $C'>0$, then there is a constant $C''>0$ such that $\beta_{\Lc_s}\meg C''\max[ (\,\cdot\,)^{\sfrac{Q_0}{\delta}}, (\,\cdot\,)^{\sfrac{Q_\infty}{\delta}}]\cdot \nu_{\R_+}$ for every $s\in [0,\infty]$.\qedhere
	\end{itemize}
\end{proof}

\begin{teo}\label{teo:1}
	Take $N_0$ and $N_\infty$ as in Proposition~\ref{prop:6}. In addition, take a non-zero $\psi\in C^\infty_c(\R_+)$ and $\alpha>0$, and for every $s\in [0,\infty]$ denote by $\Mc_{\alpha,s}$ the space of $m\in L^1_\loc(\R_+)$ such that
	\[
	\begin{split}
	\norm{m}_{\Mc_{\alpha,s}}&\coloneqq\sup_{t> s^\delta}\left(\norm{\psi \,m(t\,\cdot\,)}_{B^{(D_{0}+\alpha)/2}_{\infty,\infty}(\R)} + (t/s^\delta)^{-{N_0}/{2\delta}}\norm{\psi \,m(t\,\cdot\,)}_{B^{(D_{1}+\alpha)/2}_{\infty,\infty}(\R)}   \right)\\
		&\qquad + \sup_{0<t\meg s^\delta}\left(\norm{\psi \,m(t\,\cdot\,)}_{B^{(D_{\infty}+\alpha)/2}_{\infty,\infty}(\R)} + (t/s^\delta)^{{N_\infty}/{2\delta}}\norm{\psi \,m(t\,\cdot\,)}_{B^{(D_{1}+\alpha)/2}_{\infty,\infty}(\R)} \right)
	\end{split}
	\]
	is finite.
	Then, for every $p\in (1,\infty)$ there is a constant $C_p>0$ such that
	\[
	\norm{m(\Lc_s)}_{\Lc(L^p(G_s))}\meg C_p \norm{m}_{\Mc_{\alpha,s}}
	\]
	for every $m\in \Mc_{\alpha,s}$ and for every $s\in [0,\infty]$.
	In addition, take $s_0\in (0,\infty)$, two functions $m_0,m_\infty$ such that $\sup_{0<s<s_0} \norm{m_0}_{\Mc_{\alpha,s}}, \sup_{s>s_0}\norm{m_\infty}_{\Mc_{\alpha,s}}<\infty$,  and $f_s\in L^p(G_s)$, for every $s\in [0,\infty]$, such that $f_s$ converges to $f_0$ in $L^p(\hf_0)$ as $s\to 0^+$, and such that $f_s$ converges to $f_\infty$ in $L^p(\hf_\infty)$ as $s\to+ \infty$; then,
	\[
	\lim_{s\to 0^+} (f_s * \Kc_{\Lc_s}(m_0))= f_0*\Kc_{\Lc_0}(m_0)\qquad \text{and}\qquad \lim_{s\to +\infty}( f_s* \Kc_{\Lc_s}(m_\infty))= f_\infty* \Kc_{\Lc_\infty}(m_\infty),
	\]
	 in $L^p(\hf_0)$ and  in $L^p(\hf_\infty)$, respectively.
	
	Finally, assume that $\beta_{\Lc_1}$ has a density with respect to $\nu_{\R_+}$ bounded by $\min[ (\,\cdot\,)^{\sfrac{Q_0}{\delta}}, (\,\cdot\,)^{\sfrac{Q_\infty}{\delta}}]$; define $\Mc'_{\alpha,s}$ as the space of $m\in L^1_\loc(\R_+)$ such that
	\[
	\begin{split}
	\norm{m}_{\Mc'_{\alpha,s}}&\coloneqq\sup_{t> s^\delta}\left(\norm{\psi \,m(t\,\cdot\,)}_{H^{(D_{0}+\alpha)/2}(\R)} + (t/s^\delta)^{-\sfrac{N_0}{\delta}}\norm{\psi \,m(t\,\cdot\,)}_{H^{(D_{1}+\alpha)/2}(\R)}   \right)\\
	&\qquad + \sup_{0<t\meg s^\delta}\left(\norm{\psi \,m(t\,\cdot\,)}_{H^{(D_{\infty}+\alpha)/2}(\R)} + (t/s^\delta)^{\sfrac{N_\infty}{\delta}}\norm{\psi \,m(t\,\cdot\,)}_{H^{(D_{1}+\alpha)/2}(\R)} \right)
	\end{split}
	\]	
	is finite.  Then, $\Mc_{\alpha,s}$ may be replaced by $\Mc'_{\alpha,s}$ in the previous assertions.
\end{teo}

Taking into account Theorem~\ref{cor:5}, this generalizes~\cite{Alexopoulos} for higher-order operators, and also~\cite[Theorem 2]{Sikora} for quasi-homogeneous sums of even powers of left-invariant vector fields on a homogeneous group, with $N_0=D_{1}-D_{0}$ at least when these powers are all equal. Notice that the proofs of~\cite[Theorem]{Alexopoulos} and~\cite[Theorem 2]{Sikora}, which are based on the property of finite speed of propagation of the wave equation, cannot be extended to the present setting.

\begin{proof}
	{\bf1.} We shall denote by $M_s$ the space $\Mc_{\alpha,s}$ under the first set of assumptions, and the space $\Mc'_{\alpha,s}$ under the second set of assumptions.
	Notice that we may assume that $\psi$ is positive and chosen in such a way that $\sum_{j\in \Z} \psi(2^{-{\delta j}}\cdot\lambda)=1$ for every $\lambda>0$. 
	Fix $\eps\in (0,\alpha)$. Then, Propositions~\ref{prop:4} and~\ref{prop:6} imply that there is $p_0>1$ such that for every $\gamma$ there is a constant $\widetilde C_\gamma>0$ such that
	\[
	\begin{split}
	&\int_{G_{2^{-j}s}}\abs{\vect{X}^\gamma_{2^{-j}s} \Kc_{\Lc_{2^{-j}s}}(\psi\,m(2^{\delta j}\,\cdot\,)  )(x)}^p(1+\abs{x}_{2^{-j}s})^\eps\,\dd x\\
		&\qquad\qquad\qquad\meg \int_{G_{2^{-j}s}}\abs{\vect{X}^\gamma_{2^{-j}s} \Kc_{\Lc_{2^{-j}s}}(\psi\,m(2^{\delta j}\,\cdot\,)  )(x)}^p(1+\abs{x}_{2^{-j}s,*})^\eps\,\dd x\meg \widetilde C_\gamma \norm{m}_{M_s}
	\end{split}
	\]
	for every $p\in [1, p_0]$, for every $s\in [0,\infty]$, for every $m\in M_s$, and for every $j\in \Z$. In addition,
	\[
	\int_{G_s} \Kc_{\Lc_s}(\psi\,m(2^{\delta j}\,\cdot\,)  )(x)\,\dd x=0
	\]
	for every $s\in [0,\infty]$, for every $m\in M_s$, and for every  $j\in\Z$. 
	Observe that, since $D_{0},D_{\infty}\Meg 1$, the space $M_s$ embeds in $L^\infty(\R)$ (cf.~\cite[Propositions 2.3.2 and 2.3.6]{Martini}). 
	On the other hand, observe that
	\[
	\Kc_{\Lc_s}(\psi(2^{-\delta j}\,\cdot\,) \,m)=(2^{- j}\,\cdot\,)_* \Kc_{\Lc_{2^{-j}s}}(\psi\, m(2^{\delta j}\,\cdot\,)) 
	\]
	for every $s\in [0,\infty]$, for every $m\in M_s$, for every $\gamma$, and for every $j\in \Z$.
	
	Now, since $m$ is the sum of the series $\sum_{j\in \Z} \psi(2^{-\delta j}\,\cdot\,) m$ pointwise on $(0,\infty)$, and since the partial sums of that series are uniformly bounded, we see that 
	\[
	\Kc_{\Lc_s}(m)= \sum_{j\in \Z} \Kc_{\Lc_s}(\psi(2^{-\delta j}\,\cdot\,)\, m)
	\]
	in the space of (right) convolutors of $L^2(G_s)$, for every $s\in [0,\infty]$ and for every $m\in M_s$; in particular, in $\Sc'(G_s)$.
	
	Let us first prove that the sum converges in $L^1_\loc(G_s\setminus\Set{e})$. Indeed, take a compact subset $L$ of $G_s\setminus \Set{e}$, and observe that
	\[
	\begin{split}
	\sum_{j\in \Z}\norm{\chi_L \Kc_{\Lc_s}(\psi(2^{-\delta j}\,\cdot\,) m)}_1&\meg \sum_{j\in \Z} \norm{\chi_{2^j\cdot L} \Kc_{\Lc_{2^{-j}s}}(\psi\, m(2^{\delta j}\,\cdot\,))  }_1\\
		&\meg \sum_{j\meg 0}  \widetilde C_0^{\frac{1}{p_0}} \nu_{G_{2^{-j} s}}(2^j\cdot L)^{\frac{1}{p_0'}}+\sum_{j>0}  \widetilde C_0  \sup_{2^j\cdot L} \abs{\,\cdot\,}_{2^{-j}s}^{-\eps}\\
		&=  \widetilde C_0^{\frac{1}{p_0}}  \nu_{G_{ s}}(L)^{\frac{1}{p_0'}} \sum_{j\meg 0}\nu_{G_{2^{-j} s}}\left( B_{2^{-j}s}(2^j)\right)^{\frac{1}{p_0'}}+ \widetilde C_0 \sup_{L} \abs{\,\cdot\,}_{ s}^{-\eps}\sum_{j>0} 2^{-\eps j},
	\end{split} 
	\]
	which is finite for every $s\in [0,\infty]$ and for every $m\in M_s$ with $\norm{m}_{M_s}\meg 1$,  since $ \nu_{G_{2^{-j} s}}\left(B_{2^{-j}s}(2^j)\right)\asymp 2^{j Q_\infty} $ as $j\to -\infty$ for fixed $s\neq 0$ thanks to~{\bf6} of Proposition~\ref{prop:7}, while  $\nu_{G_{0}}\left( B_0(2^j) \right)=2^{j Q_0} $ for every $j\in\Z$.
	
 	Next, let us prove that 
	\[
	\sup_{s\in [0,\infty]}\sup_{\norm{m}_{M_s}\meg1} \sup_{y\neq e} \sum_{j\in \Z}\int_{\abs{x}_s\Meg 2\abs{y}_s} \abs{ \Kc_{\Lc_s}(\psi(2^{-\delta j}\,\cdot\,) \,m)( y^{-1} x )-  \Kc_{\Lc_s}(\psi(2^{-\delta j}\,\cdot\,) m)(x)}\,\dd x<+\infty.
	\]
	Indeed, 
	\[
	\begin{split}
	&\sum_{2^j\abs{y}_s \Meg 1}\int_{\abs{x}_s\Meg 2 \abs{y}_s} \abs{ \Kc_{\Lc_s}(\psi(2^{-\delta j}\,\cdot\,) \,m)(y^{-1} x )-  \Kc_{\Lc_s}(\psi(2^{-\delta j}\,\cdot\,)\, m)(x)}\,\dd x\\
		&=\sum_{2^j\abs{y}_s \Meg 1}\int_{\abs{x}_{2^{-j}s}\Meg 2  \abs{2^j\cdot y}_{2^{-j}s}} \abs{ \Kc_{\Lc_{2^{-j}s}}(\psi \,m(2^{\delta j}\,\cdot\,))((2^j\cdot y)^{-1} x )-  \Kc_{\Lc_{2^{-j}s}}(\psi \,m(2^{\delta j}\,\cdot\,))(x)}\,\dd x\\
		&\meg 2\sum_{2^j\abs{y}_s \Meg 1}\int_{\abs{x}_{2^{-j}s}\Meg  2^j\abs{y}_{s}} \abs{ \Kc_{\Lc_{2^{-j}s}}(\psi \,m(2^{\delta j}\,\cdot\,))(x)}\,\dd x\\
		&\meg 2  \widetilde C_0 \sum_{2^j \abs{y}_s\Meg 1} \sup_{\abs{x}_{2^{-j}s}\Meg  2^j\abs{y}_{s}} \abs{x}_{2^{-j}s}^{-\eps}\\
		&= 2  \widetilde C_0 \sum_{2^j \abs{y}_s\Meg 1} (2^j\abs{y})^{-\eps},
	\end{split}
	\]
	which is uniformly bounded for $s\in [0,\infty]$, $\norm{m}_{M_s}\meg1$,  and $y\in G_s\setminus \Set{e}$.
	
	Finally, 
	\[
	\begin{split}
	&\sum_{2^j\abs{y}_s < 1}\int_{\abs{x}_s\Meg 2 \abs{y}_s} \abs{ \Kc_{\Lc_s}(\psi(2^{-\delta j}\,\cdot\,) m)(  y^{-1} x )-  \Kc_{\Lc_s}(\psi(2^{-\delta j}\,\cdot\,) m)(x)}\,\dd x\\
	&\qquad=\sum_{2^j\abs{y}_s < 1}\int_{\abs{x}_{2^{-j}s}\Meg 2   \abs{2^j\cdot y}_{2^{-j}s}} \abs{ \Kc_{\Lc_{2^{-j}s}}(\psi \,m(2^{\delta j}\,\cdot\,))( (2^j\cdot y)^{-1} x)-  \Kc_{\Lc_{2^{-j}s}}(\psi \,m(2^{\delta j}\,\cdot\,))(x)}\,\dd x\\
	&\qquad\meg  \sum_{j'\in J}  \widetilde C_{\vect{e}_{j'}} \sum_{2^j\abs{y}_s < 1}\abs{2^j\cdot y}_s^{\dd_{j'}},
	\end{split}
	\]
	which is uniformly bounded for $s\in [0,\infty]$, $\norm{m}_{M_s}\meg1$,  and $y\in G_s\setminus \Set{e}$ (here, Lemma~\ref{lem:7} is applied to $\Kc_{\Lc_{2^{-j}s}}(\psi \overline{ m(2^{\delta j}\,\cdot\,)})=\Kc_{\Lc_{2^{-j}s}}(\psi  m(2^{\delta j}\,\cdot\,))^*$).
	
	Observe, now, that, for every $s\in [0,\infty]$ and for every $t>0$,
	\[
	\frac{\nu_{G_s}(B_s(2t))}{\nu_{G_s}(B_s(t))}=\frac{\nu_{G_{t^{-1}s}}(B_{t^{-1}s}(2))}{\nu_{G_{t^{-1}s}}(B_{t^{-1}s}(1))}= \nu_{G_{t^{-1}s}}(B_{t^{-1}s}(2)),
	\]
	which is a bounded function of $t^{-1}s$ on $[0,\infty]$.	
	Therefore, thanks to~\cite[Theorem 3 of Chapter 1]{Stein2} we see that for every $p\in (1,2]$ there is a constant $C_p>0$ such that
	\[
	\norm{m(\Lc_s)}_{\Lc(L^p(G_s))}\meg C_p \norm{m}_{M_s}
	\]
	for every $s\in [0,\infty]$ and for every $m\in M_s$. A similar assertion holds, by duality, also for $p\in (2,\infty)$.
	
	{\bf2.} Now, identify $G_s$ with $\hf_0$ for every $s\in [0,s_0]$, and observe that
	\[
	\lim_{s\to 0^+} m(\Lc_s)= m(\Lc_0)
	\] 
	in $\Lc(L^p(\hf_0))$ for every $m\in \Sc(\R)$.
	Define $\widetilde \Mc_{\alpha,s}$ replacing the Besov spaces $B_{\infty,\infty}$ with the little Besov spaces $b_{\infty,\infty}$, and define $\widetilde M_s$ as $\widetilde \Mc_{\alpha/2,s}$ and  $\Mc'_{\alpha,s}$ under the first and second set of assumptions, respectively. Observe that $M_s$ embeds continuously into $\widetilde M_{s}$ (cf.~\cite[Proposition 2.3.2]{Martini}), so that we may replace $M_s$ with $\widetilde M_s$ in the assumptions.
	Then, by means of~\cite[Corollaries 2.3.10 and 2.3.7, and Proposition 2.3.13]{Martini} we see that $\tau\Sc(\R)$ is dense in $\tau\widetilde M_s$  for every $\tau\in C^\infty_c(\R^*)$, so that
	\[
	\lim_{s\to 0^+} (\tau m_0)(\Lc_s)= (\tau m_0)(\Lc_0)
	\] 
	in $\Lc(L^p(\hf_0))$, since the $(\tau m_0)(\Lc_s)$ are equicontinuous on $L^p(\hf_0)$ thanks to~{\bf1} and the assumptions on $m_0$. Therefore, for every finite subset $J$ of $\Z$,
	\[
	\lim_{s\to 0^+} \sum_{j\in J}(\psi(2^{-\delta j}\,\cdot\,) m_0)(\Lc_s)= \sum_{j\in J}(\psi(2^{-\delta j}\,\cdot\,) m_0)(\Lc_0)
	\]
	in $\Lc(L^p(\hf_0))$. 
	Now, define $K_{m_0,s,k}\coloneqq \sum_{-k<j\meg k} \Kc_{\Lc_s}(\psi(2^{-\delta j}\,\cdot\,) m_0)$ and $\widetilde \psi\in C^\infty_c(\R)$ so that $\widetilde \psi= \sum_{j\meg 0} \psi(2^{-\delta j}\,\cdot\,)$ on $\R^*$. Then,
	\[
	K_{m_0,s,k}= \Kc_{\Lc_s}( \widetilde \psi(2^{\delta k}\,\cdot\,)-\widetilde \psi(2^{-\delta k}\,\cdot\,) ) * \Kc_{\Lc_s}(m_0)
	\]
	for every $s\in [0,s_0]$ and for every $k\in \N$. Therefore, for every $\phi \in C^\infty_c(\hf_0)$ we have, with some abuses of notation,
	\[
	\begin{split}
	\limsup_{s\to 0^+} & \norm{\phi*_{G_s}\Kc_{\Lc_s}(m_0)-\phi*_{G_0} \Kc_{\Lc_0}(m_0)  }_p \meg \limsup_{s\to 0^+}\big(  \norm{\phi*_{G_s}\Kc_{\Lc_s}(m)-\phi*_{G_s} K_{m_0,s,k}  }_p\\
		&\qquad+ \norm{\phi*_{G_s}K_{m_0,s,k}-\phi*_{G_0} K_{m_0,0,k}  }_p+\norm{\phi*_{G_0} K_{m_0,0,k}-\phi*_{G_0}\Kc_{\Lc_0}(m_0)  }_p\big) \\
		&\meg  2 C'' \sup_{s\in [0,s_0]} \norm{\phi-\phi*_{G_s}\Kc_{\Lc_s}(\widetilde \psi(2^{\delta k}\,\cdot\,)-\widetilde \psi(2^{-\delta k}\,\cdot\,)  ) }_p
	\end{split}
	\]
	where $C''= \sup_{s\in [0,s_0]} \norm{m_0(\Lc_s)}_{\Lc(L^p(G_s))}$. Now, since $\phi \in C^\infty_c(\hf_0)$ and since the $\Kc_{\Lc_s}(\widetilde \psi)$, as $s$ runs through $[0,s_0]$, stay in a bounded subset of $\Sc(\hf_0)$, it is easily seen that 
	\[
	\lim_{k\to \infty}  \sup_{s\in [0,s_0]} \norm{\phi-\phi*_{G_s}\Kc_{\Lc_s}(\widetilde \psi(2^{\delta k}\,\cdot\,)-\widetilde \psi(2^{-\delta k}\,\cdot\,)  ) }_p=0,
	\] 
	whence 
	\[
	\lim_{s\to 0^+} \phi*_{G_s}\Kc_{\Lc_s}(m_0)=\phi*_{G_0} \Kc_{\Lc_0}(m_0) 
	\]
	in $L^p(\hf_0)$. Since the $m_0(\Lc_s)$, as $s$ runs through $[0,s_0]$, induce equicontinuous endomorphisms of $L^p(\hf_0)$, the assertion in the statement follows.
	The case $s\to +\infty$ is treated similarly.	
\end{proof}

Notice that the regularity threshold in Theorem~\ref{teo:1} is not optimal, in general. We shall now present an improvement of Theorem~\ref{teo:1}, under more restrictive hypotheses, in the spirit of~\cite{Hebisch,HebischZienkiewicz,Martini3}. 
Let us briefly recall the notion of capacity introduced in~\cite{Martini,Martini3}; we shall present it in a slightly simpler way \emph{in the setting of $2$-step stratified groups}.

\begin{deff}
	Let $G'$ be a $2$-step stratified group with Lie algebra $\gf'$; let $(\gf'_1,\gf'_2)$ be the stratification of $\gf'$ and take $h\in \Set{0,\dots, \dim \gf'_2}$. Endow $\gf'$ with a scalar product.
	Then, we say that $G'$ is $h$-capacious if there is a linearly independent family $X_1,\dots, X_h$ of elements of $\gf'_1$ and a linearly independent family $T_1,\dots, T_h$ of elements of $\gf'_2$ such that
	\[
	\abs{\langle T \vert [X,\,\cdot\,\,] \rangle }_{\gf^*}\Meg \sum_{j=1}^h \abs{ \langle X\vert X_j\rangle \langle T \vert T_j\rangle }
	\]
	for every $X\in \gf'_1$ and for every $T\in \gf'_2$.
\end{deff}

For instance, if $G'$ is the product of a finite family of Métivier or abelian groups, then $G'$ is $\dim[G',G']$-capacious (cf.~\cite[Proposition 3.9]{Martini3}), so that the following result applies (with a suitable choice of $\widetilde G$) when $\Lc_1$ has the form $\sum_{j\in J_1}(i X_j)^{\alpha}$, where $\alpha\in 2\N^*$ and $(X_j)$ is a family of left-invariant vector fields on $G_1$ which generates its Lie algebra.

Notice that, when $G'$ is an $H$-type group and $\Lc_1=\Lc_1'-\sum_{j} T_j^2$, where $\Lc_1'$ is the standard (homogeneous) sub-Laplacian and the $T_j$ stay in the centre of $\gf_1$, then Theorem~\ref{teo:2} is a consequence of~\cite[Corollary 2.4]{MullerRicciStein}.

\begin{teo}\label{teo:2}
	Assume that $\widetilde G$ is a $2$-step stratified group and that $G_\infty$ is $h$-capacious for some $h\in \N$. Then, there is $h'\Meg (h-Q_\infty+Q_0)_+$ such that $G_0$ is $h'$-capacious.
	In addition, take a non-zero $\psi\in C^\infty_c(\R_+)$ and $\alpha>0$, and for every $s\in [0,\infty]$ denote by $\Mc_{\alpha,s}$ the space of $m\in L^1_\loc(\R_+)$ such that 
	\[
	\norm{m}_{\Mc_{\alpha,s}}\coloneqq \sup_{t>0} \left( \norm{\psi m(t\,\cdot\,)}_{B^{(Q_0-h'+\alpha)/2}_{\infty,\infty}}+ (1+t/s^\delta)^{(Q_0-h'-Q_\infty+h)/{2\delta} }\norm{\psi m(t\,\cdot\,)}_{B^{(Q_\infty-h+\alpha)/2}}  \right)
	\]
	is finite.
	Then, there is a constant $C>0$ such that
	\[
	\norm{m(\Lc_s)}_{\Lc(L^p(G_s))}\meg C \norm{m}_{\Mc_{s,\alpha}}
	\]
	for every $s\in [0,\infty]$ and for every $m\in\Mc_{s,\alpha}$.
\end{teo}

The second part of Theorem~\ref{teo:1} can be proved with the same techniques also under the assumptions of Theorem~\ref{teo:2}. We leave the details to the reader.

\begin{proof}
	{\bf1.} Observe first that $G_s=G_\infty$ as Lie groups for every $s\in (0,\infty]$, under the identification with $\hf_\infty$. Indeed, it suffices to observe that, if $X,Y\in \hf_\infty$, then $[X,Y]\in \widetilde \gf_2$, so that $[X,Y]_s=P_{\infty,s}[X,Y]=P_{\infty,\infty}[X,Y]=[X,Y]_\infty$ thanks to Lemma~\ref{lem:21:2}. 
	Fix scalar products on $\hf_\infty$ and $\hf_0$ such that the bases $(\widetilde X_j)_{j\in J_\infty}$ and $(\widetilde X_j)_{j\in J_0}$ are orthonormal.
	Observe that, since $G_\infty$ is $h$-capacious, there are two linearly independent families $( Y_j)_{j=1,\dots,h}$ of elements of $\hf_\infty\cap\widetilde \gf_1$ and $( T_j)_{j=1,\dots,h}$ of elements of $\hf_\infty\cap\widetilde \gf_2$ such that 
	\[
	\abs{\langle T \vert [X,\,\cdot\,\,]_\infty\rangle }_{\hf_\infty^*}\Meg \sum_{j=1}^h \abs{ \langle X\vert Y_j\rangle \langle T\vert T_j\rangle } 
	\]
	for every $X\in \hf_\infty\cap \widetilde \gf_1$ and for every $T\in \hf_\infty\cap \widetilde \gf_2$. By the preceding remarks, we also have
	\[
	\abs{\langle T \vert [X,\,\cdot\,\,]_s\rangle }_{\hf_\infty^*}\Meg \sum_{j=1}^h \abs{ \langle X\vert  Y_j\rangle \langle T\vert T_j\rangle } 
	\]
	for every $s\in (0,\infty]$, for every $X\in \hf_\infty\cap \widetilde \gf_1$ and for every $T\in \hf_\infty\cap \widetilde \gf_2$. 
	Then, repeating the arguments of~\cite[Section 3]{Martini3} with minor modifications, we see that for every $\alpha_1,\alpha_2,\alpha_3> 0$ such that $\alpha_2<\frac{1}{2}$ and $\alpha_3>\alpha_1$ there is a constant $C_1>0$ such that
	\[
	\norm*{ \vect{X}_s^\gamma\Kc_{\Lc_s}(m) (1+\abs{\,\cdot\,}_s)^{\alpha_1} \prod_{j=1}^h (1+ \abs{\langle \,\cdot\,\vert  Y_j\rangle  })^{\alpha_2} }_{L^2(G_s)}\meg C_1 \norm{ m }_{B^{\alpha_3}_{\infty,\infty}(\R)}
	\]
	for every $s\in (0,\infty]$, for every $\gamma$ with length at most $1$, and for every $m\in B^{\alpha_3}_{\infty,\infty}(\R)$ with support in $[-1,1]$. Then, arguing as in the proof of~\cite[Theorem 3.11]{Martini3}, we see that for every $\alpha>\frac{Q_\infty-h}{2}$ there are $\eps>0$, $p_0>1$, and a constant $C_2>0$ such that
	\[
	\norm*{ \vect{X}_s^\gamma\Kc_{\Lc_s}(m) (1+\abs{\,\cdot\,}_s)^{\eps} }_{L^p(G_s)}\meg C_2 \norm{ m}_{B^{\alpha}_{\infty,\infty}(\R)}
	\]
	for every $p\in [1,p_0]$, for every $s\in [1,\infty]$, for every $\gamma$ with length at most $1$, and for every $m\in B^{\alpha}_{\infty,\infty}(\R)$ with support in $[-1,1]$. 
	
	{\bf2.} Take $(Y_j)$  and $(T_j)$ as in~{\bf1}, and observe that $(Y_j)$ is the basis of an algebraic complement of $\pr_1\ifr_1$ in $\widetilde \gf_1$; in particular, $\langle (Y_j)\rangle\cap (\ifr_1\cap \widetilde \gf_1)=0 $. Since $\ifr_1\cap \widetilde \gf_1=\ifr_0\cap \widetilde \gf_1$ by definition,  we may assume that $Y_j\in \hf_0$ for every $j=1,\dots,h$. 
	Now, define $h'\coloneqq \dim[(\langle (T_j)\rangle+\ifr_0)/\ifr_0 ]$, so that $h'\Meg h-Q_\infty+Q_0$; then, we may assume that $T_1,\dots, T_{h'}$ belong to $\hf_0$, so that $P_{0,0} T_j\in \langle (T_{j'})_{j'=1,\dots,h'}\rangle$ for every $j=h'+1,\dots, h$. Since the $P_{0,1} T_j$, $j=1,\dots,h$, are linearly independent, and since $\pr_1(P_{0,1} T_j)=(P_{0,1}-P_{0,0})T_j$, we see that the $\pr_1(P_{0,1} T_j)$, for $j=h'+1,\dots,h$, are linearly independent. More precisely, se wee that the $Y_j$, $j=1,\dots,h$, and the $\pr_1(P_{0,1}T_{j'})$, $j'=h'+1,\dots,h$, are linearly independent.
	
	Therefore, there is a constant $C_1>0$ such that
	\[
	\abs*{\langle  T \vert [X,\,\cdot\,\,]_1\rangle }_{\hf_0^*}\Meg C_1 \sum_{j=1}^h \abs{ \langle X\vert  Y_j\rangle \langle T\vert P_{0,1} T_j\rangle } 
	\]
	for every $X\in \hf_0\cap \widetilde \gf_1$ and for every $T\in \hf_0\cap \widetilde \gf_2$.  Observe that the dilations are self-adjoint with respect to the chosen scalar product on $\hf_0$, so that
	\[
	\abs*{\langle  T \vert [X,\,\cdot\,\,]_s\rangle }_{\hf_0^*}\Meg C_1 \sum_{j=1}^h \abs{ \langle X\vert  Y_j\rangle \langle T\vert P_{0,s} T_j\rangle } 
	\]
	for every $s\in [0,\infty)$, for every $X\in \hf_0\cap \widetilde \gf_1$ and for every $T\in \hf_0\cap \widetilde \gf_2$.  
	In particular, for $s=0$ we infer that $G_0$ is  $h'$-capacious. 
	Then, repeating the arguments of~\cite[Section 3]{Martini3} with minor modifications, we see that for every $\alpha_1,\alpha_2,\alpha_3> 0$ such that $\alpha_2<\frac{1}{2}$ and $\alpha_3>\alpha_1$ there is a constant $C_2>0$ such that
	\[
	\norm*{ \vect{X}_s^\gamma\Kc_{\Lc_s}(m) (1+\abs{\,\cdot\,}_s)^{\alpha_1} \prod_{j=1}^h (1+ \abs{P_{0,s} T_j} \abs{\langle \,\cdot\,\vert  Y_j\rangle  })^{\alpha_2} }_{L^2(G_s)}\meg C_2 \norm{ m }_{B^{\alpha_3}_{\infty,\infty}(\R)}
	\]
	for every $s\in [0,\infty)$, for every $\gamma$ with length at most $1$, and for every $m\in B^{\alpha_3}_{\infty,\infty}(\R)$ with support in $[-1,1]$. 
	Therefore, we need to prove that
	\[
	\sup_{s\in [0,1]}\int_{\hf_0} (1+\abs{X}_s^{\alpha_1'}+s^{\alpha_1-\alpha'_1}\abs{X}_s^{\alpha_1})^{-1} \prod_{j=1}^h   (1+ \abs{P_{0,s} T_j} \abs{\langle X\vert  Y_j\rangle  })^{-\alpha_2}\,\dd X<\infty
	\]
	whenever  $0<\alpha_2<1$, $\alpha'_1+h'\alpha_2>Q_0$, and $\alpha_1+h \alpha_2>Q_\infty$. Notice that it will suffice to prove the preceding assertion when $\alpha_2$ is sufficiently close to $1$, so that we shall also assume that $\alpha_1'>Q_0-h'\alpha_2+(1-\alpha_2)(h-h')$.

	Notice that the preceding arguments imply that there are a homogeneous basis $(Z_j)_{j\in J_0}$ of $\hf_0$, a partition $(J_{0,1}, J_{0,2}, J_{0,3})$ of $J_0$, and two mappings $\kappa,\kappa'\colon \Set{1,\dots,h}\to J_0$ such that the following hold:
	\begin{itemize}
		\item $(Z_j)_{j\in J_{0,1}}$ is a basis of $\widetilde \gf_2 \cap \hf_0$ and $Z_{\kappa(j)}=T_j$ for every $j=1,\dots,h'$; 
		\item $(Z_j)_{j\in J_{0,2}}$ is the basis of $\pr_1(P_{0,1}(V))$, where $V$ is an algebraic complement of $\widetilde \gf_2 \cap (\hf_0+\ifr_\infty)=(\widetilde \gf_2\cap \hf_0)\oplus (\widetilde \gf_2\cap \ifr_\infty)$ in $\widetilde \gf_2$ and $Z_{\kappa(j)}=\pr_1(P_{0,1}(T_j))$ for every $j=h'+1,\dots,h$;
		\item $(Z_j)_{j\in J_{0,3}}$ is the basis of an algebraic complement $\pr_1(P_{0,1}(V))+(\widetilde \gf_1\cap \ifr_0)$ in $\widetilde \gf_1$ and $Z_{\kappa'(j)}=Y_j$ for every$j=1,\dots,h$. 
	\end{itemize}
	Now, take $j\in \Set{h'+1,\dots,h}$ and observe that $\langle \pr_1(P_{0,1}T_j)\vert  \pr_2(P_{0,1}T_j)\rangle=0$, so that $\abs{P_{0,s} T_j}\Meg s\abs{Z_{\gamma(j)}}$ for every $s\in [0,\infty)$. 
	In addition, using~{\bf6} of Lemma~\ref{lem:7}, it is not hard to prove that there is $C_3>0$ such that
	\[
	\abs{X}_s \Meg C_3 \begin{cases}
	\abs{\langle X\vert Z_j\rangle}^{\sfrac{1}{2}} & \text{for every $j\in J_{0,1}$}\\
	 \min(\abs{\langle X\vert Z_j\rangle}, \abs{s^{-1}\langle X\vert Z_j\rangle}^{\sfrac{1}{2}}   )   & \text{for every $j\in J_{0,2}$}\\
	 \abs{\langle X\vert Z_j\rangle} & \text{for every $j\in J_{0,3}$}
	\end{cases}
	\]
	for every $X\in \hf_0$. Denote by $p_{s,j}(X)$ the right-hand side of the preceding inequality.

	Now, observe that our assumptions on $\alpha_1$ and $\alpha_2$ show that we may find $\beta_j>0$ and $\beta'_j$ for every $j\in J_0$ such that the following hold: $\alpha_1=\sum_{j\in J_0}\beta_j$ and $\alpha'_1=\sum_{j\in J_0} \beta'_j$; $\beta_j=\beta'_j>2$ for every $j\in J_{0,1}\setminus \kappa(\Set{1,\dots,h'})$;  $\beta_j>2$ and $\beta'_j>1$ for every $j\in J_{0,2}\setminus \kappa(\Set{h'+1,\dots,h})$;  $\beta_j=\beta'_j>2-\alpha_2$  for $j\in \kappa(\Set{1,\dots,h})$; $\beta_j=\beta'_j>1$ for every $j\in J_{0,3}$. Therefore, it will suffice to prove that the following  least upper bound
	\[
	\sup_{s\in [0,1]} \int_{\hf_0} \prod_{j\in J} (1+p_{s,j}(X)^{\beta'_j}+s^{\beta_j-\beta'_j}p_{s,j}(X)^{\beta_j})^{-1} \prod_{j=1}^{h'} \left(1+\abs*{\langle X\vert Z_{\kappa(j)} \rangle}\right)^{-\alpha_2} \prod_{j=h'+1}^h \left(1+s\abs*{\langle X\vert Z_{\kappa(j)} \rangle}\right)^{-\alpha_2}\,\dd X
	\]
	is finite.
	Now, use Tonelli's theorem to integrate separately each coordinate with respect to the basis $(Z_j)$. We shall prove that the integrals of the factors corresponding to $Z_{\kappa(j)}$ for $j=h'+1,\dots,h$ are uniformly bounded for $s\in [0,1]$; the other factors are easier and left to the reader.
	Then, we have to prove that
	\[
	\sup_{s\in [0,1]}\int_0^\infty (1+\min(x, \sqrt{x/s}))^{-\beta}(1+s x)^{\alpha_2} \,\dd x<\infty,
	\]  
	where $\beta>2-\alpha_2(>1)$. Now, on the one hand,
	\[
	\begin{split}
	\int_0^{1/s} (1+\min(x, \sqrt{x/s}))^{-\beta}(1+s x)^{-\alpha_2} \,\dd x\meg \int_1^{1/s} (1+x)^{-\beta} \,\dd x= \frac{1-(1+1/s)^{1-\beta}  }{\beta-1},
	\end{split}
	\]
	which is uniformly bounded for $s\in [0,1]$ since $\beta>1$. On the other hand,
	\[
	\begin{split}
	\int_{1/s}^{+\infty} (1+\min(x, \sqrt{x/s}))^{-\beta}(1+s x)^{-\alpha_2} \,\dd x\meg s^{\frac{\beta}{2}-\alpha_2}\int_{1/s}^{+\infty}   x^{-\frac{\beta}{2}-\alpha_2} \,\dd x= s^{\frac{\beta}{2}-\alpha_2} \frac{s^{\frac{\beta}{2}+\alpha_2-1}  }{\frac{\beta}{2}+\alpha_2-1},
	\end{split}
	\]	
	which is uniformly bounded for $s\in [0,1]$ since $\frac{\beta}{2}+\alpha_2>1+\frac{\alpha_2}{2}>1$ and $\beta>1$.
	The proof is then completed as that of Theorem~\ref{teo:1}.
\end{proof}

\section{Quasi-Homogeneous Operators}\label{sec:5}

We shall now investigate further the properties of the Plancherel measures $\beta_{\Lc_s}$ in some specific situations: following~\cite{Sikora}, we shall prove that, when $\Lc_s$ is `quasi-homogeneous' in a suitable sense, then  $\beta_{\Lc_s}$ has a density of class $C^\infty$ with respect to $\nu_{\R_+}$, with complete and almost explicit asymptotic expansions at $0$ and at $\infty$.

In addition to the assumptions of Sections~\ref{sec:1} and~\ref{sec:3}, we assume now that there is a finite family $(\widetilde \Lc_{\ell})_{\ell\in L}$ of self-adjoint, positive, homogeneous, left-invariant differential operators on $\widetilde G$ with the same degree $\delta$ such that $\widetilde \Lc= \sum_{\ell\in L} \widetilde \Lc_\ell$.
We also assume that $G_1$ is endowed with the structure of a homogeneous group of homogeneous dimension $Q$, and that $\dd \pi_1(\widetilde \Lc_{\ell})$ is homogeneous of degree $\delta_{\ell}$ for every $\ell \in L$.

Before proceeding further, let us make an example.

\begin{ex}\label{ex:1}
	Let $(X'_\ell)_{\ell\in L}$ be a (finite) generating family of \emph{homogeneous} elements of the Lie algebra of $G_1$, and define $\Lc_1=\sum_{\ell \in L} (i X'_\ell)^{\alpha_\ell}$, where $\alpha_\ell\in 2 \N^*$ for every $\ell\in L$. 
	In addition, let $\widetilde G$ be the free nilpotent group with $L$ generators and the same step as $G$; denote by $(\widetilde X'_\ell)_{\ell\in L}$ the generators of its Lie algebra.  
	We endow $\widetilde G$ with the unique gradation for which $\widetilde X'_\ell$ is homogeneous of degree $ \prod_{\ell'\neq \ell} \alpha_{\ell'} $ for every $\ell\in L$.
	Let $\pi_1 \colon \widetilde G\to G_1$ be the unique homomorphism of Lie groups such that that $\dd\pi_1(\widetilde X'_\ell)=X'_\ell$ for every $\ell\in L$. 
	In this context, we may define $\widetilde \Lc_\ell\coloneqq(i \widetilde X'_\ell)^{\alpha_\ell}$, $\delta\coloneqq \prod_{\ell\in L} \alpha_\ell$, and $\delta_\ell\coloneqq \dd'_\ell\alpha_\ell$, where $\dd'_\ell$ is the degree of $X'_\ell$, for every $\ell\in L$.
\end{ex}

Now, for every $\theta\in (0,\pi]$ define $\Sigma_\theta\coloneqq \Set{e^{x+i y}\colon x\in \R, y\in ]-\theta,\theta[}$, and for every $a\in \C^{ L}$ define
\[
\widetilde\Lc_{a}\coloneqq \sum_{\ell\in L} a_{\ell} \widetilde\Lc_{\ell};
\]
the reader may easily verify that $\widetilde \Lc_{a}+\widetilde \Lc_{a}^*=\widetilde \Lc_{\Re a}$ is a positive Rockland operator for every $a\in \Sigma_{\sfrac{\pi}{2}}^{ L}$.
We define  $\Lc_{s,a}\coloneqq \dd\pi_s(\widetilde \Lc_a)$ for every $a\in \Sigma_{\sfrac{\pi}{2}}^L$ and for every $s\in [0,\infty]$; observe that $\Lc_{s,a}$ is weighted subcoercive, so that we may denote by $(h_{s,a,t})_{t>0}$ its heat kernel. 
In addition, we define  $t\cdot a\coloneqq ( t^{\delta_{\ell}}a_{\ell} )_{\ell}$ for every $a\in \C^{L}$ and for every $t\in \C\setminus \R_-$; we still denote by $r a$ the multiplication of $a$ by the scalar $r$ for every $a\in \C^{ L}$ and for every $r\in \C$.

\begin{prop}\label{prop:10}\label{prop:21:11}
	 Denote by  $\Omega$ the set of $(t,a)\in\C\times \C^L$ such that $t a \in \Sigma_{\sfrac{\pi}{2}}^{L}$, and observe that $h_{1,t,a}$ is defined for every $(t,a)\in \Omega$.
	In addition, the following hold:
	\begin{itemize}
		\item the mapping $\Omega\ni (t,a)\mapsto h_{1,t,a}\in C^\infty(G_1)$ is holomorphic;
		
		\item $h_{1,t, r a}=h_{1, r t,a}$ for every $a\in \C^L$ and for every $r,t\in \C$ such that $(t, r a), (r t, a)\in \Omega $;
		
		\item $h_{1,t,r \cdot a}(e)= r^{-Q} h_{1,t,a}(e)$ for every  $a\in \C^L$ and for every $r,t\in \C$ such that $(t, r\cdot a), (t, a)\in \Omega $.
	\end{itemize}
\end{prop}

\begin{proof}
	Let us prove that, for every $p\in \N$ and for every $(t,a)\in \Omega$, $\dom( \Lc_{1,t a}^p)$ is the space $W^p$ of $f\in L^2( G_1)$ such that $\vect{X}^\gamma_1 f\in L^2( G_1)$ for every $\gamma$ such that $\dd_\gamma\meg \delta p$, endowed with the topology induced by the hilbertian norm $f\mapsto\left(  \sum_{ \dd_\gamma\meg \delta p } \norm{ \vect{X}^\gamma_1 f }_2^2\right) ^{\sfrac{1}{2}}$.  
	On the one hand, arguing as in the proof of Corollary~\ref{cor:4}, we see that $\vect{ X}^\gamma_1 (I+ \Lc_{1,t a}^p)^{-1} $ induces a bounded operator on $L^2(G_1)$ for every such $\gamma$, so that $\dom(\Lc_{1,t a}^p)$ embeds continuously into $W^p$.
	On the other hand, it is easily seen that $C^\infty_c (G_1)$, which is contained (and dense) in $\dom(\Lc_{1,t a}^p)$, is contained and dense in $W^p$, whence the asserted equality.
	
	Now, it is clear that, if $f\in W^1$, then the mapping $\Omega\ni (t,a)\mapsto \Lc_{1,t a} f\in L^2(G_1) $ is holomorphic, so that $(\Lc_{1,t a})_{(t,a)\in \Omega}$ is an analytic family of type $(A)$ in the sense of~\cite{Kato} (more precisely, the restriction of $(\Lc_{1,t a})_{(t,a)\in \Omega}$ to every complex line is an analytic family of type $(A)$). 
	In addition, $\Lc_{1,t a}$ is weighted subcoercive thanks to the preceding remarks, so that it is the generator of a holomorphic semi-group by~\cite[Theorem 8.2]{ElstRobinson}. 
	Therefore,~\cite[Theorem and 2.6 of Chapter 9]{Kato} implies  that the mapping $\Omega\ni (t,a)\mapsto e^{-\Lc_{1,t a}}\in \Lc(L^2( G_1))$ is holomorphic.\footnote{ First apply~\cite[Theorem and 2.6 of Chapter 9]{Kato} to the intersection of every complex line with $\Omega$, and then recall that a mapping from $\Omega$ into the Banach space $\Lc(L^2(G_1))$ is holomorphic if and only if it is holomorphic on every line.  }
	Therefore, taking the derivatives in $t$ we see that, for every $p\in \N$, the mapping
	\[
	\Omega\ni (t,a)\mapsto \Lc_{1,t a}^p e^{- \Lc_{1,t a}}\in \Lc(L^2( G_1))
	\]
	is holomorphic, so that the mapping
	\[
	\Omega\ni (t,a)\mapsto e^{- \Lc_{1,t a}}\in \Lc(L^2( G_1); W^p)
	\]
	is holomorphic. By the arbitrariness of $p$, this implies that the mapping
	\[
	\Omega\ni (t,a)\mapsto e^{- \Lc_{1,t a}}\in \Lc(L^2(G_1); W^\infty)
	\]
	is holomorphic, where $W^\infty$ is the intersection of the $W^p$, endowed with the corresponding topology. Since $\Lc_{1,t a}^*=\Lc_{\overline {t a}}$, and since $\Omega$ is conjugate-symmetric, by (sesquilinear) transposition we see that the mapping 
	\[
	\Omega\ni (t,a)\mapsto e^{-\Lc_{1,t a}}\in \Lc(W^{-\infty};L^2( G_1))
	\]
	is holomorphic, where $W^{-\infty}$ is the strong dual of $W^\infty$.\footnote{ In principle we should endow $ \Lc(W^{-\infty};L^2( G_1))$ with the topology of uniform convergence on the equicontinuous subsets of $W^{-\infty}$, instead of the topology of bounded convergence. However, $W^\infty$ is a reflexive Fréchet space since it is isomorphic to a closed subspace of the reflexive Fréchet space $L^2(G_1)^{\N^{\dim G_1}}$ (cf.~\cite[Propositions 14 and 15 of Chatper IV, § 1, No.\ 5 and Corollary to Theorem 1 of Chapter IV, § 2, No.\ 2]{BourbakiTVS}), so that $W^{-\infty}$ is bornological by~\cite[Proposition 4 of Chapter IV, § 3, No.\ 4]{BourbakiTVS}; therefore, a subset of $W^{-\infty}$ is bounded if and only if it is equicontinuous on $W^\infty$ by~\cite[Propositions 9 and 10 of Chapter III, § 3, No.\ 7]{BourbakiTVS}.  } Finally, arguing again as above we see that the mapping
	\[
	\Omega\ni (t,a)\mapsto e^{-\Lc_{1, t  a}}\in \Lc(W^{-\infty};W^\infty)
	\]
	is holomorphic. 	Now, the Sobolev embeddings easily show that the inclusion $W^\infty\subseteq C^\infty(G_1)$ is continuous, so that the canonical mapping $ \Lc(W^{-\infty};W^\infty)\to\Lc(\Ec'(G_1);C^\infty(G_1))$ is continuous; now, observe that $\Lc(\Ec'(G_1);C^\infty(G_1)) $ is canonically isomorphic to $C^\infty(G_1\times G_1)$ by the Schwartz's kernel theorem (cf.~\cite[Proposition 50.5]{Treves}).
	Therefore, the mapping
	\[
	\Omega\ni (t,a)\mapsto h_{1,t,a}\in C^\infty(G_1)
	\]
	is holomorphic.
	
	The second assertion is trivial, while, for what concerns the third one, just observe that $(\rho^{G_1}_r)_* \Lc_{1,t a}= \Lc_{r\cdot (t a),1}$ for every $(t,a)\in \Omega$ and for every $r>0$, where $\rho^{G_1}_r$ denotes the dilation by $r$ in $G_1$ (not to be confused with the mapping $r\,\cdot\,\colon G_1\to G_{r^{-1}}$ of the preceding sections); the general assertion follows by holomorphy.
\end{proof}

\begin{cor}\label{cor:21:4}
	Take $a\in \R_+^{L}$. Then, there is $\eps>0$ such that the mapping $t \mapsto h_{1,t,a}(e)$ extends to a holomorphic mapping $H_a\colon \Sigma_{\sfrac{\pi}{2}+\eps}\to \C$. In addition, for every $k\in \N$ there is a constant $C_k>0$ such that, for every $t\in \R^*$,
	\[
	\abs*{\frac{\dd^k}{\dd t^k} H_a(i t)}\meg C_k\min\left( \abs{t}^{-\frac{Q_0}{\delta}-k},  \abs{t}^{-\frac{Q_\infty}{\delta}-k}  \right).
	\]
\end{cor}

The proof is similar to that of~\cite[Lemma 4]{Sikora} and is omitted.

\begin{teo}\label{cor:5}
	Take $a\in \R_+^L$. Then, $\beta_{\Lc_{1,a}}$ has a  density $f_{a}$ of class $C^\infty$  with respect to $\nu_{R_+}$. In addition, there are two constants $C_0,C_\infty>0$ such that, for every $k\in \N$,
	\[
	f^{(k)}_{a}(\lambda)\sim C_{0}\left( \frac{Q_0}{\delta}\right)_k \lambda^{\frac{Q_0}{\delta}-k}
	\]
	as $\lambda\to 0^+$, while
	\[
	f^{(k)}_{a}(\lambda)\sim C_{\infty}\left( \frac{Q_\infty}{\delta}\right)_k \lambda^{\frac{Q_\infty}{\delta}-k}
	\]
	as $\lambda\to +\infty$, where $x_k\coloneqq x(x-1)\cdots (x-k+1) $ for every $x\in \R$.
\end{teo}

In particular, in this situation we may apply the second part of Theorem~\ref{teo:1}, thus extending~\cite[Theorem 2]{Sikora}, which corresponds to the case  $\alpha_\ell=2$ for every $\ell\in L$ in the situation of Example~\ref{ex:1}.

\begin{proof}
	Observe that, with the notation of Corollary~\ref{cor:21:4}, 
	\[
	H_{a}(t)=\int_{[0,\infty)} e^{-t \lambda}\,\dd \beta_{\Lc_{1,a}}(\lambda)
	\]
	for every $t>0$, so that
	\[
	\Fc(e^{-\eps\,\cdot\,}\beta_{\Lc_{1,a}})(t)= H_a(\eps+i t)
	\]
	for every $\eps>0$ and for every $t\in \R$. Passing to the limit for $\eps\to 0^+$, we see that the restriction of $\Fc(\beta_{\Lc_{1,a}})$ to $\R\setminus \Set{0}$ has a density of class $C^\infty$, and that
	\[
	\Fc(\beta_{\Lc_{1,a}})(t)= H_a(i t)
	\] 
	for every $t\neq 0$.
	Then, Corollary~\ref{cor:21:4} and~\cite[Proposition 1]{Sikora} show that $\beta_{\Lc_{1,a}}$ has a  density $f_{a}$ of class $C^\infty$ with respect to $\nu_{\R_+}$ such that for every $k\in \N$ there is a constant $C_k>0$ such that
	\[
	\abs{f_a^{(k)}(\lambda)}\meg C_k \min\left(\lambda^{\frac{Q_0}{\delta}-k}, \lambda^{\frac{Q_\infty}{\delta}-k}  \right)
	\]
	for every $\xi>0$. Now, Proposition~\ref{prop:2} shows that $s^{Q_0}f_a(s^{-\delta}\,\cdot\,)\nu_{\R_+}$ converges vaguely to the measure $\beta_{\Lc_{0,a}}$ as $s\to 0^+$. 
	In addition, by homogeneity it is easily seen that there is a constant $C_0>0$ such that $\beta_{\Lc_{0,a}}=C_0 (\,\cdot\,)^{\frac{Q_0}{\delta}} \nu_{\R_+}$. 
	Finally, the preceding estimates show that the  $s^{Q_0}f_a(s^{-\delta}\,\cdot\,) $ stay bounded in $C^\infty(\R_+)$, so that they converge to $C_0 \lambda^{\frac{Q_0}{\delta}}$ in $C^\infty(\R_+)$. The first assertion follows; the second one is proved similarly.
\end{proof}

\section{Appendix: Technical Lemmas}

In this section we consider a homogeneous vector space $V$, endowed with a homogeneous basis $\partial$ of translation-invariant vector fields and a homogeneous norm $\abs{\,\cdot\,}$; for every $\gamma$, we denote by $\dd_\gamma$ the degree of $\partial^\gamma$. We fix $\eps>0$, $\eta\in \R$ and $\eta'\in \R^\N$.

Recall that $\Sc(V)$ denotes the Schwartz space, $\Sc'(V)$ the space of tempered distributions, $\Dc'(V)$ the space of distributions,  and $\Ec'(V)$ the space of distributions with compact support on $V$. We denote by $\Mc^1$ the space of bounded (Radon) measures.

\begin{deff}\label{def:1}
	Define $\Hc_{\eps,\eta}(V)$ as the space of $H\in C(\R_+\times V)$ such that the the set of $t^{Q\eps+\eta} H(t,t^\eps\,\cdot\,)$, as $t$ runs through $\R_+$, is bounded in $\Sc(V)$.
	
	We endow $\Hc_{\eps,\eta}(V)$ with the topology induced by the norms
	\[
	H\mapsto \sup_{t>0} \sup_{x\in V} (1+\abs{x})^h \sum_{\dd_\gamma\meg h} t^{Q\eps+\eta+\eps \dd_\gamma} \abs{\partial_2^\gamma H(t,t^\eps\cdot x)},
	\]
	for $h\in \N$, so that $\Hc_{\eps,\eta}(V)$ becomes a metrizable locally convex space (actually, a Fréchet space).
\end{deff}

\begin{lem}
	Take $H\in \Hc_{\eps,\eta}(V)$. Then, the mapping $t\mapsto H(t,\,\cdot\,)$ belongs to $C(\R_+; \Sc(V))$. In particular,  the function $\partial_2^\gamma H$ is continuous for every $\gamma$. 
\end{lem}

\begin{proof}
	Take $t_0>0$, and observe that the set of $H(t,\,\cdot\,)$, as $t$ runs through $[\frac{t_0}{2},2 t_0]$, is bounded, hence relatively compact, in $\Sc(V)$. Therefore, $H(t,\,\cdot\,)$ has at least one cluster point in $\Sc(V)$ as $t\to t_0$. On the other hand, each cluster point of $H(t,\,\cdot\,)$  in $\Sc(V)$ as $t\to t_0$ is also a cluster point of $H(t,\,\cdot\,)$ in $\Ec^0(V)$ as $t\to t_0$, so that it must equal $H(t_0,\,\cdot\,)$ by the continuity of $H$. The assertion follows.
\end{proof}

\begin{lem}\label{lem:1}
	Take $H\in \Hc_{\eps,\eta'_0}(V)$ such that $\partial_1^k H\in \Hc_{\eps,\eta'_k}(V)$ for every $k\in \N$. Then, the following conditions are equivalent:
	\begin{enumerate}
		\item the mapping $\R_+\ni t \mapsto \partial^k_1 H(t,\,\cdot\,)\in \Dc'(V)$ extends by continuity to $[0,\infty)$, for every $k\in \N$;
		
		\item the mapping $\R_+\ni t \mapsto \partial^k_1 H(t,\,\cdot\,)\in \Ec'(V)+\Sc(V)$ extends by continuity to $[0,\infty)$, for every $k\in \N$;
		
		\item there is $\tau\in C^\infty_c(V)$ such that  $\tau$ equals $1$ on a neighbourhood of $0$ and such that the mapping $\R_+\ni t \mapsto \tau \partial^k_1 H(t,\,\cdot\,)\in \Ec'(V)$ extends by continuity to $[0,\infty)$, for every $k\in \N$;
		
		\item for every $\tau\in C^\infty_c(V)$ such that $\tau$ equals $1$ on a neighbourhood of $0$,  the mapping $\R_+\ni t \mapsto \tau \partial^k_1 H(t,\,\cdot\,)\in \Ec'(V)$ extends by continuity to $[0,\infty)$, for every $k\in \N$.
	\end{enumerate}
\end{lem}

\begin{proof}
	It is clear that~{\bf1} implies~{\bf4}, that~{\bf4} implies~{\bf3} and that~{\bf2} implies~{\bf1}. Let us then prove that~{\bf3} implies~{\bf2}. Then, take $\tau$ as in~{\bf3}, and observe that it will suffice to prove that the mapping $\R_+\ni t \mapsto (1-\tau) \partial^k_1 H(t,\,\cdot\,)\in \Sc(V)$ extends by continuity to $[0,\infty)$, for every $k\in \N$.  
	Then, take $h,k\in \N$, and observe that, since $\partial_1^k H\in \Hc_{\eps,\eta'_k}(V)$, for every $N\in \N$ there is a constant $C_N>0$ such that
	\[
	\sup_{\dd_\gamma\meg h} \abs{\partial_1^k \partial_2^\gamma H(t,x)}\meg \frac{C_N}{(1+\abs{t^{-\eps}\cdot x})^{N} t^{Q\eps+\eta'_k+\eps \dd_\gamma}  }  
	\]
	for every $(t,x)\in \R_+\times V$. Then, for every $(t,x)\in \R_+\times V$
	\[
	\sum_{\dd_\gamma\meg h} \abs{\partial_1^k \partial_2^\gamma [((1-\tau)\circ \pr_2) H](t,x)}\meg \chi_{\Supp{1-\tau}}(x) \sum_{\dd_\gamma \meg h} \sum_{\gamma'+\gamma''=\gamma}  \frac{\gamma!}{\gamma'! \gamma''!}\norm{\partial^{\gamma'} \tau  }_\infty \frac{C_N}{\abs{x}^{N} t^{Q\eps+\eta'_k+\eps \dd_{\gamma''}-\eps N}  }  ;
	\]
	Therefore, there is $C'_N>0$ such that, for every $(t,x)\in (0,1]\times V$,
	\[
	\sum_{\dd_\gamma\meg h} \abs{\partial_1^k \partial_2^\gamma [((\chi_V-\tau)\circ \pr_2) H](t,x)}\meg  C'_N \frac{t^{ \eps N-Q\eps-\eta'_k-\eps h  }}{(1+\abs{x})^N},
	\]
	which tends to $0$ as $t\to 0^+$ provided that $N>Q+\frac{\eta'_k}{\eps}+ h$. The assertion follows by the arbitrariness of $k$, $h$, and $N$.
\end{proof}

\begin{deff}
We define $\widetilde \Hc_{\eps,\eta'}(V)$ as the space of $H$  satisfying the equivalent conditions of Lemma~\ref{lem:1}.
We endow $\widetilde \Hc_{\eps,\eta'}(V)$ with the topology induced by the norms of $\Hc_{\eps,\eta'_k}(V)$ applied to $\partial_1^k H$ ($k\in \N$), and by the semi-norms
\[
H\mapsto \sup_{t\in (0,1]}\sup_{\phi \in B } \abs*{\langle \tau \partial_1^k H(t,\,\cdot\,),\phi\rangle}
\] 
as $k$ runs through $\N$, $\tau$ is an element of $C^\infty_c(V)$ which equals $1$ on a neighbourhood of $0$, and $B$ runs through the bounded subsets of $C^\infty(V)$.
\end{deff}

\begin{lem}\label{lem:21:5}
	Take $\tau\in C^\infty_c(V)$ such that $\tau-1$ vanishes of order $\infty$ at $0$, and fix $p\in \N$.
	In addition, let $M_p$ be the set of (Radon) measures $\mi $ on $\R_+$ such that $\int_0^1 t^{p} \dd \abs{\mi}(t)<+\infty$, and such that $\int_1^{+\infty} t^{k}\dd \abs{\mi}(t)<+\infty $ for every $k\in \N$. Endow $M_p$ with the corresponding topology. 
	
	Then, for every $\mi\in M_p$ and for every $H\in \Hc_{\eps,\eta}(V)$, the mapping $t \mapsto  (1-\tau) H(t,\,\cdot\,) \in \Sc(V)$ is $\mi$-integrable.
	In addition, the bilinear mapping
	\[
	M_p\times \Hc_{\eps,\eta}(V)\ni \mi \mapsto \int_{0}^{+\infty} (1-\tau)   H(t,\,\cdot\,)\,\dd \mi(t)\in \Sc(V)
	\]
	is continuous.
\end{lem}

\begin{proof}
	Indeed, take $\mi \in M_p$ and $H\in \Hc_{\eps,\eta}(V)$. Observe that, for every $N\in\N$, there is a continuous semi-norm $\rho_N$ on  $\Hc_{\eps,\eta}(V)$ such that
	\[
	\abs{\partial_2^\gamma H(t,x)}\meg \frac{\norm{H}_{\rho_N}}{ t^{\eps (Q+\dd_\gamma)+\eta} (1+\abs{t^{-\eps}\cdot x})^N } 
	\]
	for every $x\in V$ and for every $\gamma$ such that $\dd_\gamma \meg N$; fix $k\in \N$ and $\gamma$.
	Then, apply Leibniz's rule and, for $t\meg 1$, estimate the derivatives of $1-\tau$ with $\abs{\,\cdot\,}^{N-k}$ for some fixed $N\Meg \max(\dd_\gamma,Q+\dd_\gamma+\frac{p+\eta}{\eps})$; we then see that there is a constant $C'>0$ such that
	\[
	\abs{x}^{k}\abs{\partial^\gamma [(1-\tau)H(t,\,\cdot\,)](x)}\meg C' t^{ \eps(N-Q-\dd_\gamma )-\eta } \norm{H}_{\rho_{N}}\meg C' t^p \norm{H}_{\rho_{N}}
	\] 
	for every $x\in V$. On the other hand, if $t\Meg 1$, then simply estimate the derivatives of $1-\tau$ with $\chi_{\Supp(1-\tau)}$; we then see that there is a constant $C''>0$ such that
	\[
	\abs{x}^{k}\abs{\partial^\gamma [(1-\tau)H(t,\,\cdot\,)](x)}\meg C'' t^{ \eps(k-Q)-\eta  } \norm{H}_{\rho_{\max(k,\dd_\gamma)}}
	\]
	for every $x\in V$.
	Therefore,
	\[
	\begin{split}
	\int_0^{+\infty} \abs{x}^{k}\abs{\partial^\gamma [(\chi_V-\tau)H(t,\,\cdot\,)](x)}\,\dd \abs{\mi}(t)&\meg C'\norm{H}_{\rho_{N}} \int_{(0,1]} t^p\,\dd \abs{\mi}(t)\\
		&\qquad + C'' \norm{H}_{\rho_{\max(k,\dd_\gamma)}} \int_{[1,+\infty)} t^{ \eps(k-Q)-\eta  } \,\dd \abs{\mi}(t).
	\end{split}
	\]
	By the arbitrariness of $k$ and $\gamma$, the assertion follows.
\end{proof}

\begin{lem}\label{lem:21:4}
	For every $\mi\in \Mc^1((0,1])$ and for every $H\in \widetilde\Hc_{\eps,\eta'}(V)$, the mapping 
	\[
	t \mapsto  t^{-k} \left(  H(t,\,\cdot\,)- \sum_{j<k} \partial_1^j H(0,\,\cdot\,)\frac{t^j}{j!} \right)\in \Ec'(V)+\Sc(V)
	\]
	is scalarly $\mi$-integrable and its integral belongs to $\Ec'(V)+\Sc(V)$.
	In addition, the bilinear mapping
	\[
	\Mc^1((0,1])\times \widetilde\Hc_{\eps,\eta'}(V)\ni (\mi,H) \mapsto \int_{0}^{1} t^{-k}\left(  H(t,\,\cdot\,)- \sum_{j<k} \partial_1^j H(0,\,\cdot\,)\frac{t^j}{j!} \right)\,\dd \mi(t)\in \Ec'(V)+\Sc(V)
	\]
	is continuous.
\end{lem}

\begin{proof}
	Take some $\tau\in C^\infty_c(V)$ which equals $1$ in a neighbourhood of $0$, and let us prove that the mapping $t \mapsto  t^{-k} \tau\left(  H(t,\,\cdot\,)- \sum_{j<k} \partial_1^j H(0,\,\cdot\,)\frac{t^j}{j!} \right)\in \Ec'(V)$ is scalarly $\mi$-integrable and that its integral belongs to $\Ec'(V)$.
	Observe that, since $\Ec'(V)$ is quasi-complete, by~\cite[Proposition 8 of Chapter VI, § 1, No.\ 2]{BourbakiInt1} it will suffice to prove that the mapping $t \mapsto t^{-k} \tau\left(H(t,\,\cdot\,) - \sum_{j<k} \partial_1^j H(0,\,\cdot\,) \frac{t^j}{j!} \right)\in \Ec'(V)$ is continuous and bounded.
	However, Taylor's formula implies that
	\[
	t^{-k} \tau\left(H(t,\,\cdot\,) - \sum_{j<k} \partial_1^j H(0,\,\cdot\,) \frac{t^j}{j!} \right)=\int_0^1 \tau \partial_1^k H(t s,\,\cdot\,) \frac{(1-s)^{k-1}}{(k-1)!}\,\dd s.
	\]
	Now, for every bounded subset $B$ of $C^\infty(V)$ there is a continuous semi-norm $\rho_B$ of $\widetilde \Hc_{\eps,\eta'}(V)$ such that 
	\[
	\abs*{\langle \tau \partial_1^k H(t,\,\cdot\,), \phi \rangle }\meg \norm{H}_{\rho_B} 
	\]
	for every $\phi \in B$ and for every $t\in (0,1]$. Hence,
	\[
	 t^{-k} \sup_{\phi\in B}\abs*{ \langle H(t,\,\cdot\,)- \sum_{j<k} \partial_1^j H(0,\,\cdot\,)\frac{t^j}{j!} ,\tau \phi\rangle  }\meg \frac{ \norm{H}_{\rho_B} }{k!},
	\]
	whence our claim (cf.~also Lemma~\ref{lem:1}).
	The assertion then follows by means of Lemma~\ref{lem:21:5}.
\end{proof}

Recall that $\Oc_C(V)$ is the set of $f\in C^\infty(V)$ such that there is $k\in \N$ such that $\partial^\alpha f(x)=O\big(\abs{x}^k\big)$ as $x\to \infty$ for every $\alpha$; $\Oc_C(V)$ can then be identified with the dual of the space $\Oc'_C(V)$ of convolutors of $\Sc(V)$, and carries the corresponding strong dual topology (cf.~\cite[pp.~244 and 245]{Schwartz3} and~\cite[Chapter II, § 4, No.\ 4]{Grothendieck}).

\begin{lem}\label{lem:21:3}\label{lem:21:6}
	For every $k\Meg 0$, for every $\mi \in \Mc^1([1,+\infty))$, and for every $H\in \Hc_{\eps,\eta}(V)$, the mapping 
	\[
	t \mapsto t^{\eps(Q+k)+\eta}\left( H(t,\,\cdot\,)-   \sum_{\dd_\gamma<k} \partial^\gamma_2 H(t,0)\frac{(\,\cdot\,)^\gamma}{\gamma!}  \right)\in  \Oc_C(V)
	\]
	is scalarly $\mi$-integrable, and its integral belongs to $\Oc_C(V)$.
	In addition, the bilinear mapping
	\[
	\Mc^1([1,+\infty))\times \Hc_{\eps,\eta}(V)\ni (\mi,H) \mapsto \int_1^{+\infty} t^{\eps(Q+k)+\eta}\left( H(t,\,\cdot\,)-   \sum_{\dd_\gamma<k} \partial^\gamma_2 H(t,0)\frac{(\,\cdot\,)^\gamma}{\gamma!}  \right)\,\dd \mi(t)\in  \Oc_C(V)
	\]
	is continuous.	
\end{lem}

\begin{proof}
	Observe that~\cite[Proposition 8 of Chapter VI, § 1, No.\ 2]{BourbakiInt1} implies that it will suffice to prove that the mapping $t \mapsto   t^{\eps(Q+k)+\eta}\left( H(t,\,\cdot\,)-   \sum_{\dd_\gamma<k} \partial^\gamma_2 H(t,0)\frac{(\,\cdot\,)^\gamma}{\gamma!}  \right)\in \Oc_C(V)$  is continuous on $[1,+\infty)$ and takes values in an equicontinuous subset of $\Oc_C(V)$ (considered as the strong dual of $\Oc_C'(V)$). Now, continuity is clear.
	In addition, fix $\gamma'$ and observe that~\cite[Theorem 1.37]{FollandStein} implies that there is a constant $C_{\gamma'}>0$ such that
	\[
	\abs*{  \partial_2^{\gamma'} H(t,x)- \sum_{\dd_\gamma<k} \partial_2^{\gamma+{\gamma'}} H(t,0)\frac{x^\gamma}{\gamma!} }\meg C_{\gamma'}\sum_{\substack{\sum_j \gamma_j\meg \left[\frac{k}{d}\right]+1\\ \dd_\gamma \Meg k}} \abs{x}^{\dd_\gamma} \sup_{ \abs{x'}\meg C_{\gamma'} \abs{x} } \abs*{ \partial^{\gamma+{\gamma'}}_2 H(t,x') },
	\]
	where $d$ is the minimum degree of the non-zero homogeneous elements of $V$, for every $x\in V$ and for every $t>0$. Therefore,  there is a continuous semi-norm $\rho_{\gamma'}$ on $\Hc_{\eps,\eta}(V)$ such that
	\[
	\abs*{ \partial^{\gamma'}_2 H(t,x)- \sum_{\dd_\gamma<k} \partial^{\gamma+{\gamma'}}_2 H(t,0)\frac{x^\gamma}{\gamma!} }\meg  t^{-\eps(Q+k+\dd_{\gamma'})-\eta} (1+\abs{x})^{D \left( \left[\frac{k}{d}\right]+1 \right) } \norm{H}_{\rho_{\gamma'}},
	\]
	where $D$ is the maximum degree of the non-zero homogeneous elements of $V$, for every $x\in V$ and for every $t\Meg 1$.  The assertion follows easily.
\end{proof}

\end{document}